\documentclass[review]{elsarticle}

\usepackage{amsmath}
\usepackage{amsfonts}
\usepackage{amssymb}
\usepackage{amsthm}
\usepackage{bm}
\usepackage{indentfirst}
\usepackage{graphicx}
\usepackage{subfigure}
\usepackage{array}
\usepackage{fullpage}

\DeclareMathAlphabet{\mathsfsl}{OT1}{cmss}{m}{sl}

\newcommand{\PreserveBackslash}[1]{\let\temp=\\#1\let\\=\temp}
\newcolumntype{C}[1]{>{\PreserveBackslash\centering}p{#1}}
\newcolumntype{R}[1]{>{\PreserveBackslash\raggedleft}p{#1}}
\newcolumntype{L}[1]{>{\PreserveBackslash\raggedright}p{#1}}

\numberwithin{equation}{section}
\newtheorem{thm}{Theorem}[section]
\newtheorem{lem}[thm]{Lemma}

\theoremstyle{definition}

\usepackage{hyperref}
\usepackage{fancyvrb}
\usepackage{fancyhdr}
\usepackage{caption}
\biboptions{sort&compress}
\usepackage{bbm}
\usepackage{multirow}
\usepackage{bm}
\usepackage{mathrsfs}
\usepackage{amsmath}
\usepackage{amsfonts}
\usepackage{amsthm}
\usepackage{color}
\usepackage{setspace}
\usepackage{float}
\usepackage{url}
\usepackage{subfigure}
\usepackage{multicol}
\usepackage{txfonts}
\usepackage{lineno,hyperref}
\usepackage{enumitem,kantlipsum}
\modulolinenumbers[5]
\newtheorem{theorem}{Theorem}

\newcommand{\horizon}{\delta}

\newcommand{\mcL}{\mathcal{L}}

\newcommand{\mcO}{\mathcal{O}}

\usepackage{xcolor}

\def\omg{{\Omega}}

\def \Fb{\mathbf{F}}

\def \Ub{\mathbf{U}}

\def \xb{\mathbf{x}}

\def \nb{\mathbf{n}}

\begin{document}

\theoremstyle{definition}%plain/definition/remark
\newtheorem{remark}{Remark}%[section]

% triple bar norm, adapted from forum thread
% http://tex.stackexchange.com/questions/54385/spacing-between-triple-vertical-lines
%\newcommand{\vertiii}[1]{{\left\vert\kern-0.2ex\left\vert\kern-0.2ex\left\vert #1
%    \right\vert\kern-0.2ex\right\vert\kern-0.2ex\right\vert}}
\newcommand{\vertiii}[1]{{\left\vert\left\vert\left\vert #1
    \right\vert\right\vert\right\vert}}

\begin{frontmatter}

\title{An Asymptotically Compatible Formulation for Local-to-Nonlocal Coupling Problems without Overlapping Regions}

\address[yy]{Department of Mathematics, Lehigh University, Bethlehem, PA 18015, USA}
\address[dk]{Department of Mechanical and Aerospace Engineering, The University of California, San Diego, CA, 92093}

\author[yy]{Huaiqian You}
\author[yy]{Yue Yu\corref{cor1}}\ead{yuy214@lehigh.edu}
\author[dk]{David Kamensky}

\begin{abstract}
In this paper we design and analyze an explicit partitioned procedure for a 2D dynamic local-to-nonlocal (LtN) coupling problem, based on a new nonlocal Robin-type transmission condition. The nonlocal subproblem is modeled by the nonlocal heat equation with a finite horizon parameter $\delta$ characterizing the range of nonlocal interactions, and the local subproblem is described by the classical heat equation. We consider a heterogeneous system where the local and nonlocal subproblems present different physical properties, and employ no overlapping region between the two subdomains. 
%To develop a robust explicit coupling procedure with the Robin-type transmission condition, 
We first propose a new generalization of classical local Neumann-type condition by converting the local flux to a correction term in the nonlocal model, and show that the proposed Neumann-type boundary formulation recovers the local case as $\mcO(\delta^2)$ in the $L^{\infty}$ norm. We then extend the nonlocal Neumann-type boundary condition to a Robin-type boundary condition, and develop a local-to-nonlocal coupling formulation with Robin-Dirichlet transmission conditions. To stabilize the explicit coupling procedure and to achieve asymptotic compatibility, the choice of the coefficient in the Robin condition is obtained via amplification factor analysis for the discretized system with coarse grids. Employing a high-order meshfree discretization method in the nonlocal solver and a linear finite element method in the local solver, the selection of optimal Robin coefficients are verified with numerical experiments on heterogeneous and complicated domains. With the developed optimal coupling strategy, we numerically demonstrate the coupling framework's asymptotic convergence to the local limit with an $\mcO(\delta)=\mcO(h)$ rate, when there is a fixed ratio between the horizon size $\delta$ and the spatial discretization size $h$.
\end{abstract}


\end{frontmatter}

\tableofcontents

%\linenumbers

\section{Introduction}

In the last decades, there has been an increasing interest in the simulation of nonlocal integro-differential equations (IDEs) such as nonlocal diffusion and peridynamics\cite{silling_2000,bazant2002nonlocal,zimmermann2005continuum,emmrich2007analysis,emmrich2007well,zhou2010mathematical, du2011mathematical,podlubny1998fractional,mainardi2010fractional,magin2006fractional,burch2011classical,du2014nonlocal,defterli2015fractional,lischke2018fractional,du2014peridynamics,antoine2005approximation,dayal2007real,sachs2013priori,bucur2016nonlocal}, since they can describe phenomena not well represented by classical Partial Differential Equations (PDEs). The nonlocal models with integral operators in space allow for the description of long-range interactions and reduce the regularity requirements on problem solutions, and therefore provide exceptional simulation fidelity for a broad spectrum of applications such as fracture mechanics, anomalous subsurface transport, phase transitions, image processing, multiscale and multiphysics systems, magnetohydrodynamics, and stochastic processes. 

However, despite the nonlocal IDEs' improved accuracy, the usability of nonlocal models could be hindered by several modeling and numerical challenges such as the unconventional prescription of nonlocal boundary conditions, the calibration of nonlocal model parameters and the expensive computational cost. Moreover, in real-world applications nonlocal effects are often concentrated only in some parts of the domain, and in the remaining parts the system can be accurately described by a PDE. Thus, local-to-nonlocal coupling strategies are required such that the resultant coupling framework can support the nonlocal model near the regions where the nonlocal interaction occurs as well as the efficient classical PDE model employed for the other parts. In recent years, many strategies have been proposed to couple local-to-nonlocal or two nonlocal models with different nonlocality\cite{seleson2013interface,azdoud2013morphing,han2012coupling,prudhomme2008computational,d2014optimal,du2017quasinonlocal,li2016quasinonlocal,d2016coupling,lubineau2012morphing,seleson2015concurrent,askari2008peridynamics,TTD19,silling2015variable,RW_Macek2007,E_Oterkus2010,A_Agwai2009,W_Liu2012,Lee2016,Han2016336,Galvanetto201641}. Just to name a few, examples include the optimal-control based coupling method \cite{d2014optimal,d2016coupling}, the overlapping partitioned procedure with Robin conditions \cite{yu2018partitioned}, the Arlequin method \cite{prudhomme2008computational,han2012coupling}, 
the Morphing approach \cite{lubineau2012morphing,azdoud2013morphing,Han2016336}, the quasi-nonlocal coupling method \cite{li2016quasinonlocal,DuLiLuTian2018}, the force-based blending method \cite{seleson2013force,seleson2015concurrent}, the splice method \cite{silling2015variable}, the varying horizon method \cite{bobaru2011adaptive,bobaru2009convergence,seleson2013interface,ren2016dual,seleson2010peridynamic,silling2015variable,TTD19,TD17trace}, the submodeling approach \cite{RW_Macek2007,A_Agwai2009,E_Oterkus2010}, and so on. %
However, most of the above local-to-nonlocal coupling approaches focus on the scenario where the local and nonlocal models are physically consistent, i.e., when the nonlocal interaction range $\delta$ shrinks, the nonlocal model converges to the local model, and there is no jump of the physical properties across the local-nonlocal interface. To the authors' best knowledge, there is very little work on dynamic local-to-nonlocal coupling approaches for heterogeneous domains where the local and nonlocal regions present dramatically different physical properties, although those approaches are required for applications with both nonlocal effects and multiscale/multiphysics dynamics. 

Therefore, we aim to develop a dynamic local-to-nonlocal coupling method based on an explicit coupling partitioned procedure with transmission conditions applied on the sharp interface, so that the method is capable of handling the physical property jumps across the interface. For concreteness, in this paper we focus on coupling the nonlocal heat equation with the classical heat equation, although the proposed technique is applicable to more general problems. The numerical approximation of this type of heterogeneous system is challenging, due to {potential numerical instabilities in coupling schemes for domain-decomposition problems} 
%\DK{[What does ``domain-decomposition nature'' mean here?]}\YY{[I mean the possible instability caused by decoupling the two subproblems as in a typical domain decomposition problem. Maybe I should change it to "due to the peculiarities of decoupling the two subproblems"?]}\DK{[How about: ``due to potential instabilities in explicitly-coupled schemes''?]}\YY{[Since the implicit schemes might also suffer from instabilities, what about "due to potential instabilities in coupling schemes of domain decomposition problems"?]}\DK{[Maybe change ``of'' to ``for''; otherwise okay.]}
and the nonlocal effects involved. Specifically, the local-to-nonlocal coupling method for heterogeneous systems presents both modeling and numerical difficulties/desired properties:
\begin{itemize}
    \item To apply the transmission condition on the nonlocal side, a nonlocal boundary condition on the sharp interface is required. However, in general nonlocal boundary conditions must be defined on a layer surrounding the domain. Therefore, new definitions of the nonlocal boundary conditions are required when only the surface data are available at the sharp local-to-nonlocal interface.
    \item A key feature in the discretization of nonlocal models has been the concept of \textit{asymptotic compatibility} \cite{tian2014asymptotically}, meaning that the nonlocal discretization has to recover a corresponding local model as both the nonlocal interaction range $\delta$ and the characteristic discretization lengthscale are reduced at the same rate. To ensure that the local-to-nonlocal coupling model recovers a well-understood classical limit, it is advocated that the developed coupling framework should also preserve asymptotic compatibility.
        \item In explicit coupling partitioned procedures, both the local and nonlocal subproblems are solved only once per time step and do not satisfy exactly the coupling transmission conditions. As a consequence, the work exchanged between the two subproblems is not perfectly balanced and this may induce instabilities in the coupling scheme. Therefore, stabilization strategies are required to develop a robust explicit coupling method with partitioned procedure.
\end{itemize}

In this paper, we address the above three difficulties with three steps. Firstly, to resolve the modeling difficulty of defining the nonlocal transmission condition we introduce a nonlocal boundary treatment that is designed to convert the local Neumann-type boundary conditions defined on sharp surfaces into nonlocal volume constraints in the nonlocal model, and rigorously prove that this nonlocal boundary value problem recovers the desired local Neumann problem with an optimal $\mcO(\delta^2)$ rate as the nonlocal interaction range $\delta\rightarrow 0$. Based on the nonlocal Neumann-type boundary condition, we further develop the nonlocal Robin-type boundary condition on a sharp surface.  Although there are several previous attempts to tackle the conversion of surface data and nonlocal volume constraints (see, e.g., \cite{tao2017nonlocal,DElia2018-P,yu2018partitioned,DuZhangZheng,cortazar2007boundary,cortazar2008approximate,d2019physically}), to the authors' best knowledge the proposed formulation has for the first time provided a Robin-type boundary condition for nonlocal problems and obtained the optimal second order asymptotic convergence to the local limit. Secondly, to ensure the asymptotic compatibility of the nonlocal solver, based on the new formulation for nonlocal boundary condition we develop an asymptotically compatible meshfree discretization method with the generalized moving least squares (GMLS) approximation framework \cite{Trask2018paper,You2018}. In the last part of the paper, we investigate a stabilization strategy for coupling local and nonlocal heat equations. In classical domain-decomposition problems, the Robin transmission condition, which is a linear combination of the Dirichlet and Neumann transmission conditions, has been proven to be very efficient in enhancing the coupling stability in explicit partitioned procedures (see, e.g., \cite{badia2008fluid,chen2011parallel,discacciati2007robin,douglas1997accelerated}). Therefore, to resolve the last difficulty we propose an explicit partitioned procedure based on the developed Robin-type boundary condition applied on the sharp local-to-nonlocal interface, improving upon the implicit partitioned procedure with overlapping regions found in the literature \cite{yu2018partitioned}. In the nonlocal subdomain, the proposed Robin-type transmission condition is applied on the interface and the nonlocal heat equation is discretized with the meshfree discretization method. In the local subdomain, classical Dirichlet transmission condition is applied, while the classical heat model is discretized with finite elements. To investigate the optimal coupling strategy for this partitioned coupling framework, we develop stability analysis in general geometries for predicting the values of the optimal Robin coefficients numerically. %We then demonstrate the performance of this coupling framework and verify the optimal coefficient with convergence and patch tests. To investigate the capability of this coupling framework on more complicated scenarios, we also test the flexibility of this method for different application problems.

The paper is organized as follows. We first present in Section \ref{sec:models} governing equations of nonlocal and local models and the discretization methods, respectively. In Section \ref{section:robin} a nonlocal Robin-type boundary condition based on sharp surface data is proposed: we firstly develop a nonlocal Neumann-type boundary condition and provide a consistency result for the resultant nonlocal boundary value problem in Section \ref{sec:neumann}, then generalize the nonlocal Neumann-type boundary condition to a nonlocal Robin-type boundary condition in Section \ref{sec:robinmath}. The consistency of the proposed Robin-type boundary condition is then numerically verified in Section \ref{sec:Robintest} where the optimal $O(\delta^2)$ convergence to the local limit is obtained. With the developed nonlocal Robin condition, the coupling procedure is detailed in Section \ref{sec:couple}. For the full partitioned algorithm presented in Section \ref{41}, in Section \ref{sec:Robincoef} we present stability analysis for the fully discretized problem and develop a numerical approach to approximate the optimal Robin coefficient. In Section \ref{42} we then demonstrate the performance of this coupling framework and verify the optimal coefficient with convergence and patch tests. To investigate the capability of this coupling framework on more complicated scenarios, we also test the flexibility of this method for problems with different domain settings. Section \ref{sec:conclusion} summarizes our findings and discusses future research.

\section{Preliminaries on Local and Nonlocal Models}\label{sec:models}

\begin{figure}[!htb]\centering
 \subfigure{\includegraphics[width=0.35\textwidth]{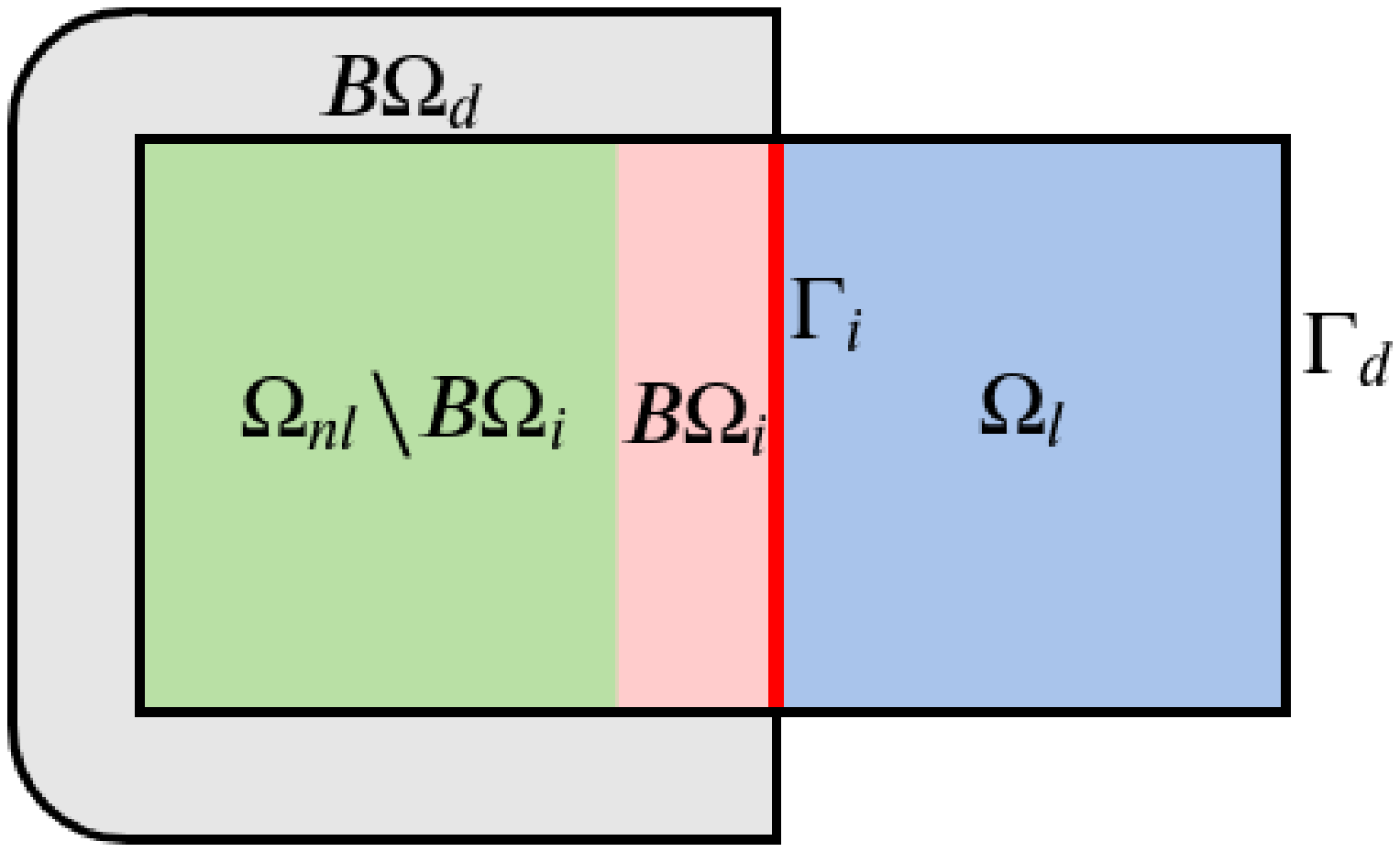}}$\qquad\qquad$
  \subfigure{\includegraphics[width=0.22\textwidth]{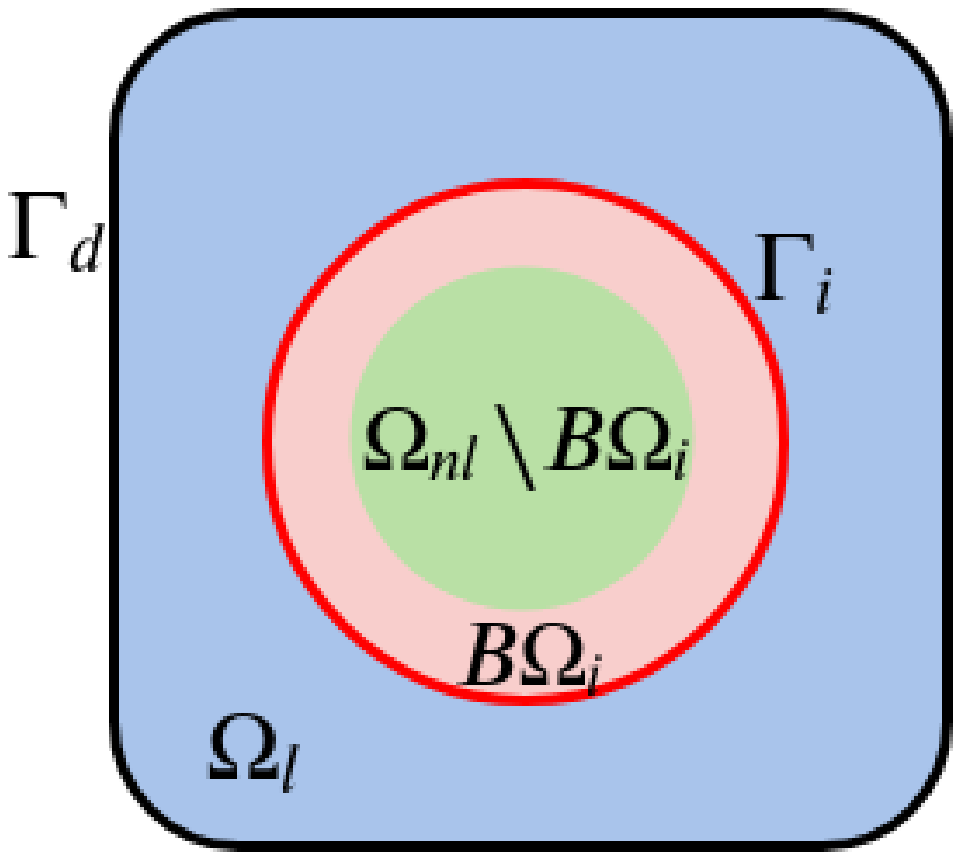}}
 %\subfigure{\includegraphics[scale=.4]{convex.pdf}
 %\put(-55,70){$\mathbf{x}$}\put(-20,70){$\overline{\mathbf{x}}$}
 %\put(-20,10){$\tau(\overline{\mathbf{x}})$}
 %\put(-15,90){$T_\delta$}
 %\put(-33,115){$A_\delta$}\put(-35,35){$A_\delta$}
 %\put(-50,135){$\partial\Omega$}
 %\put(-100,90){$B(\mathbf{x},\delta)$}}
 \caption{Notations for representative domain decomposition settings. The nonlocal subdomain $\Omega_{nl}$ is represented by the green and red regions together, local subdomain $\Omega_l$ is represented in blue, the sharp interface $\Gamma_i$ is highlighted by red. Nonlocal Dirichlet boundary condition is given on $B\Omega_d$, and nonlocal Robin transmission condition is applied on the red region $B\Omega_i$. On the local subdomain, Dirichlet boundary condition is given on $\Gamma_d$ and Dirichlet transmission condition is applied on $\Gamma_i$. %Right: notations for the geometric estimates in Lemma \ref{Mdelta}. Here green represents the region in $B(\mathbf{x},\delta)$ which lies on the other side of the tangential line at $\overline{\mathbf{x}}$ with respect to $\Omega$, and cyan represents the region in $B(\mathbf{x},\delta)$ which lies between $\partial\Omega$ and the tangential line.
 }
 \label{fig1}
\end{figure}

In this section, we define the formulation for the solution $u(\mathbf{x})$ in a two-dimensional body occupying the domain $\Omega\subset\mathbb{R}^2$. The domain $\Omega$ is composed of two parts: the nonlocal subdomain $\Omega_{nl}$ (as shown by the green and red regions in Figure \ref{fig1}) where the problem is described by a nonlocal model based on integro-differential equations, and the classical theory subdomain $\Omega_l$ (as shown by the blue region in Figure \ref{fig1}) occupied by a local model described by classical PDEs. To develop a nonoverlapping coupling framework, for both the local and nonlocal models the interface boundary conditions are applied on a 1D curve, which is marked as $\Gamma_i$. On $\Gamma_i$, a classical Dirichlet type boundary condition is applied on the local side. In the nonlocal solver, to impose a generalization of the Robin type boundary condition on $\Gamma_i$, a modified nonlocal formulation is applied in a collar region $B\Omega_i$. 
%Further details of the Robin transmission conditions on the coupling interface will be provided in Section \ref{section:robin}. 
For the external boundary conditions, we assume that suitable Dirichlet boundary conditions are imposed on the local and nonlocal subdomains, without loss of generality. Specifically, on the external boundary of the nonlocal side, the Dirichlet boundary condition is applied on a collar consisting of all points outside the domain that interact with points inside the domain, which is marked by $B\Omega_d$ (as shown in grey in the left plot of Figure \ref{fig1}). On the external boundary of the local side, the Dirichlet boundary condition is applied on a sharp 1D curve $\Gamma_d$.

Although the proposed technique is applicable to more general problems, in the local subdomain $\Omega_l$ we model the problem with a classical heat equation. In the nonlocal subdomain $\Omega_{nl}$ we consider a nonlocal integro-differential equation (IDE) which is a nonlocal analog to the classical heat equation. {We also assume that $\Gamma_i$ and $\Omega_{nl}$ are both bounded and connected.} Note that since the local and nonlocal regions interact on a sharp interface $\Gamma_i$, the proposed coupling framework can be applied on the general heterogeneous local-to-nonlocal (LtN) coupling problems when there is large jump in physical (diffusivity) properties across the local-nonlocal interface $\Gamma_i$. Further details of the local and nonlocal problems will be described in Sections \ref{sec:nl_model} and \ref{sec:l_model}, respectively, and we leave the discussions of Robin transmission conditions on the coupling interface to a later Section \ref{section:robin}.

\subsection{Nonlocal Heat Problem}\label{sec:nl_model}

For the nonlocal subproblem we study compactly supported nonlocal integro-differential equations (IDEs) with radial kernels:
\begin{align}
\nonumber&\dot{u}_{nl,\delta}(\xb,t)-\alpha_{nl}\mcL_{\delta} u_{nl,\delta}(\xb,t)=f_{nl}(\xb,t),\quad \mathbf{x}\in\Omega_{nl}\\
\nonumber&u_{nl,\delta}(\xb,t)=u^D_{nl}(\xb,t),\quad \mathbf{x}\in B\Omega_{d}\\
\nonumber&u_{nl,\delta}(\xb,0)=u^{IC}(\xb),\quad \mathbf{x}\in\Omega_{nl}\\
& \text{ where }\quad\mcL_{\delta} u_{nl,\delta}(\xb,t):=2\int_{B(\mathbf{x},\delta)} J_\delta(|\mathbf{x}-\mathbf{y}|)(u_{nl,\delta}(\mathbf{y},t)-u_{nl,\delta}(\mathbf{x},t))d\mathbf{y}.\label{eqn:nonlocaldelta}
\end{align}
Here $B(\mathbf{x},\delta)$ is the ball centered at $\mathbf{x}$ with radius $\delta$, $u_{nl,\delta}(\mathbf{x},t)$ is the nonlocal solution, $\dot{u}_{nl,\delta}$ is the first derivative in time of $u_{nl,\delta}$, $\alpha_{nl}$ denotes the diffusivity coefficient for $\Omega_{nl}$, and $f_{nl}(\mathbf{x},t)$, $u^D_{nl}(\mathbf{x},t)$ are given data and nonlocal Dirichlet boundary condition, respectively. $u^{IC}(\xb)$ is the initial condition. The kernel function $J_\delta:\mathbb{R}\rightarrow\mathbb{R}$ is parameterized by a positive horizon parameter $\delta$ which measures the extent of nonlocal interaction. In this nonlocal setting every point in a domain interacts with a neighborhood Euclidean ball of surrounding points $B(\mathbf{x},\delta)$. Therefore, the external boundary conditions are no longer prescribed on a sharp interface $\partial\Omega_{nl}$, but on a layer of thickness $\horizon$ surrounding the domain that we refer to as $B\omg_{d}$.

In this paper we further take a popular choice of $J_\delta$ as a rescaled kernel given by \cite{cortazar2007boundary}
%\begin{equation}\label{eqn:Jdelta}
\begin{equation}\label{eqn:Jdelta}
 J_{\delta}(|\bm{\xi}|)=\dfrac{1}{\delta^{4}}J\left(\dfrac{|\bm{\xi}|}{\delta}\right),
\end{equation}
where $J:[0,\infty)\rightarrow[0,\infty)$ is a nonnegative integrable function with $\int_{\mathbb{R}^2} J(|\mathbf{z}|) |\mathbf{z}|^2 d\mathbf{z}=2$. 
%\DK{[and $c$ is a constant chosen to normalize $\int_{\mathbb{R}^2}J_\delta(\vert\bm{\xi}\vert)\vert\bm{\xi}\vert^2\,d\bm{\xi}$]}\YY{[I explicitly give $c$ now to avoid confusion.]} 
Similar as in \cite{tao2017nonlocal}, we assume that $J(r)$ is nonincreasing in $r$, strictly positive in $r\in[0,1]$ and vanishes when $r>1$.  It can be shown that at the limit of vanishing nonlocality, i.e. as $\horizon\to 0$, the above nonlocal diffusion operator $J_\delta$ converges to the classical Laplacian $\Delta$ operator (see, e.g. \cite{tao2017nonlocal,You2018}):
\begin{equation}\label{eq:local-limit}
\mcL_\delta v(\xb) = \Delta v(\xb) + \mcO(\horizon^2) {\rm D}^{(4)}v(\xb),
\end{equation}
where $v$ is a sufficiently smooth function and ${\rm D}^{(4)}$ is a combination of the fourth-order derivatives of $v$. Examples of properly scaled kernels in 2D include
\begin{displaymath}
J^1_\delta(r)=
\left\{\begin{array}{cl}
\dfrac{4}{\pi\delta^4},\; &\text{ for }r\leq \delta; \\
0,\;& \text{ for }r> \delta; \\
\end{array}
\right.
\quad {\rm and} \quad
J^2_\delta(r)=\left\{\begin{array}{cl}
\dfrac{3}{\pi\delta^3r},\; &\text{ for }r\leq \delta; \\
0,\;& \text{ for }r> \delta. \\
\end{array}
\right.
\end{displaymath}
%Although the discussions and the proposed formulations in this paper are not tied to a specific kernel, in numerical tests we demonstrate the numerical performances with $J_\delta(r)=J^1_\delta(r)$, and the method can also be applied to other choice.

To discretize the nonlocal subproblem spatially, we employ a meshfree quadrature rule based on the generalized moving least squares (GMLS) approximation framework \cite{Trask2018paper}. In the following we consider the Dirichlet-type boundary conditions only, leaving the Robin-type boundary condition on $\Gamma_i$ to Section \ref{section:robin}. The nonlocal subdomain $\Omega_{nl}$ and the nonlocal volumetric boundary $B\Omega_{d}$ are discretized by a collection of points $ \chi_{h} = \{\mathbf{x}_i\}_{\{i=1,2,\cdots,N_p\}} \subset \Omega_{nl} \cup B\Omega_{d}$, where the fill distance
\begin{equation}
	h := \sup\limits_{\mathbf{x}_i \in \chi_{h}} \min\limits_{1\leq j \leq N_p, j\neq i} |\mathbf{x}_i-\mathbf{x}_j|
\end{equation}
is a length scale characterizing the resolution of the point cloud, and $N_p$ denotes the total number of points. Similar as in \cite{You2018}, here we assume that the point set is quasi-uniform. For each point $\mathbf{x}_i\in \chi_h$, denoting the set of indices for points in $B(\mathbf{x}_i,\delta)$ as
\begin{equation}
I(\mathbf{x}_i) \equiv I(\mathbf{x}_i,\delta,\chi_{h}) := \{j\in\{1,\cdots,N_p\}:|\mathbf{x}_i-\mathbf{x}_j|<\delta\},
\end{equation}
and $\#I(\mathbf{x}_i)$ as the number of indices in $I(\mathbf{x}_i)$, we then aim to reconstruct a degree $m$ polynomial approximation $s_{u,\chi_{h},i}(\mathbf{x},t)$ for the nonlocal solution $u_{nl,\delta}(\mathbf{x},t)$ in $B(\mathbf{x}_i,\delta)$. Specifically, define as a basis for the $m$-th order polynomial space $\pi_m(\mathbb{R}^2)$ the set $\{p_1(\mathbf{x}),p_2(\mathbf{x}),\cdots,p_Q(\mathbf{x})\}$, $s_{u,\chi_{h},i}$ is the solution to the optimization problem 
\begin{equation}\label{eqn:gmls}
s_{u,\chi_{h},i}(\mathbf{x},t) =\underset{p \in \pi_m(\mathbb{R}^2)} {\operatorname{arg\,min}} \Bigg \{ \sum_{j=1}^{N_p}[u_{nl,\delta}(\mathbf{x}_j,t)-p(\mathbf{x}_j)]^2w(\mathbf{x}_i,\mathbf{x}_j)\Bigg\},
\end{equation}
where $w(\mathbf{x},\mathbf{y})$ is a translation-invariant positive weight function with compact support $\delta$:
\begin{displaymath}
w(\mathbf{x},\mathbf{y}) := \Phi_{\delta}(\mathbf{x}-\mathbf{y})=\left\{\begin{array}{cl}
(1-\frac{|\mathbf{x}-\mathbf{y}|}{\delta})^4, & \text{ when }|\mathbf{x}-\mathbf{y}|\leq \delta,\\
0& \text{ when }|\mathbf{x}-\mathbf{y}|> \delta.
\end{array}
\right.
\end{displaymath}
Here we note that for a quasi-uniform pointset with sufficiently large ratio $\delta/h$, 
%\DK{[Maybe it would be better to say ``sufficiently large $\delta/h$'', since we plan to hold that constant.  (There must also be some technical conditions, like all the points not being on a straight line, right?  Is that part of the definition of ``quasi-uniform'' for point clouds?)]} \YY{[Yes we need sufficiently large ratio between $\delta$ and $h$, as well as the quasi-uniform assumption. Modified.]} 
the optimization problem possesses a unique solution \cite{wendland2004scattered}
\begin{equation}
s_{u,\chi_{h},i}(\mathbf{x},t) = \tilde{u}_\delta(t)DP(P^{\mathrm{T}}DP)^{-1}R(\mathbf{x}),
\end{equation}
where 
\begin{align*}
\tilde{u}_\delta(t) &:= (u_{nl,\delta}(\mathbf{x}_j,t):j\in I(\mathbf{x}_i))^{\mathrm{T}} \in \mathbb{R}^{\#I(\mathbf{x}_i)},\qquad
P := (p_k(\mathbf{x}_j))_{j\in I(\mathbf{x}_i),1\leq k \leq Q} \in \mathbb{R}^{\#I(\mathbf{x}_i) \times Q},\\
D &:= \operatorname{diag}(\Phi_\delta(\mathbf{x}_i-\mathbf{x}_j):j \in I(\mathbf{x}_i)) \in \mathbb{R}^{\#I(\mathbf{x}_i) \times 
\#I(\mathbf{x}_i)},\qquad
R(\mathbf{x}) := (p_1(\mathbf{x}),\cdots,p_Q(\mathbf{x}))^{\mathrm{T}}\in \mathbb{R}^Q.
\end{align*}
Note that when $u_{nl,\delta} \in \pi_m(\mathbb{R}^2)$, the above reconstruction is exact, i.e., $s_{u,\chi_{h},i}(\mathbf{x},t)=u_{nl,\delta}(\mathbf{x},t)$. We then employ the reconstruction to evaluate the nonlocal model in \eqref{eqn:nonlocaldelta} and obtained the semi-discretized formulation for $\tilde{u}_\delta(t)$:
\begin{equation}\label{eqn:nonlocaldx}
\dot{{u}}_{nl,\delta}(\xb_i,t)-2\alpha_{nl}\tilde{u}_\delta(t)DP(P^{\mathrm{T}}DP)^{-1}\int_{B(\mathbf{x}_i,\delta)} J_{\delta}(|\mathbf{y}-\mathbf{x}_i|)(R(\mathbf{y})-R(\mathbf{x}_i))d\mathbf{y} = f_{nl}(\xb_i,t).
\end{equation}
For further details on analysis and implementation of the meshfree quadrature rule we refer the interested readers to the previous work \cite{Trask2018paper,You2018}, where we have employed this meshfree quadrature rule to develop asymptotically compatible spatial discretizations for static nonlocal diffusion model and peridynamics. To discretize \eqref{eqn:nonlocaldx} in time, in this paper we employ the backward Euler scheme for simplicity and solve for $(U_{\delta})^{k}_j\approx u_{nl,\delta}(\xb_j,t^k)$ with:
\begin{equation}\label{eqn:nonlocaldxdt}
\dfrac{1}{\Delta t}((U_{\delta})^{k+1}_i-(U_{\delta})^{k}_i)-2\alpha_{nl}\tilde{u}_\delta^{k+1}DP(P^{\mathrm{T}}DP)^{-1}\int_{B(\mathbf{x}_i,\delta)} J_{\delta}(|\mathbf{y}-\mathbf{x}_i|)(R(\mathbf{y})-R(\mathbf{x}_i))d\mathbf{y} = f_{nl}(\xb_i,t^{k+1}),
\end{equation}
with Dirichlet boundary condition $(U_{\delta})_j^k=u^D_{nl}(\xb_j,t^k)$ for $\xb_j\in B\Omega_d$. Here $\tilde{u}^k_{\delta}=((U_\delta)_j^k:j\in I(\mathbf{x}_i))^{\mathrm{T}} \in \mathbb{R}^{\#I(\mathbf{x}_i)}$. In the following we refer the nonlocal numerical solution with spatial discretization length scale $h$ at the $M$-th time step as $u^{M,h}_{nl,\delta}$.

\subsection{Local Heat Problem}\label{sec:l_model}

For the local subproblem we consider the classical heat equation
\begin{align}
\nonumber&\dot{u}_l(\xb,t)-\alpha_{l}\Delta u_l(\xb,t)=f_l(\xb,t),\quad \mathbf{x}\in\Omega_{l}\\
\nonumber&u_l(\xb,t)=u_l^D(\xb,t),\quad \mathbf{x}\in \Gamma_{d}\\
&u_l(\xb,0)=u^{IC}(\xb),\quad \mathbf{x}\in\Omega_{l}\label{eqn:local}
\end{align}
where $u_l(\xb,t)$ is the local solution, $\dot{u}_l$ is the first derivative of $u_l$ in time, $\alpha_l$ is the diffusivity in the local region, $u^{IC}$ is the initial condition, $f_l(\xb,t)$ is the given data, and $u_l^D(\xb,t)$ is the given Dirichlet boundary condition on the 1D external boundary $\Gamma_d$. Note that for coupling framework with overlapping regions such as \cite{yu2018partitioned}, it usually requires $\alpha_{l}=\alpha_{nl}$ such that the nonlocal model will be equivalent with the local model as $\delta\rightarrow 0$. However, since in the current paper a nonoverlapping coupling framework is considered, it is possible that $\alpha_{nl}\neq \alpha_l$.

The local subproblem is solved with a finite element method code based on the FEniCS package \cite{alnaes2015fenics,logg2012automated}. Spatially, with a test function $v$, when considering the Dirichlet boundary conditions only the local subproblem can be written into its weak form
\begin{equation}\label{eqn:localweak}
\int_{\Omega_l} \dot{u}_l(\xb,t){v}(\xb)d\xb+\alpha_l\int_{\Omega_l}\nabla {u}_l(\xb,t)\cdot\nabla {v}(\xb)d\xb=\int_{\Omega_l}f_l(\xb,t){v}(\xb)d\xb,\;\forall {v}\in \mathbf{W}_0,
\end{equation}
where the solution ${u}_l(\mathbf{x},t)\in \mathbf{W}=%
\{{w}(\mathbf{x})\in H^1(\Omega_l)|{w}(\mathbf{x})=u^D_l(\mathbf{x},t)\; %
\text{on}\; \Gamma_{d}\}$ and $\mathbf{W}_0=\{{w}(\mathbf{x})%
\in H^1(\Omega_l)|{w}(\mathbf{x})=0\; \text{on}\; \Gamma_{d}\}$. $\Omega_l$ is discretized with a regular quasi-uniform triangulation $\mathcal{T}_h(\Omega_l)$ of mesh size $h=\underset{T\in \mathcal{T}_h(\Omega_l)}{\max}h_T$, and the local solution $u_l$ is approximated by continuous linear finite elements\footnote{Because the partitioned procedure is employed, the proposed framework can also be applied to the case when the two subdomain are discretized with different discretization length scales, i.e., $h_{nl}\neq h_l$. However, in this paper we demonstrate the method and results for $h_{nl}=h_l=h$ unless otherwise stated.}. With linear shape functions $\psi_p(\xb)$  for each element, the local solution $u_l(\xb,t)$ and the test function $v_l(\xb)$ are expanded as
$$u_l(\xb,t)=\sum_{p=1}^{3}(U_l)_{ip}(t)\psi_p(\xb),\qquad v_l(\xb)=\sum_{p=1}^{3}(V_l)_{ip}\psi_p(\xb),$$
where $(U_l)_{ip}(t)$ and $(V_l)_{ip}$ are the expansion coefficients and $ip$ is the global index of the coefficient. Substuting the above expansions into the weak formulation \eqref{eqn:localweak} and assembling globally, we obtain
$$M_l\dot{\Ub}_l(t)+B_l\Ub_l(t)=\Fb_l(t).$$
Here $\Ub_l$ is the global vector of unknown expansion coefficients, $\Fb_l$ is the global vector of the external loads, $M_l$ is the mass matrix and $B_l$ is the stiffness matrix. We then employ the backward Euler scheme for time integration and solve for $(\Ub_l)^{k+1}$:
\begin{equation}\label{eqn:localdisc}
\dfrac{M_l}{\Delta t}({\Ub}^{k+1}_l-{\Ub}^{k}_l)+B_l\Ub_l^{k+1}=\Fb^{k+1}_l,
\end{equation}
at the $k$-th time step. In the following we denote the local numerical solution with mesh size $h$ at the $M$-th time step as $u^{M,h}_{l}$.

\section{Boundary Conditions for Nonlocal Problems}\label{section:robin}

\begin{figure}[!htb]\centering
 \subfigure{\includegraphics[width=0.25\textwidth]{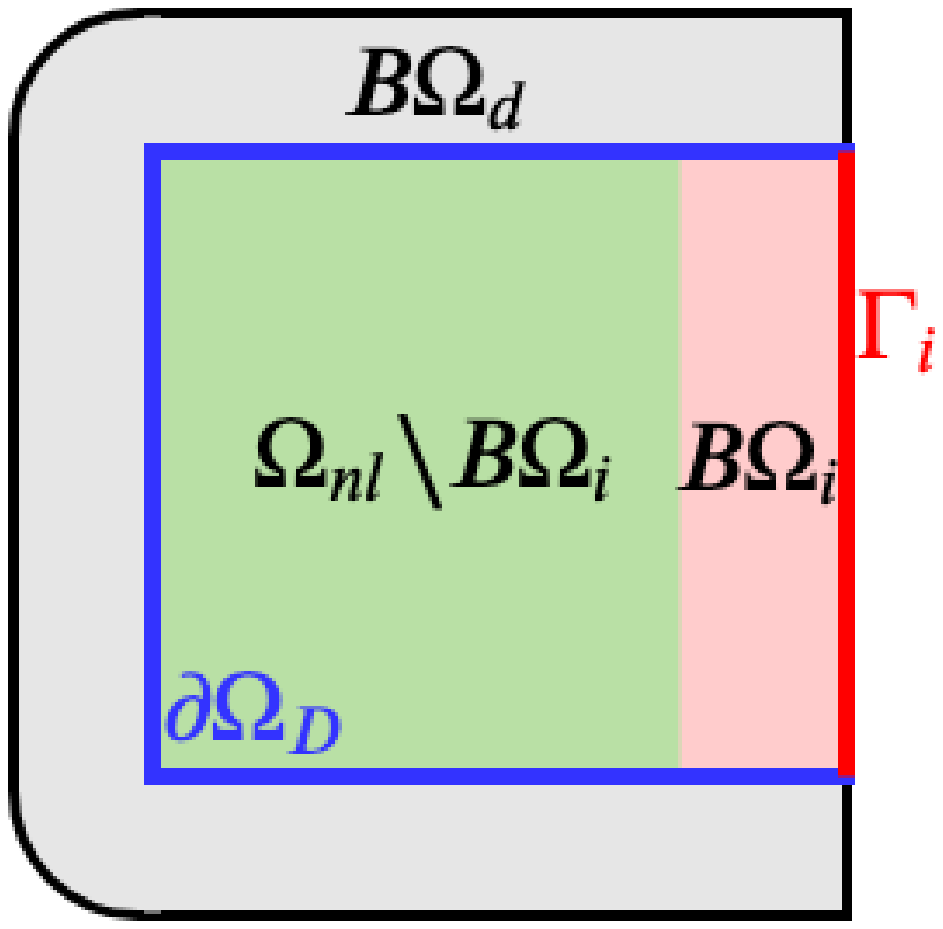}}$\qquad\qquad$
 \subfigure{\includegraphics[width=0.3\textwidth]{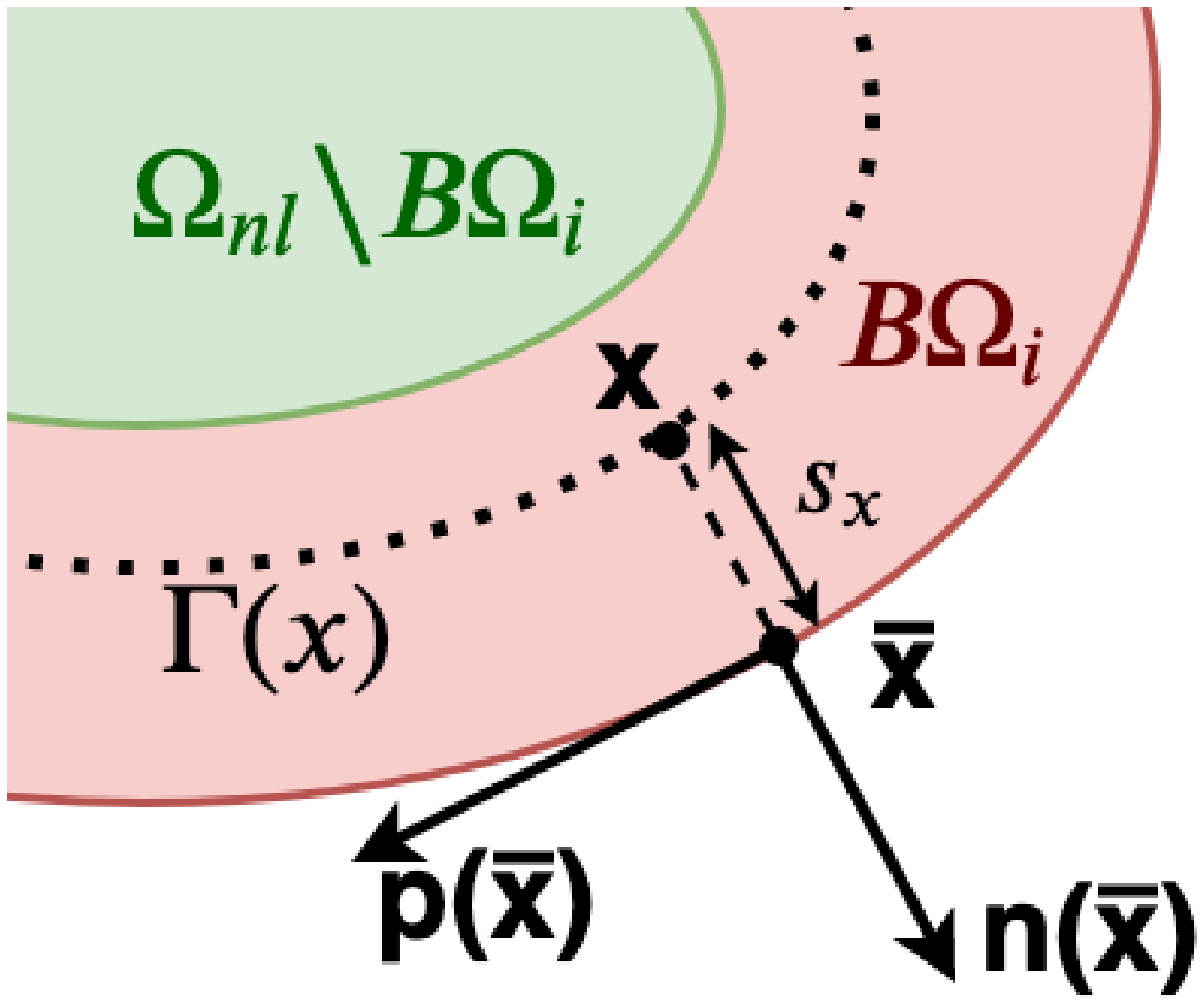}}
 %\subfigure{\includegraphics[scale=.4]{convex.pdf}
 %\put(-55,70){$\mathbf{x}$}\put(-20,70){$\overline{\mathbf{x}}$}
 %\put(-20,10){$\tau(\overline{\mathbf{x}})$}
 %\put(-15,90){$T_\delta$}
 %\put(-33,115){$A_\delta$}\put(-35,35){$A_\delta$}
 %\put(-50,135){$\partial\Omega$}
 %\put(-100,90){$B(\mathbf{x},\delta)$}}
 \caption{Notations for the Neumann and Robin-type boundary conditions in the nonlocal subdomain, where the nonlocal subdomain $\Omega_{nl}$ is represented by the green and red regions together. Nonlocal Dirichlet boundary condition is given on $B\Omega_d$, and nonlocal Robin transmission condition is applied on the red region $B\Omega_i$. On the corresponding local limit, Dirichlet boundary condition is given on $\partial\Omega_D$ and Robin transmission condition is applied on $\Gamma_i$. Right: Notations for the projection of point $\mathbf{x}\in B\Omega_i$, the corresponding unit tangential vector $\mathbf{p}(\overline{\mathbf{x}})$ and the unit normal vector $\mathbf{n}(\overline{\mathbf{x}})$. %Right: notations for the geometric estimates in Lemma \ref{Mdelta}. Here green represents the region in $B(\mathbf{x},\delta)$ which lies on the other side of the tangential line at $\overline{\mathbf{x}}$ with respect to $\Omega$, and cyan represents the region in $B(\mathbf{x},\delta)$ which lies between $\partial\Omega$ and the tangential line.
 }
 \label{fig2}
\end{figure}

In this section, we consider the nonlocal subproblem only, with problem setting as shown in Figure \ref{fig2}. Given that $\Omega_{nl}\in \mathbb{R}^2$ is a bounded, convex, connected and {$C^{3}$} domain, $\beta$ is the Robin coefficient, we seek a nonlocal analogue to the local Robin boundary condition $\beta u(\mathbf{x})+\dfrac{\partial u}{\partial \mathbf{n}}(\mathbf{x})=g(\mathbf{x})$, $\mathbf{x}\in \Gamma_i$ in the following corresponding classical problem
\begin{align}
\nonumber&\dot{u}_0(\xb,t)-\alpha_{nl}\Delta u_0(\xb,t)=f_{nl}(\xb,t),\quad \mathbf{x}\in\Omega_{nl}\\
\nonumber&\beta u_0(\mathbf{x},t)+\dfrac{\partial u_0}{\partial \mathbf{n}}(\mathbf{x},t)=g(\mathbf{x},t),\quad\text{ on }\Gamma_i\\
\nonumber&u_0(\xb,t)=u_{nl}^D(\xb,t),\quad \mathbf{x}\in \partial\Omega_{D}\\
&u_0(\xb,0)=u^{IC}(\xb),\quad \mathbf{x}\in\Omega_{nl}.\label{eqn:locallimit}
\end{align}
Here we assume that the body load, boundary conditions and initial conditions satisfy proper consistency conditions. As shown in Figure \ref{fig2}, here $\mathbf{n}({\mathbf{x}})$ is the unit exterior normal to $\Omega$ at ${\mathbf{x}}$, $\mathbf{p}({\mathbf{x}})$ is the unit tangential vector with orientation clockwise to $\mathbf{n}({\mathbf{x}})$, $\partial \Omega_D$ and $\Gamma_i$ are both 1D curves with classical Dirichlet and Robin-type boundary conditions defined on them, respectively. Nonlocal Dirichlet-type constraint is applied on $B\Omega_d=\{\mathbf{x}\in{\mathbb{R}^2\backslash\Omega_{nl}}:%
\text{dist}(\mathbf{x},\partial\Omega_{D})\leq\delta\}$. In the analysis of this section, we assume $u(\mathbf{x})=0$ on $B\Omega_d$ without loss of generality. Similarly, to apply the Robin-type constraint on $\Gamma_i$, we denote $B\Omega_{i}=\{\mathbf{x}\in\Omega_{nl}:\overline{\mathbf{x}}\in\Gamma_i,\text{dist}(\mathbf{x},\Gamma_{i})\leq\delta\}$. Here we assume sufficient regularity of the boundary {(e.g., that it satisfies the hypotheses of the $\epsilon$-neighborhood theorem from differential geometry)} that we may take $\delta$ sufficiently small so that for any $\mathbf{x}\in\Omega$ within distance $\delta$ to $\Gamma_i$, there exists a unique orthogonal projection of $\mathbf{x}$ onto $\Gamma_i$. We denote this projection as $\overline{\mathbf{x}}$. Therefore, one has
$\overline{\mathbf{x}}-\mathbf{x}=s_{\xb}\mathbf{n}(\overline{\mathbf{x}})$ for $\mathbf{x}\in B\Omega_{i}$, where $0\leq s_\xb\leq\delta$. 
%\DK{[I would bold the subscript $x$ in $s_\xb$, just for consistency.]}\YY{[Done.]}  
We also assume that for $\mathbf{x}\in B\Omega_i$, we can find a contour $\Gamma(\mathbf{x})$ which is parallel to $\Gamma_i$ {(i.e., a level-set of a signed distance function)}, as shown in the right plot of Figure \ref{fig2}. In the following contents, we denote $\mathbf{x}_{l}$ as the point with distance $l$ to $\mathbf{x}$ along $\Gamma(\mathbf{x})$ following the $\mathbf{p}(\overline{\mathbf{x}})$ direction, and $\mathbf{x}_{-l}$ as the point with distance $l$ to $\mathbf{x}$ in the opposite direction. Moreover, we employ the following notations for the directional components of the Hessian matrix of a scalar function $v$:
%
%\DK{[The notation $\nabla^2$ is usually taken to mean $\nabla\cdot\nabla$, i.e., $\Delta$, but it appears to be used for $\nabla\circ\nabla$ here.  (I agree that $A\circ A$ is the usual interpretation of ``$A^2$'' for some operator $A$, but, historically, $\nabla$ is treated differently.)]}\YY{[Done. Changed to $\nabla\circ\nabla$.]}
\begin{align*}
[v({{\mathbf{x}}})]_{pp}:=\mathbf{p}^T(\overline{\mathbf{x}}) [\nabla\circ\nabla v({{\mathbf{x}}})] \mathbf{p}(\overline{\mathbf{x}}),\qquad [v({{\mathbf{x}}})]_{nn}:=\mathbf{n}^T(\overline{\mathbf{x}}) [\nabla\circ\nabla v({{\mathbf{x}}})] \mathbf{n}(\overline{\mathbf{x}}),\qquad[v({{\mathbf{x}}})]_{pn}:=\mathbf{p}^T(\overline{\mathbf{x}}) [\nabla\circ\nabla v({{\mathbf{x}}})] \mathbf{n}(\overline{\mathbf{x}}),
\end{align*}
and the higher order derivative components are similarly defined.

With the above notations and assumptions, in this section we first introduce a nonlocal Neumann boundary condition in Section \ref{sec:neumann}, then estimate the order of convergence rate to the corresponding local limit. With the Neumann-type boundary condition, we then propose a new generalization of classical Robin condition for nonlocal problems in Section \ref{sec:robinmath}. 
%In Section \ref{sec:linfty} we prove the $O(\delta^2)$ convergence rate of the continuous nonlocal solution $u_\delta$ to $u_0$ in the $L^{\infty}(\Omega)$ norm without extra regularity assumptions on $u_\delta$. 
To verify the asymptotic convergence of the proposed boundary treatment, we discretize the proposed Robin-type constraint problem with the meshfree quadrature rule \cite{Trask2018paper,You2018} in Section \ref{sec:Robintest}, then use numerical examples to demonstrate the convergence of the discrete model to the analytical local limit as the discretization length scale $h$, time step size $\Delta t$ and the nonlocal interaction length scale $\delta$ all vanish simultaneously.

\subsection{A Nonlocal Neumann-Type Boundary Condition}\label{sec:neumann}

When $\beta=0$, in \eqref{eqn:locallimit} the Neumann boundary condition is imposed on $\Gamma_i$. Inspired by \cite{cortazar2008approximate,tao2017nonlocal,You2018}, we propose the nonlocal Neumann-type boundary condition by firstly considering the following modification for $\xb\in B\Omega_i$:
\begin{align}
 \nonumber&\dot{u}_{nl,\delta}(\xb,t)-2\alpha_{nl}\int_{\Omega_{nl}} J_{\delta}(|\mathbf{x}-\mathbf{y}|)(u_{nl,\delta}(\mathbf{y},t)-u_{nl,\delta}(\mathbf{x},t))d\mathbf{y}
 -\alpha_{nl}\int_{\mathbb{R}^2\backslash\Omega_{nl}} J_{\delta}(|\mathbf{x}-\mathbf{y}|)(\mathbf{y}-\mathbf{x})\cdot\mathbf{n}(\overline{\mathbf{x}})%
 \left(g({\mathbf{x}},t)+g({\mathbf{y}},t)\right)d\mathbf{y}\\
&-\alpha_{nl}\int_{\mathbb{R}^2\backslash\Omega_{nl}} J_{\delta}(|\mathbf{x}-\mathbf{y}|)|(\mathbf{y}-\mathbf{x})\cdot\mathbf{p}(\overline{\mathbf{x}})|^2%
 d\mathbf{y}[ u_{nl,\delta}({{\mathbf{x}}},t)]_{pp}=f_{nl}(\mathbf{x},t).\label{eqn:formula1}
\end{align}
The last two terms on the left hand side of the above formulation provide an approximation to
$$-2\alpha_{nl}\int_{\mathbb{R}^2\backslash\Omega_{nl}} J_{\delta}(|\mathbf{x}-\mathbf{y}|)(u_{nl,\delta}(\mathbf{y},t)-u_{nl,\delta}(\mathbf{x},t))d\mathbf{y},$$
which account for the contributions from material points outside the nonlocal domain \cite{cortazar2008approximate,park2006surface}. To apply the Robin transmission condition $g(\mathbf{x},t)$ which is defined only on the sharp interface $\Gamma_i$, the $g({\mathbf{x}},t)$ and $g({\mathbf{y}},t)$ terms in \eqref{eqn:formula1} will be approximated with the following (local) extensions
\begin{align}
 \nonumber&g({\mathbf{x}},t)\approx g(\overline{\mathbf{x}},t)+ \dfrac{1}{\alpha_{nl}}(\mathbf{x}-\overline{\mathbf{x}})\cdot\mathbf{n}(\overline{\mathbf{x}})(\dot{u}_{nl,\delta}({{\mathbf{x}}},t)-f_{nl}({{\mathbf{x}}},t))-%
 (\mathbf{x}-\overline{\mathbf{x}})\cdot\mathbf{n}(\overline{\mathbf{x}}) [ u_{nl,\delta}({{\mathbf{x}}},t)]_{pp},\\
 \nonumber&g({\mathbf{y}},t)\approx g(\overline{\mathbf{x}},t)+\dfrac{1}{\alpha_{nl}}(\mathbf{y}-\overline{\mathbf{x}})\cdot\mathbf{n}(\overline{\mathbf{x}}) (\dot{u}_{nl,\delta}({{\mathbf{x}}},t)-f_{nl}({{\mathbf{x}}},t))-%
 (\mathbf{y}-\overline{\mathbf{x}})\cdot\mathbf{n}(\overline{\mathbf{x}}) [ u_{nl,\delta}({{\mathbf{x}}},t)]_{pp}.
\end{align}
Furthermore, we replace $[ u_{nl,\delta}({{\mathbf{x}}},t)]_{pp}$ with its approximation $2\int_{-\delta}^{\delta} H_{\delta}(|l|) (u_{nl,\delta}(\mathbf{x}_{l},t)-u_{nl,\delta}(\mathbf{x},t)) d\mathbf{x}_{l}-\kappa(\overline{\mathbf{x}})g(\overline{\mathbf{x}},t)$.
Here $d\mathbf{x}_{l}$ is the line integral along the contour $\Gamma(\mathbf{x})$, $\kappa(\overline{\mathbf{x}})$ is the curvature of $\partial\Omega_{nl}$ at $\overline{\mathbf{x}}$, and  $H_{\delta}(|r|)=\dfrac{1}{\delta^3}H\left(\dfrac{|r|}{\delta}\right)$ is the kernel for %
1D nonlocal diffusion model such that $H:[0,\infty)\rightarrow[0,\infty)$ is a nonnegative and continuous function with $\int_{\mathbb{R}} H(|z|) |z|^2 dz=1$. $H(r)$ is nonincreasing in $r$, strictly positive in $[0,1]$ and vanishes for $|z|>1$. Moreover, we add a further requirement on $H$ that $\int_{\mathbb{R}} H(z) dz:=C_H< \infty$. Substituting the above two approximations into \eqref{eqn:formula1}, we obtain the following model
\begin{align}
\nonumber&Q_\delta(\xb)\dot{u}_{nl,\delta}(\xb,t)-2\alpha_{nl}\int_{\Omega_{nl}} J_{\delta}(|\mathbf{x}-\mathbf{y}|)(u_{nl,\delta}(\mathbf{y},t)-u_{nl,\delta}(\mathbf{x},t))d\mathbf{y}-2\alpha_{nl}M_\delta(\mathbf{x}) \int_{-\delta}^{\delta} H_{\delta}(|l|) (u_{nl,\delta}(\mathbf{x}_{l},t)-u_{nl,\delta}(\mathbf{x},t)) d\mathbf{x}_{l}\\
\nonumber=&Q_\delta(\xb)f_{nl}({{\mathbf{x}}},t)+\alpha_{nl}V_\delta(\xb)g(\overline{\mathbf{x}},t)%\label{eqn:formula3}
\end{align}
where
\begin{equation}\label{eqn:Q}
Q_\delta(\xb):=1-\int_{\mathbb{R}^2\backslash\Omega_{nl}} J_{\delta}(|\mathbf{x}-\mathbf{y}|)%
 \left[|(\mathbf{y}-\overline{\mathbf{x}})\cdot\mathbf{n}(\overline{\mathbf{x}})|^2-|(\mathbf{x}-\overline{\mathbf{x}})\cdot\mathbf{n}%
 (\overline{\mathbf{x}})|^2\right]d\mathbf{y},
\end{equation}
\begin{equation}\label{eqn:S}
V_\delta(\xb):=2\int_{\mathbb{R}^2\backslash\Omega_{nl}} J_{\delta}(|\mathbf{x}-\mathbf{y}|)(\mathbf{y}-\mathbf{x})\cdot\mathbf{n}(\overline{\mathbf{x}}) d\mathbf{y}-M_\delta(\mathbf{x})\kappa(\overline{\mathbf{x}}),
\end{equation}
\begin{equation}\label{eqn:M}
M_\delta(\mathbf{x}):=\int_{\mathbb{R}^2\backslash\Omega_{nl}} J_{\delta}(|\mathbf{x}-\mathbf{y}|)\left[|(\mathbf{y}-\mathbf{x})%
 \cdot\mathbf{p}(\overline{\mathbf{x}})|^2-|(\mathbf{y}-\overline{\mathbf{x}})\cdot\mathbf{n}(\overline{\mathbf{x}})|^2+|(\mathbf{x}-\overline{\mathbf{x}})\cdot\mathbf{n}(\overline{\mathbf{x}})|^2\right]%
 d\mathbf{y}.
 \end{equation}
Thus, by defining the nonlocal operator:
\begin{align}
L_{N\delta} u:=&2\int_{\Omega_{nl}} J_{\delta}(|\mathbf{x}-\mathbf{y}|)(u(\mathbf{y},t)+u(\mathbf{x},t))d\mathbf{y}+2M_\delta(\mathbf{x}){\int_{-\delta}^{\delta} H_{\delta}(|l|) (u(\mathbf{x}_{l},t)-u(\mathbf{x},t)) d\mathbf{x}_{l}}\label{eqn:LN}
\end{align}
in problems with Neumann-type boundary conditions we obtain the following proposed nonlocal formulation
\begin{align}
\nonumber&\dot{u}_{nl,\delta}(\xb,t)-\alpha_{nl}\mcL_{\delta} u_{nl,\delta}(\xb,t)=f_{nl}(\xb,t),\quad \mathbf{x}\in\Omega_{nl}\backslash B\Omega_{i}\\
\nonumber&Q_\delta(\xb)\dot{u}_{nl,\delta}(\xb,t)-\alpha_{nl}\mcL_{N\delta} u_{nl,\delta}(\xb,t)=Q_\delta(\xb)f_{nl}(\xb,t)+\alpha_{nl}V_\delta (\xb)g(\overline{\mathbf{x}},t),\quad \mathbf{x}\in B\Omega_{i}\\
\nonumber&u_{nl,\delta}(\xb,t)=u^D_{nl}(\xb,t),\quad \mathbf{x}\in B\Omega_{d}\\
&u_{nl,\delta}(\xb,0)=u^{IC}(\xb).\quad \mathbf{x}\in\Omega_{nl}\label{eqn:nonlocaleqn}
\end{align}
The corresponding nonlocal energy seminorm $||\cdot||_{S_\delta}$ is given by
{\begin{align*}
||v||^2_{S_\delta(\Omega_{nl})}=&\int_{\Omega_{nl}}\int_{\Omega_{nl}} J_{\delta}(|\mathbf{x}-\mathbf{y}|)[v(\mathbf{y})-v(\mathbf{x})]^2d\mathbf{y}d\mathbf{x}
+\int_{B\Omega_i} M_\delta(\mathbf{x}) \int_{-\delta}^{\delta} H_{\delta}(|l|) [v(\mathbf{x}_{l})-v(\mathbf{x})]^2d\mathbf{x}_{l}d\mathbf{x}
\end{align*} }
which defines the energy space\footnote{We note that for a fixed $\delta$ and integrable kernels $J$, $H$, based on the results in \cite{Ponce2004estimate,du2012analysis,mengesha2013analysis,You2018} we have
$$||u||_{L^2(\Omega_{nl})}\leq C_1||u||_{S_\delta(\Omega_{nl})}\leq C_2(\delta)||u||_{L^2(\Omega_{nl})},$$
where $C_1, C_2(\delta)$ are constants independent of $u$ but $C_2(\delta)$ depends on $\delta$. Therefore, ${S_\delta(\Omega)}$ is equivalent to the space of $L^2(\Omega)$ functions.}
%\DK{[How can the energy space be equivalent to $L^2$ if it involves line integrals over sets of measure zero?  The function $v$ is not even well-defined on sets of measure zero in $L^2$, which is really just a set of equivalence classes of functions that are a.e. equal.  Let's say I add 1000 to the value of $v$ on $\Gamma_i$.  It's still in the same equivalence class as $v$, because $\Gamma_i$ has measure zero, so, in $L^2$, it's ``the same function''.  However, the value of the energy norm above will change.]}\YY{[See the comments on page 11.]}
\begin{displaymath}
 S_{\delta}(\Omega_{nl})=\left\{v\in L^2(\Omega_{nl}): ||v||_{S_\delta(\Omega_{nl})}<\infty\right\}.
\end{displaymath}

We now develop the analysis for homogeneous Neumann-type constraints, i.e., $g(\mathbf{x},t)=0$. Throughout this section, we consider the symbol ``$C$'' to indicate a generic constant that is independent of $\delta$, but may have different numerical values in different situations. Moreover, we make a critical geometric assumption for the simplicity of analysis (as illustrated by the left plot of Figure \ref{graph}): let $\{\mathbf{z}_1,\mathbf{z}_2\}:= \partial\Omega_D\cap \Gamma_i$, $\pi_{\partial\Omega_{nl}}$ be the projection operator onto $\partial \Omega_{nl}=\partial\Omega_D\cup \Gamma_i$, and $\tau(\mathbf{z}_1)$ (resp. $\tau(\mathbf{z}_2)$) be the tangent line to $\partial\Omega_{nl}$ at $\mathbf{z}_1$ (resp. $\mathbf{z}_2$), then we assume that the intersecting point $\tilde{\mathbf{z}}:=\tau(\mathbf{z}_1)\cap \tau(\mathbf{z}_2)$ satisfies $\pi_{\partial\Omega_{nl}}(\tilde{\mathbf{z}})\in \Gamma_i$. Here we note that due to the convexity of $\Omega_{nl}$, the map $\pi_{\partial\Omega_{nl}}(\mathbf{x})$ is always well defined and single-valued for any point $\mathbf{x}\notin\Omega_{nl}$.  %\DK{[If I'm understanding this assumption correctly, it's definitely not satisfied by the numerical test with the two square domains (since $\Gamma_i\cap\partial\Omega_D$ consists of corners, without well-defined tangents); we might point out that this assumption is probably unnecessarily-strong.]}\YY{[This assumption is for the definition of barrier function in \eqref{eqn:phi}. The example you mentioned still satisfy the assumption because $\tau(\mathbf{z}_3)$ will coincide with the $\Gamma_i$ and the estimates for $\phi(\xb)$ after \eqref{eqn:phi} still hold true. I added a remark below to address this special case.]}

\begin{remark}
When $\Gamma_i$ is flat, $\tau(\mathbf{z}_1)$ and $\tau(\mathbf{z}_2)$ coincide. One can take the intersection point $\tilde{\mathbf{z}}$ as any point on $\Gamma_i$, and the analysis below still holds true.
\end{remark}

  With the analysis in \cite[Lemma~3.1]{You2018}, we note that there exists a $\overline{\delta}>0$ such that for $\delta\leq \overline{\delta}$, $\mathbf{x}\in B\Omega_i$ we have $0\leq M_\delta(\mathbf{x})\leq C$. Moreover, with the geometric assumptions on $\Omega_{nl}$ we have bounds for $Q_\delta(\xb)$:

\begin{figure}[!htb]
\centering
\subfigure{\includegraphics[scale=0.4]{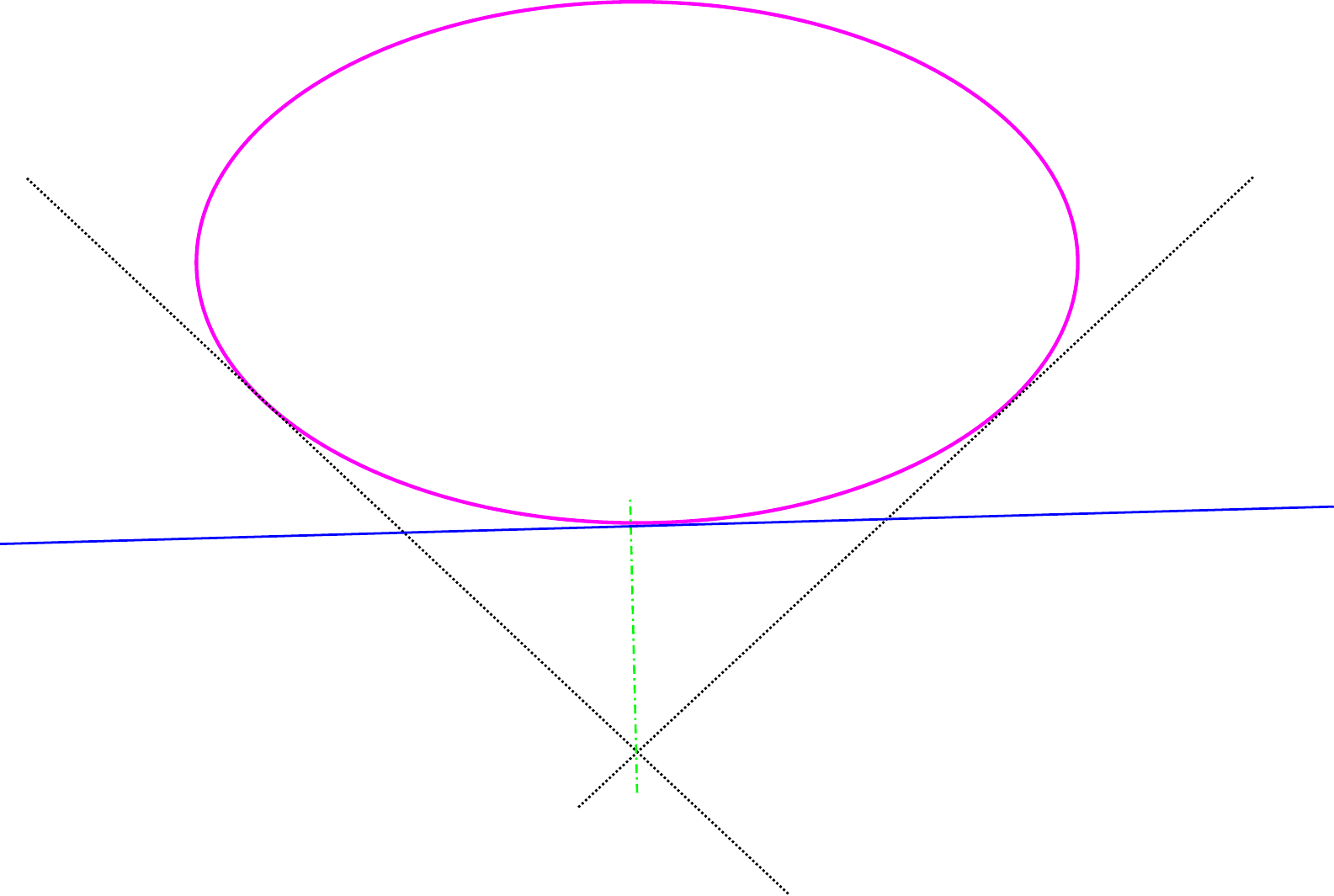}  
%\put(-200,130){$\Omega$}
\put(-100,115){$\partial\Omega_D$}\put(-80,60){$\Gamma_i$}
\put(-155,60){$\mathbf{z}_1$}\put(-45,60){$\mathbf{z}_2$}
\put(-100,45){$\mathbf{z}_3$}
\put(-200,70){$\Pi$}
\put(-180,100){$\tau(\mathbf{z}_1)$}\put(-35,100){$\tau(\mathbf{z}_2)$}
\put(-25,45){$\tau(\mathbf{z}_3)$}
\put(-100,10){$\tilde{\mathbf{z}}$}}$\qquad\qquad$
\subfigure{\includegraphics[scale=.4]{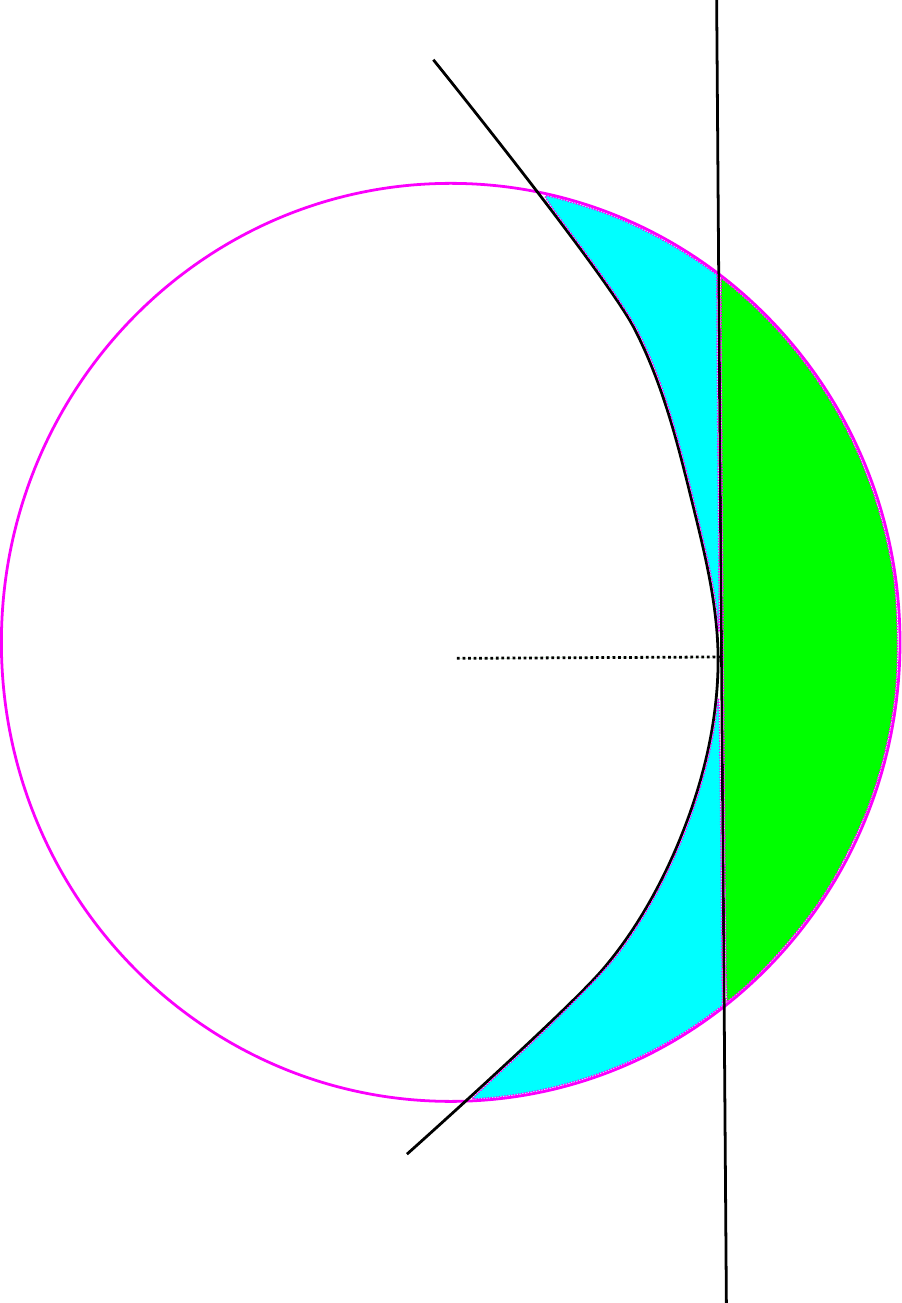}
  \put(-55,70){$\mathbf{x}$}
  \put(-20,70){$\overline{\mathbf{x}}$}
  \put(-20,10){$\tau(\overline{\mathbf{x}})$}
  \put(-15,90){$D_\delta$}
  \put(-33,115){$A_\delta$}
  \put(-35,35){$A_\delta$}
  \put(-50,135){$\Gamma_i$}
  \put(-100,90){$B(\mathbf{x},\delta)$}}
\caption{Notation for the geometric estimates in Section \ref{sec:robinmath}. Left: illustration of the geometric assumption and notation for the barrier function $\phi(\mathbf{x})$ definition in Theorem \ref{thm:conv}. Right: regions in Lemma \ref{lem:Q}. Green represents $D_\delta$, the region in $B(\mathbf{x},\delta)$ which lies on the other side of the tangential line at $\overline{\mathbf{x}}$ with respect to $\Omega_{nl}$. Cyan represents $A_\delta$, the region in $B_\delta(\mathbf{x})$ which lies between $\mathbb{R}^2\backslash\Omega_{nl}$ and the tangential line.}
 \label{graph}
\end{figure}
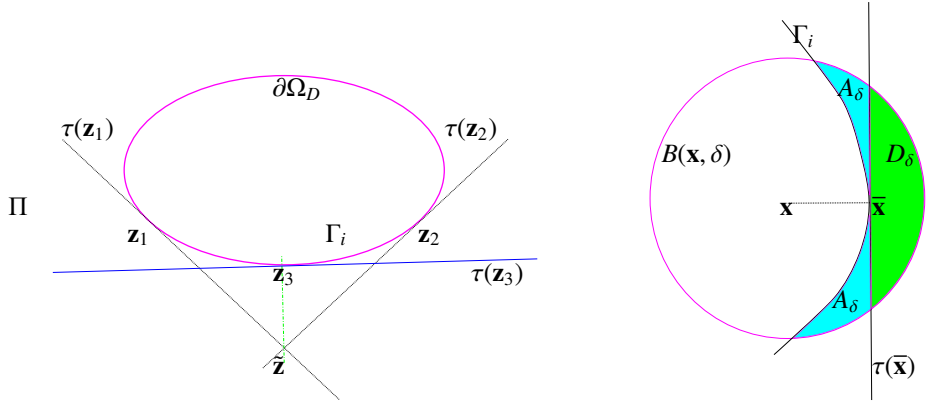

\begin{lem}\label{lem:Q}
Let $\Omega_{nl}\in \mathbb{R}^2$ {be a convex and $C^3$ domain}, then there exists $\overline{\delta}>0$ such that for $0<\delta\leq \overline{\delta}$, $Q_\delta(\xb)$ is bounded from above and below independent of $\delta$. Specifically, $0<\dfrac{1}{2}-C_q\delta\leq Q_\delta(\xb) \leq \dfrac{3}{2}+C_q\delta$ where $C_q>0$ is a constant independent of $\delta$.
\end{lem}

\begin{proof}
As shown in the right plot of Figure \ref{graph}, we note that
\[Q_\delta(\mathbf{x})=1- 
\int_{D_\delta}  J_{\delta}(|\mathbf{x}-\mathbf{y}|)%
 \left[|(\mathbf{y}-\overline{\mathbf{x}})\cdot\mathbf{n}(\overline{\mathbf{x}})|^2-|(\mathbf{x}-\overline{\mathbf{x}})\cdot\mathbf{n}%
 (\overline{\mathbf{x}})|^2\right] d \mathbf{y}
-\int_{A_\delta}  J_{\delta}(|\mathbf{x}-\mathbf{y}|)%
 \left[|(\mathbf{y}-\overline{\mathbf{x}})\cdot\mathbf{n}(\overline{\mathbf{x}})|^2-|(\mathbf{x}-\overline{\mathbf{x}})\cdot\mathbf{n}%
 (\overline{\mathbf{x}})|^2\right] d \mathbf{y}.\]
With $\tau(\overline{\mathbf{x}})$ representing the tangent line to $\Gamma_i$
 at  $\overline{\mathbf{x}}$, here $D_\delta$ is the region of $B(\mathbf{x},\delta)$ on the side of $\tau(\overline{\mathbf{x}})$ {\em not} containing $\Omega_{nl}$ (as shown in the green region of Figure \ref{graph}), and $A_\delta:=B(\mathbf{x},\delta) \backslash (D_\delta\cup \Omega_{nl})$ (as shown in the cyan region of Figure \ref{graph}). We consider first the $D_\delta$ part. When $\mathbf{y}\in D_\delta$,  we note that $|(\mathbf{y}-{\mathbf{x}})\cdot\mathbf{n}(\overline{\mathbf{x}})|^2\geq|(\mathbf{x}-\overline{\mathbf{x}})\cdot\mathbf{n}(\overline{\mathbf{x}})|^2$ and $|(\mathbf{y}-{\mathbf{x}})\cdot\mathbf{n}(\overline{\mathbf{x}})|^2\geq|(\mathbf{y}-\overline{\mathbf{x}})\cdot\mathbf{n}(\overline{\mathbf{x}})|^2$, therefore
 %\DK{[I don't see why this inequality is true. Can't $\mathbf{y}\in D_\delta$ be taken arbitrarily close to $\overline{\mathbf{x}}$, thus making the left-hand side arbitrarily small (while the right-hand side is fixed w.r.t. $\mathbf{y}$)?]}\YY{[Ooops you are right! I changed the bound and the proof.]}
\begin{align}
 \nonumber &\int_{D_\delta}  J_\delta( |\mathbf{x}-\mathbf{y}| ) \left[|(\mathbf{y}-\overline{\mathbf{x}})\cdot\mathbf{n}(\overline{\mathbf{x}})|^2-|(\mathbf{x}-\overline{\mathbf{x}})\cdot\mathbf{n}%
 (\overline{\mathbf{x}})|^2\right] d \mathbf{y} \geq -\int_{D_\delta}  J_\delta( |\mathbf{x}-\mathbf{y}| ) |(\mathbf{y}-{\mathbf{x}})\cdot\mathbf{n}(\overline{\mathbf{x}})|^2 d \mathbf{y}\geq -\dfrac{1}{2}.\\
 \nonumber &\int_{D_\delta}  J_\delta( |\mathbf{x}-\mathbf{y}| ) \left[|(\mathbf{y}-\overline{\mathbf{x}})\cdot\mathbf{n}(\overline{\mathbf{x}})|^2-|(\mathbf{x}-\overline{\mathbf{x}})\cdot\mathbf{n}%
 (\overline{\mathbf{x}})|^2\right] d \mathbf{y} \leq\int_{D_\delta}  J_\delta( |\mathbf{x}-\mathbf{y}| ) |(\mathbf{y}-\overline{\mathbf{x}})\cdot\mathbf{n}(\overline{\mathbf{x}})|^2 d \mathbf{y}\leq \dfrac{1}{2}.
\end{align} 
For the $A_\delta$ region, similar as in \cite[Lemma~3.1]{You2018}, it can be shown that the area of $A_\delta$ satisfies $|A_\delta|\leq C \delta^3$ since $\Omega_{nl}$ is a $C^3$ domain. Hence there exists a $\overline{\delta}$ such that for $\delta<\overline{\delta}$,
$$\left|\int_{A_\delta}  J_{\delta}(|\mathbf{x}-\mathbf{y}|)%
 \left[|(\mathbf{y}-\overline{\mathbf{x}})\cdot\mathbf{n}(\overline{\mathbf{x}})|^2-|(\mathbf{x}-\overline{\mathbf{x}})\cdot\mathbf{n}%
 (\overline{\mathbf{x}})|^2\right] d \mathbf{y}\right|\leq C_q\delta,$$
and $\dfrac{1}{2}-C_q\delta\leq Q_\delta(\xb)\leq \dfrac{3}{2}+C_q\delta$.
%\DK{[Considering that the bounds on integrals are $\delta$-dependent, then it seems like the choice of constants for the bounds on $Q_\delta$ are somewhat arbitrary.  For instance, $C\delta$ can be taken less than any arbitrary $\epsilon$ if we add the stipulation ``for $\delta$ sufficiently small''.  Is there a reason to stop at $\epsilon = 1/6$?]}\YY{[It doesn't have to be $\epsilon=1/6$. Actually $\epsilon$ can be any positive constant. Will that look better if we change keep the $\epsilon$ and change the conclusion to $\dfrac{1}{2}-\epsilon\leq Q_\delta(\xb)\leq \dfrac{3}{2}+\epsilon$, where $\epsilon\in(0,1/2)$?]}\DK{[Maybe best would be to have the bounds directly in terms of $C$ and $\overline{\delta}$.]}\YY{[Sounds good!]}
\end{proof}
For problem \eqref{eqn:nonlocaleqn} we have the nonlocal maximum principle stated below
%\DK{[What is the precise meaning of $C(\Omega)$ here?  Continuous and bounded?  If $\Omega$ is closed, doesn't this imply that a function in $C(\Omega)$ is bounded on the boundary?  If so, then some conditions in this lemma are redundant.]}\YY{[Yes $\Omega$ should be bounded. I added the assumption of bounded and connected in page 5. I also assume $\Gamma_i$ is connected since it's included in the geometric assumption in page 13.]}
\begin{lem}\label{eqn:maxprinciple}
{ For $u\in C^1(0,T;C(\overline{\Omega_{nl}})\cap C(B\Omega_{d}))\cup C([0,T]\times \overline{\Omega_{nl}})$, $u(\xb,t)$ bounded on $(\xb,t)\in B\Omega_{d}\times [0,T]$, assuming that $u$ satisfies $\dot{u}-\alpha_{nl}L_\delta u\leq 0$ for all $x\in \Omega_{nl}\backslash B\Omega_{i}$ and $Q_\delta \dot{u}-\alpha_{nl}L_{N\delta} u\leq 0$ for all $x\in B\Omega_{i}$, we have
\begin{equation}\label{eqn:maxp}
 \sup_{(\mathbf{x},t)\in (\overline{\Omega_{nl}}\cup B\Omega_{d}) \times [0,T]} u(\mathbf{x},t)\leq \max\left\{\sup_{\mathbf{x}\in \overline{\Omega_{nl}}\cup B\Omega_{d}} u(\mathbf{x},0),\sup_{(\mathbf{x},t)\in B\Omega_{d}\times[0,T]} u(\mathbf{x},t)\right\}. 
\end{equation}
}
\end{lem}
\begin{proof}
Assuming that \eqref{eqn:maxp} doesn't hold true, then there exists $(\mathbf{x}^*,t^*)\in (\Omega_{nl}\cup\Gamma_i)\times (0,T]$ such that $u(\mathbf{x}^*,t^*)$ achieves the maximum. One can then obtain a contradiction following a similar argument as \cite{You2018}:

%\DK{[Where does the zero lower bound on $\dot{u}(\mathbf{x}^*,t^*)$ come from?]}\YY{[Based on the assumption we have ${u}(\mathbf{x}^*,t^*)=\max_t {u}(\mathbf{x}^*,t)$ for any $t\in(0,T]$ and ${u}(\mathbf{x}^*,t)\in C^1((0,T))$. Therefore, $t^*$ is either a critical point or $t^*=T$. For either case $\dot{u}(\xb^*,t^*)\geq 0$.]}

%\DK{[Technically, it doesn't look like either case covers $(\mathbf{x}^*,t^*)\in\Gamma_i$.]}\YY{[From the definition of $B\Omega_i$ in page 10, $\Gamma_i\subset B\Omega_i$. So it is covered in Case 2.]}

Case 1: if $\xb^*\in \Omega_{nl}\backslash B\Omega_i$, then
$\dot{u}(\xb^*,t^*)\geq 0$, $-\alpha_{nl}\mcL_{\delta} u(\xb^*,t^*)\geq 0$. Therefore, $\dot{u}(\xb^*,t^*)=0$ and $\mcL_{\delta} u(\xb^*,t^*)=0$. For all $\mathbf{y}\in (\overline{\Omega_{nl}}\cup B\Omega_{d})\cap B(\mathbf{x}^*,\delta)$,  $u(\mathbf{y},t^*)=u(\mathbf{x}^*,t^*)$ achieves the maximum.

Case 2: if $\xb^*\in B\Omega_i$, then $Q_\delta(\xb)\dot{u}(\xb,t)\geq 0$, $-\alpha_{nl}\mcL_{N\delta} u(\xb,t)\geq 0$. Therefore, $Q_\delta(\xb)\dot{u}(\xb,t)=-\alpha_{nl}\mcL_{N\delta} u(\xb,t)=0$, and $u(\mathbf{y},t^*)=u(\mathbf{x}^*,t^*)$ achieves the maximum for all $\mathbf{y}\in \overline{\Omega_{nl}}\cap B(\mathbf{x}^*,\delta)$. 

We then apply the same arguments with $\mathbf{y}$ in place of $\mathbf{x}^*$. This process can be repeated until the region where $u(\mathbf{z},t^*)=\underset{\overline{\Omega_{nl}}\cup B\Omega_{d}}\sup u$  expands to the entire domain of $\overline{\Omega_{nl}}\cup B\Omega_{d}$. In other words, to have a global maximum inside $\Omega_{nl}$, the only possibility is for $u(\cdot,t^*)$ to be constant on $\overline{\Omega_{nl}}\cup B\Omega_{d}$, which contradicts with the assumption.
\end{proof}

Moreover, when considering a semi-discretized problem with backward Euler method:
\begin{align}
\nonumber&\dfrac{1}{\Delta t}({u}^{k+1}_{nl,\delta}(\xb)-{u}^{k}_{nl,\delta}(\xb))-\alpha_{nl}\mcL_{\delta} u^{k+1}_{nl,\delta}(\xb)=f_{nl}(\xb,t^{k+1}),\quad \mathbf{x}\in\Omega_{nl}\backslash B\Omega_{i}\\
\nonumber&\dfrac{1}{\Delta t}Q_\delta(\xb)({u}^{k+1}_{nl,\delta}(\xb)-{u}^{k}_{nl,\delta}(\xb))-\alpha_{nl}\mcL_{N\delta} u^{k+1}_{nl,\delta}(\xb)=Q_\delta(\xb)f_{nl}(\xb,t^{k+1})+\alpha_{nl}V_\delta (\xb)g(\overline{\mathbf{x}},t^{k+1}),\quad \mathbf{x}\in B\Omega_{i}\\
\nonumber&u^{k+1}_{nl,\delta}(\xb)=u^D_{nl}(\xb,t^{k+1}),\quad \mathbf{x}\in B\Omega_{d}\\
&u^0_{nl,\delta}(\xb)=u^{IC}(\xb).\quad \mathbf{x}\in\Omega_{nl}\label{eqn:nonlocaleqnBE}
\end{align}
the nonlocal maximum principle also holds true:
\begin{lem}\label{eqn:maxprincipleBE}
{ For a sequence of semi-discretized solutions $\{u^k(\xb)\}$, $k=0,1,\cdots,M$ where $M=\dfrac{T}{\Delta t}$, $u^k\in C(\overline{\Omega_{nl}})\cap C(B\Omega_{d}))$, $u(\xb,0)$ is bounded on $\xb\in\Omega_{nl}$ and $u^k(\xb)$ is bounded on $\xb\in B\Omega_{d}$, assuming that $u^k$ satisfies $\dfrac{1}{\Delta t}({u}^{k+1}-u^k)-\alpha_{nl}L_\delta u^{k+1}\leq 0$ for all $x\in \Omega_{nl}\backslash B\Omega_{i}$ and $\dfrac{1}{\Delta t}Q_\delta ({u}^{k+1}-u^k)-\alpha_{nl}L_{N\delta} u^{k+1}\leq 0$ for all $x\in B\Omega_{i}$, we have
\begin{equation}\label{eqn:maxpBE}
 \max_{k=0}^{M}\sup_{\mathbf{x}\in \overline{\Omega_{nl}}\cup B\Omega_{d}} u^k(\mathbf{x})\leq \max\left\{\sup_{\mathbf{x}\in \overline{\Omega_{nl}}\cup B\Omega_{d}} u(\mathbf{x},0),\max_{k=0}^{M}\sup_{\mathbf{x}\in B\Omega_{d}} u^k(\mathbf{x})\right\}. 
\end{equation}
}
\end{lem}
\begin{proof}
The argument is similarly obtained as in the proof of Lemma \ref{eqn:maxprinciple}.
\end{proof}

We now assume that $u^k_{nl,\delta}$ is the solution of \eqref{eqn:nonlocaleqnBE} and $u^k_0$ is the solution of semi-discretized local problem \eqref{eqn:locallimit} with the backward Euler method. Denote $e^k_\delta(\mathbf{x}):={u}^k_{nl,\delta}(\mathbf{x})-u^k_0(\mathbf{x})$ and
\begin{align*}
&T^k_\delta(\mathbf{x}):=\alpha_{nl}(-\Delta u^k_0(\mathbf{x})+L_\delta u^k_0(\mathbf{x})), &\text{ for } \mathbf{x}\in\Omega_{nl}\backslash B\Omega_{i},\\ &T^k_\delta(\mathbf{x}):=\alpha_{nl}(-\Delta u^k_0(\mathbf{x},t)+L_{N\delta} u^k_0(\mathbf{x}))+(Q_\delta(\xb)-1)\left(f_{nl}(\mathbf{x},t^k)-\dfrac{1}{\Delta t}[{u}^k_0(\mathbf{x})-{u}^{k-1}_0(\mathbf{x})]\right), &\text{ for } \mathbf{x}\in B\Omega_{i}.
\end{align*}
Then for $\mathbf{x}\in\Omega_{nl}\backslash B\Omega_{i}$, 
$$\dfrac{1}{\Delta t}({e}^k_\delta(\mathbf{x})-{e}^{k-1}_\delta(\mathbf{x}))-\alpha_{nl}L_\delta e^k_\delta(\mathbf{x})=\alpha_{nl}(-\Delta u^k_0(\mathbf{x})+L_\delta u^k_0(\mathbf{x}))=T^k_\delta(\mathbf{x}),$$
and for $\mathbf{x}\in B\Omega_{i}$, 
$$\dfrac{1}{\Delta t}Q_\delta(\mathbf{x})({e}^k_\delta(\mathbf{x})-{e}^{k-1}_\delta(\mathbf{x}))-\alpha_{nl}L_{N\delta} e^k_\delta(\mathbf{x})=Q_\delta(\mathbf{x})f_{nl}(\xb,t^k)-\dfrac{1}{\Delta t}Q_\delta(\mathbf{x})({u}^k_0(\mathbf{x})-{u}^{k-1}_0(\mathbf{x}))+\alpha_{nl}L_{N\delta} u^k_0(\mathbf{x})=T^k_\delta(\mathbf{x}).$$
In the following we take a specific kernel $J_{\delta}(s)=J^1_{\delta}(s)=\dfrac{4}{\pi\delta^4}$ for $s\leq\delta$ for simplicity. With Taylor expansion we can obtain the following truncation estimate for $T_\delta$:
\begin{lem}\label{lem:T}
Suppose $u^k_0$ is the solution to semi-discretized local problem \eqref{eqn:locallimit}, then
\begin{align*}
 |T^k_\delta(\mathbf{x})|\leq &C(T)(\delta^2),\qquad \text{for }\mathbf{x}\in\Omega_{nl}\backslash B\Omega_{i},\\
 |T^k_\delta(\mathbf{x})|\ \leq &C(T) [\delta - s_\xb]^{3/2}\delta^{-1/2}+\mcO(\delta^2),\qquad \text{for }\mathbf{x}\in B\Omega_{i},
\end{align*}
where $C(T)$ is independent of $\delta$ but might depend on $T$.
\end{lem}
\begin{proof}
The proof is based on the Taylor expansion of $u_0$ and an estimate for the asymmetric part in $A_\delta$, similar as in \cite[Lemma~4.2]{You2018}.
\end{proof}

Furthermore, with the maximum principle in Lemma \ref{eqn:maxprincipleBE}, when $f_{nl}$ and $u^D_{nl}$ are both continuous we have the following lemma. %\DK{[What is $G$ in this lemma?  Just any positive function?]}\YY{[In this Lemma $G(\xb)$ can be anything. However, when we use this Lemma in the proof of Theorem 1, we take $G(\xb)=L_\delta \phi$.]}
\begin{lem}\label{thm:phi}
Suppose that a nonnegative continuous function $\phi(\mathbf{x})$ is defined on $\overline{\Omega_{nl}}$, and $L_\delta\phi\geq G(\mathbf{x})>0$ %
for $\mathbf{x}\in \Omega_{nl}\backslash B\Omega_{i}$, $L_{N\delta}\phi\geq G(\mathbf{x})>0$ for $\mathbf{x}\in B\Omega_{i}$. Then
\begin{equation}
\max_{k=0}^{M}\sup_{\mathbf{x}\in \overline{\Omega_{nl}}}|e^k_\delta(\mathbf{x})|\leq \sup_{\mathbf{x}\in \overline{\Omega_{nl}}\cup B\Omega_d}\phi(\mathbf{x}) \max_{k=0}^{M}\sup_{\mathbf{x}\in \overline{\Omega_{nl}}}\dfrac{|T^k_\delta(\mathbf{x})|}{G(\mathbf{x})}.
\end{equation}
\end{lem}
\begin{proof}
The proof is obtained with the maximum principle: Let $K=\underset{k=0}{\overset{M}\max}\underset{\mathbf{x}\in \overline{\Omega_{nl}}}\sup \dfrac{|T^k_\delta(\mathbf{x})|}{G(\mathbf{x})}$, then for $K \phi(\mathbf{x})+e^k_\delta(\mathbf{x})$ %
we have: 
\begin{align*}
 \nonumber \dfrac{1}{\Delta t}(e^k_\delta(\mathbf{x})-e^{k-1}_\delta(\mathbf{x})) - \alpha_{nl} L_\delta (K \phi(\mathbf{x})+e^k_\delta(\mathbf{x}))=&\alpha_{nl}\left[-\max_{k=0}^{M}\sup_{\mathbf{x}\in \overline{\Omega_{nl}}}\dfrac{|T^k_\delta(\mathbf{x})|}{G(\mathbf{x})}L_\delta\phi(\mathbf{x})+T^k_\delta(\mathbf{x})\right] \leq 0
\end{align*}
for $\mathbf{x}\in \Omega_{nl}\backslash B\Omega_i$. A similar argument holds for $\mathbf{x}\in B\Omega_i$. With the maximum principle in Lemma \ref{eqn:maxprincipleBE} we have
\begin{align*}
 \nonumber\max_{k=0}^{M}\sup_{\mathbf{x}\in \overline{\Omega_{nl}}} e^k_\delta(\mathbf{x})\leq& \max_{k=0}^{M}\sup_{\mathbf{x}\in \overline{\Omega_{nl}}}(K \phi(\mathbf{x})+e^k_\delta(\mathbf{x}))
 \leq K \max_{k=0}^{M}\sup_{\mathbf{x}\in \overline{\Omega_{nl}}\cup B\Omega_d} \phi(\mathbf{x}).
\end{align*}
Similarly, we have $-\dfrac{1}{\Delta t}(e^k_\delta(\mathbf{x})-e^{k-1}_\delta(\mathbf{x}))-\alpha_{nl}L_\delta(K\phi(\mathbf{x})-e^k_\delta(\mathbf{x})) \leq0$ for $\mathbf{x}\in \Omega_{nl}\backslash B\Omega_{i}$ and %
$-\dfrac{1}{\Delta t}(e^k_\delta(\mathbf{x})-e^{k-1}_\delta(\mathbf{x}))-\alpha_{nl}L_{N\delta}(K\phi(\mathbf{x})-e^k_\delta(\mathbf{x})) \leq0$ for $\mathbf{x}\in B\Omega_{i}$, hence
\begin{align*}
 \nonumber\max_{k=0}^{M}\sup_{\mathbf{x}\in \overline{\Omega_{nl}}} (-e^k_\delta(\mathbf{x}))\leq& \max_{k=0}^{M}\sup_{\mathbf{x}\in \overline{\Omega_{nl}}}(K \phi(\mathbf{x})-e^k_\delta(\mathbf{x}))
 \leq K \max_{k=0}^{M}\sup_{\mathbf{x}\in \overline{\Omega_{nl}}\cup B\Omega_d} \phi(\mathbf{x}).
\end{align*}
\end{proof}

With the above lemma and assuming that the datum has sufficient H\"{o}lder continuity, we obtain the following main theorem:
\begin{theorem}\label{thm:conv}
Suppose $f_{nl}\in C^{1+\epsilon/2,\epsilon}((0,\infty)\times \mathbb{R}^2)$, $u^D_{nl}\in C^{2+\epsilon/2,2+\epsilon}((0,\infty)\times (\overline{\Omega_{nl}}\cup B\Omega_d))$, $\dfrac{\partial u^D_{nl}}{\partial \mathbf{n}}=0$ on $\Gamma_i$, $u^{IC}\in C^{2+\epsilon}(\mathbb{R}^2)$, and $\epsilon\in(0,1)$, $\{{u}^k_{nl,\delta}(\xb)\}$ are the semi-discretized results from the backward Euler method to the nonlocal problem.  Then for sufficiently small $\delta$, there exists a constant $C$ independent of $\delta$ such that
 \begin{equation}\label{eqn:tbound}
  \sup_{\mathbf{x}\in\Omega_{nl}}|u_{nl,\delta}^M(\mathbf{x})-u_0(\mathbf{x},T)|\leq C(T)(\Delta t+\delta^2),
 \end{equation}
 where $M=\dfrac{T}{\Delta t}$.
\end{theorem}
%\begin{theorem}
% Suppose $f_{nl}$ is continuous, ${u}_\delta$ solves the nonlocal problem \eqref{eqn:nonlocaleqn} and $u_0$ is the solution to the corresponding local problem \eqref{eqn:locallimit}, then for sufficiently small $\delta$ there exists a constant $C$ independent of $\delta$ such that
% \begin{equation}\label{eqn:s4bound}
%  \sup_{\Omega_{nl} \times [0,T]}|u_\delta(\mathbf{x},t)-u_0(\mathbf{x},t)|\leq C(T)\delta^2.
% \end{equation}
%\end{theorem}
\begin{proof}
With the regularity of the given datum and the domain, we have $u_0(\xb,t)\in C^{2+\epsilon/2,2+\epsilon}((0,\infty)\times\Omega_{nl})$ (see, e.g., \cite[Theorem~10.4.1]{krylov1996lectures}) and therefore $\underset{\xb\in\Omega_{nl}}\sup |u_0^M(\mathbf{x})-u_0(\mathbf{x},T)|\leq C(T)\Delta t$. Since
$|u_{nl,\delta}^M(\mathbf{x})-u_0(\mathbf{x},T)|\leq |u_{nl,\delta}^M(\mathbf{x})-u^M_0(\mathbf{x})|+|u_0^M(\mathbf{x})-u_0(\mathbf{x},T)|$, it suffices to show that $\underset{\xb\in\Omega_{nl}}\sup |u_{nl,\delta}^M(\mathbf{x})-u^M_0(\mathbf{x})|\leq C(T)\delta^2$. As shown in the left plot of Figure \ref{graph}, let $\mathbf{z}_3\in \partial\Omega$ be a point such that $\tau(\mathbf{z}_3)$ is orthogonal to the bisector of the angle $\angle \mathbf{z}_2 \tilde{\mathbf{z}}\mathbf{z}_1$. Set the barrier function as
\begin{equation}\label{eqn:phi} 
\phi(\mathbf{x}):=|\text{dist}(\mathbf{x},\tau(\mathbf{z}_3))+1|^2.
\end{equation}
Then it can be shown that
\begin{align*}
L_{\delta}\phi(\mathbf{x})&\ge  C, &\text{ for }\xb\in\Omega_{nl}\backslash B\Omega_i\\
L_{N\delta}\phi(\mathbf{x})&\ge  C[\delta-s_\xb]^{3/2}\delta^{-5/2} +C_1>0. &\text{ for }\xb\in B\Omega_i
\end{align*}
Taking $G(\mathbf{x})=L_\delta\phi$ for $\mathbf{x}\in \Omega_{nl}\backslash B\Omega_{i}$, $G(\mathbf{x})=L_{N\delta}\phi$ for $\mathbf{x}\in B\Omega_{i}$ in Lemma \ref{thm:phi}, combining the above bounds with the truncation bounds $|T_\delta|$ provided in Lemma \ref{lem:T} we finish the proof. For further details on the bounds, we refer the interested readers to \cite{You2018}.
\end{proof}

\begin{remark}
With the maximum principle \ref{eqn:maxprinciple},  assuming that the nonlocal solution $u_{nl,\delta}(\xb,t)$ has sufficient regularity and employing a $q^{\text{th}}$-order temporal discretization method to the nonlocal problem, then for sufficiently small $\delta$, there exists a constant $C$ independent of $\delta$ such that
 \begin{equation*}
  \sup_{\mathbf{x}\in\Omega_{nl}}|u_{nl,\delta}^M(\mathbf{x})-u_0(\mathbf{x},T)|\leq C(T)(\Delta t^q+\delta^2).
 \end{equation*}
\end{remark}

\begin{remark}
The convergence rate in Theorem \ref{thm:conv} is optimal considering the $\mcO(\delta^2)$ convergence of the nonlocal equation to its local limit away from the boundary.
\end{remark}

\subsection{A Nonlocal Robin-Type and Corner Boundary Condition}\label{sec:robinmath}

Based on the Neumann-type constraint problem, we now develop the nonlocal analog to the classical Robin boundary condition $\beta u(\mathbf{x})+\dfrac{\partial u}{\partial \mathbf{n}}(\mathbf{x})=g(\mathbf{x})$ with $\beta\neq 0$ on a sharp interface $\Gamma_i$. Specifically, we propose the nonlocal Robin-type boundary condition with a modified formulation in $B\Omega_i$: %\DK{[As in the Neumann case, this seems to rely on un-stated regularity assumptions on $u_{nl,\delta}$; suppose we assume $u_{nl,\delta}$ is in the nonlocal energy space, and that space is equivalent to $L^2$.  As in comments earlier, the values of $u_{nl,\delta}$ could be altered arbitrarily on the zero-measure set $\Gamma_i$ while staying in the same $L^2$ element equivalence class, but this would change the residual of the formulation over a volume with nonzero measure, because of the projection from $\mathbf{x}$ to $\overline{\mathbf{x}}$ (and the line integral issue before)].}\YY{[Unlike the Neumann case, here we realy have the regularity issue... And this is the well-posedness issue I mentioned in our email. For the Robin-type BC when apply a test function $v$ we got
\begin{align}
\nonumber&Q_\delta(\xb)\dot{u}_{nl,\delta}(\xb,t)-2\alpha_{nl}\int_{\Omega_{nl}} J_{\delta}(|\mathbf{x}-\mathbf{y}|)(u_{nl,\delta}(\mathbf{y},t)-u_{nl,\delta}(\mathbf{x},t))d\mathbf{y}+\alpha_{nl}\beta V_\delta(\xb)u_{nl,\delta}(\overline{\mathbf{x}},t)\\
&-2\alpha_{nl}M_\delta(\mathbf{x}) \int_{-\delta}^{\delta} H_{\delta}(|l|) (u_{nl,\delta}(\mathbf{x}_{l},t)-u_{nl,\delta}(\mathbf{x},t)) d\mathbf{x}_{l}=Q_\delta(\xb)f_{nl}({{\mathbf{x}}},t)+\alpha_{nl}V_\delta(\xb)g(\overline{\mathbf{x}},t),\label{eqn:formula4}
\end{align}
where $Q_\delta(\xb)$, $V_\delta(\xb)$, and $M_\delta(\mathbf{x})$ are as defined in \eqref{eqn:Q}-\eqref{eqn:M}. 
We then obtain the following nonlocal constraint problem
\begin{align}
\nonumber&\dot{u}_{nl,\delta}(\xb,t)-\alpha_{nl}\mcL_{\delta} u_{nl,\delta}(\xb,t)=f_{nl}(\xb,t),\quad \mathbf{x}\in\Omega_{nl}\backslash B\Omega_{i}\\
\nonumber&Q_\delta(\xb)\dot{u}_{nl,\delta}(\xb,t)-\alpha_{nl}\mcL_{N\delta} u_{nl,\delta}(\xb,t)+\alpha_{nl}\beta V_\delta (\xb)u_{nl,\delta}(\overline{\mathbf{x}},t)=Q_\delta(\xb)f_{nl}(\xb,t)+\alpha_{nl}V_\delta (\xb)g(\overline{\mathbf{x}},t),\quad \mathbf{x}\in B\Omega_{i}\\
\nonumber&u_{nl,\delta}(\xb,t)=u^D_{nl}(\xb,t),\quad \mathbf{x}\in B\Omega_{d}\\
&u_{nl,\delta}(\xb,0)=u^{IC}(\xb).\quad \mathbf{x}\in\Omega_{nl}\label{eqn:nonlocaleqnRobin}
\end{align}
Employing the backward Euler scheme for time integration and the meshfree quadrature rule described in Section \ref{sec:nl_model}, for $\xb_i\in B\Omega_i$, we solve for $(U_{\delta})^{k}_i\approx u_{nl,\delta}(\xb_i,t^k)$ with:
\begin{align}%\label{eqn:nonlocaldxdt}
\nonumber&\dfrac{Q_\delta(\xb_i)}{\Delta t}((U_{\delta})_i^{k+1}-(U_{\delta})_i^k)-2\alpha_{nl}\tilde{u}_\delta^{k+1}DP(P^{\mathrm{T}}DP)^{-1}\int_{B(\mathbf{x}_i,\delta)\cap\Omega_{nl}} J_{\delta}(|\mathbf{y}-\mathbf{x}_i|)(R(\mathbf{y})-R(\mathbf{x}_i))d\mathbf{y}\\ \nonumber&-2 \alpha_{nl}\tilde{u}_\delta^{k+1}DP(P^{\mathrm{T}}DP)^{-1} M_\delta(\mathbf{x}_i)\int_{-\delta}^{\delta} H_{\delta}(|l|) (R(\mathbf{x}_{l})-R(\mathbf{x}_i)) d\mathbf{x}_{l}+\alpha_{nl}\beta V_\delta(\xb_i)(U_{\delta})_{\overline{i}}^{k+1}\\
&= Q_\delta(\xb_i)f_{nl}(\xb_i,t^{k+1})+\alpha_{nl}V_\delta(\xb_i)g(\overline{\mathbf{x}_i},t^{k+1}),\label{eqn:nonlocaldxdtbc}
\end{align}
with $\tilde{u}^k_{\delta}=((U_\delta)_j^k:j\in I(\mathbf{x}_i))^{\mathrm{T}} \in \mathbb{R}^{\#I(\mathbf{x}_i)}$. $(U_{\delta})_{\overline{i}}^{k}$ represents the solution corresponding to $\overline{\xb}_i$ which may be estimated based on the generalized moving least squares (GMLS) approximation framework if $\overline{\xb}_i$ is not in the collection of grid points $\chi_{h}$.

%\DK{[Maybe this could be an explanation of the regularity issue:
\begin{remark}
The statement of the Robin problem requires solution regularity beyond the $L^2$-equivalent nonlocal energy space introduced for the Neumann problem, due to the evaluation of $u_{nl,\delta}$ at $\overline{\xb}$.  However, the proposed Robin problem is only intended for use in the spatially-discretized setting, where this extra regularity is available.  In this work, we consider asymptotically-compatible discretizations, and the continuous problem is in fact the local heat equation, although, more generally, one might incorporate additional phenomena (e.g., bond damage in nonlocal elasticity) such that the vanishing-horizon limit of the nonlocal problem does not correspond to a local problem.
\end{remark}
%]}\YY{[sounds a good idea to me! Thank you!]}

 \begin{figure}
 \centering
 \includegraphics[scale=0.6]{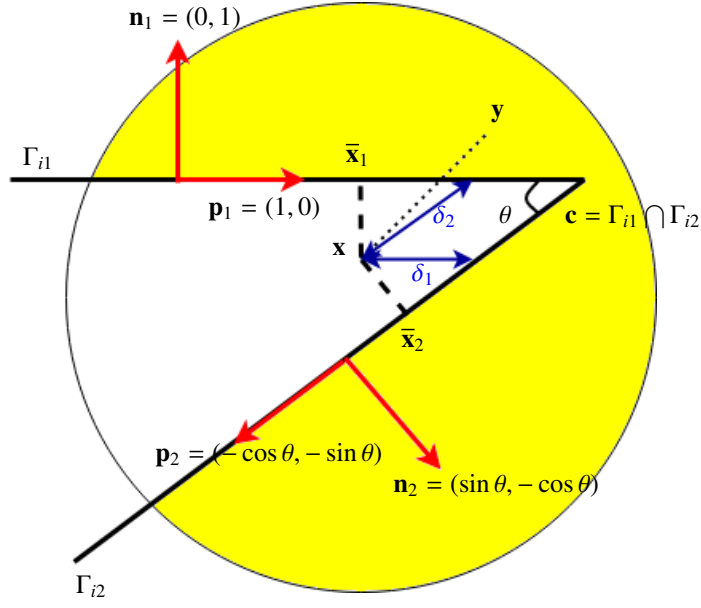}
 \put(-123,128){$\mathbf{x}$}\put(-118,163){$\overline{\mathbf{x}}_1$}\put(-97,93){$\overline{\mathbf{x}}_2$}
 \put(-100,38){$\mathbf{n}_2=(\sin\theta,-\cos\theta)$}\put(-200,215){$\mathbf{n}_1=(0,1)$}
 \put(-190,50){$\mathbf{p}_2=(-\cos\theta,-\sin\theta)$}\put(-170,143){$\mathbf{p}_1=(1,0)$}
 \put(-35,140){$\mathbf{c}=\Gamma_{i1}\bigcap \Gamma_{i2}$}
 \put(-240,162){$\Gamma_{i1}$}\put(-220,0){$\Gamma_{i2}$}
 \put(-60,140){$\theta$}\put(-63,180){$\mathbf{y}$}
 \put(-85,140){{\color{blue}$\delta_2$}}\put(-93,118){{\color{blue}$\delta_1$}}
 \caption{Geometric assumptions and notation for the corner case. Here the yellow region denotes $B(\mathbf{x},\delta)\cap (\mathbb{R}^2\backslash\Omega_{nl})$.}
\label{fig:corner}
\end{figure}

Similar as in \cite{You2018}, to investigate how the new Robin-type constraint formulation extrapolates to the setting of Lipschitz domains, we further extend the proposed formulation to boundary $\Gamma_i$ with corners. As shown in Figure \ref{fig:corner}, here we assume that there are two intersecting boundaries with Robin-type boundary conditions:
\begin{align}
 \beta u+\dfrac{\partial u}{\partial \mathbf{n}_1}=g_1,\quad&\text{ on }\Gamma_{i1},\\
 \beta u+\dfrac{\partial u}{\partial \mathbf{n}_2}=g_2,\quad&\text{ on }\Gamma_{i2},
\end{align}
and the two boundaries intersect at $\mathbf{c}=\Gamma_{i1}\cap \Gamma_{i2}$. For any point $\mathbf{x}$ satisfying $|\mathbf{x}-\mathbf{c}|<\delta$, we project $\mathbf{x}$ onto the two boundaries respectively, i.e., $\mathbf{x}=\overline{\mathbf{x}}_1-s_{\mathbf{x}_1}\mathbf{n}_1(\overline{\mathbf{x}}_1)=\overline{\mathbf{x}}_2-s_{\mathbf{x}_2}\mathbf{n}_2(\overline{\mathbf{x}}_2)$. Here we assume that both $\Gamma_{i1}$ and $\Gamma_{i2}$ are straight lines near the corner $\mathbf{c}$, although the formulation can be further extended to more general cases. Denote $\theta$ as the angle between $\Gamma_{i1}$ and $\Gamma_{i2}$, without loss of generality we further denote $\mathbf{n}_1 = (0,1)$ and $\mathbf{n}_2 = (\sin \theta,-\cos \theta)$. Correspondingly, we have $\mathbf{p}_1 = (1,0)$ and $\mathbf{p}_2 = (-\cos \theta,-\sin \theta)$. For each point $\mathbf{x}= (x_1,x_2)$, with Taylor expansion we have the following approximation for $u(\mathbf{y})-u(\mathbf{x})$ with $\mathbf{y} = (y_1,y_2)\in B(\mathbf{x},\delta)\cap \partial \Omega_{N\delta}$:
\begin{align*}
 \nonumber &u(\mathbf{y},t)-u(\mathbf{x},t)\\
 =&d_1\dfrac{\partial u(\mathbf{x},t)}{\partial \mathbf{n}_1}+d_2\dfrac{\partial u(\mathbf{x},t)}{\partial \mathbf{n}_2}+\dfrac{1}{2}d_1^2[u(\mathbf{x},t)]_{n_1n_1}+\dfrac{1}{2}d_2^2[u(\mathbf{x},t)]_{n_2n_2}+d_1d_2[u(\mathbf{x},t)]_{n_1n_2}+\mcO(\delta^3)\\
 =&d_1[g_1(\overline{\mathbf{x}}_1,t)-\beta u(\overline{\mathbf{x}}_1,t)]%
 +d_2[g_2(\overline{\mathbf{x}}_2,t)-\beta u(\overline{\mathbf{x}}_2,t)]+\dfrac{1}{\alpha_{nl}}\left(\dfrac{1}{2}d_1^2-%
 (\overline{\mathbf{x}}_1-\mathbf{x})\cdot \mathbf{n}_1 d_1\right)\left(-f_{nl}(\mathbf{x},t)-\alpha_{nl}[u(\mathbf{x},t)]_{p_1p_1}+\dot{u}(\mathbf{x},t)\right)\\
 \nonumber&+\dfrac{1}{\alpha_{nl}}\left(\dfrac{1}{2}d_2^2-%
 (\overline{\mathbf{x}}_2-\mathbf{x})\cdot \mathbf{n}_2 d_2\right)\left(-f_{nl}(\mathbf{x},t)-\alpha_{nl}[u(\mathbf{x},t)]_{p_2p_2}+\dot{u}(\mathbf{x},t)\right)\\
 \nonumber&+\dfrac{1}{2\sin \theta}d_1d_2 \left(\dfrac{\partial g_1(\overline{\mathbf{x}}_1,t)}{\partial \mathbf{p}_1}-\beta \dfrac{\partial u(\overline{\mathbf{x}}_1,t)}{\partial \mathbf{p}_1}-\dfrac{\partial g_2(\overline{\mathbf{x}}_2,t)}{\partial \mathbf{p}_2}+\beta\dfrac{\partial u(\overline{\mathbf{x}}_2,t)}{\partial \mathbf{p}_2}+\dfrac{1}{\alpha_{nl}}[f_{nl}(\mathbf{x},t)-\dot{u}(\mathbf{x},t)]\sin \theta \cos \theta \right)+\mcO(\delta^3),
 \end{align*}
where $d_1 := \dfrac{\cos \theta}{\sin \theta}(y_1-x_1) + (y_2-x_2)$, $d_2 := \dfrac{1}{\sin \theta}(y_1-x_1)$. Moreover, we have
\begin{align*}
&[u(\mathbf{x},t)]_{p_1p_1}+[u(\mathbf{x},t)]_{p_2p_2}=\dfrac{1}{\alpha_{nl}}[-f_{nl}(\mathbf{x},t)+\dot{u}(\mathbf{x},t)]+\cot\theta \left[\dfrac{\partial g_1(\overline{\mathbf{x}}_1,t)}{\partial \mathbf{p}_1}-\beta\dfrac{\partial u(\overline{\mathbf{x}}_1,t)}{\partial \mathbf{p}_1}-\dfrac{\partial g_2(\overline{\mathbf{x}}_2,t)}{\partial \mathbf{p}_2}+\beta\dfrac{\partial u(\overline{\mathbf{x}}_2,t)}{\partial \mathbf{p}_2}\right]+O (\delta),\\
&\dfrac{\partial u(\overline{\mathbf{x}}_1,t)}{\partial \mathbf{p}_1}=\cot\theta g_1(\overline{\mathbf{x}}_1,t)+\dfrac{1}{\sin\theta}g_2(\overline{\mathbf{x}}_2,t)+\mcO(\delta),\qquad \dfrac{\partial u(\overline{\mathbf{x}}_2,t)}{\partial \mathbf{p}_2}=-\dfrac{1}{\sin\theta} g_1(\overline{\mathbf{x}}_1,t)-\cot\theta g_2(\overline{\mathbf{x}}_2,t)+\mcO(\delta).
\end{align*}
Let 
\[D_1 = 2\int_{\mathbb{R}^2\backslash\Omega_{nl}}J_\delta(|\mathbf{x}-\mathbf{y}|)\left[\dfrac{1}{2}d_1^2-(\overline{\mathbf{x}}_1-\mathbf{x})\cdot \mathbf{n}_1 d_1\right]d \mathbf{y},\qquad D_2 = 2\int_{\mathbb{R}^2\backslash\Omega_{nl}}J_\delta(|\mathbf{x}-\mathbf{y}|)\left[\dfrac{1}{2}d_2^2-(\overline{\mathbf{x}}_2-\mathbf{x})\cdot \mathbf{n}_2 d_2\right] d\mathbf{y},\]
%we then substitute the above approximations into the nonlocal formulation and neglecting the higher order terms give the algorithm. 
for $D_1>D_2$, we take $\delta_1$ as the arc length from $\mathbf{x}$ to $\Gamma_{i}$ following the contour parallel to $\Gamma_{i1}$ and use $2\int_{-\delta_1}^{\delta_1} H_{\delta_1}(|l|)(u(\mathbf{x}_{l1},t)-u(\mathbf{x},t))d\mathbf{x}_{l1}$ to denote the integral on this contour which approximates $[u(\mathbf{x},t)]_{p_1p_1}$. We obtain the following formulation for $\xb\in B(\mathbf{c},\delta)\cap \Omega_{nl}$:
\begin{align}
 \nonumber&Q^c_\delta(\xb)\dot{u}_{nl,\delta}(\xb,t)-2\alpha_{nl}\int_{\Omega_{nl}} J_{\delta}(|\mathbf{x}-\mathbf{y}|)(u_{nl,\delta}(\mathbf{y},t)-u_{nl,\delta}(\mathbf{x},t))d\mathbf{y}+4\alpha_{nl}(D_1-D_2)\int_{-\delta_1}^{\delta_1} H_{\delta_1}(|l|)(u_{nl,\delta}(\mathbf{x}_{l1},t)-u_{nl,\delta}(\mathbf{x},t))d\mathbf{x}_{l1}\\
 \nonumber&+2\alpha_{nl}\beta\int_{\mathbb{R}^2\backslash\Omega_{nl}} J_{\delta}(|\mathbf{x}-\mathbf{y}|)
 \bigg( d_1u_{nl,\delta}(\overline{\mathbf{x}}_1,t) +d_2u_{nl,\delta}(\overline{\mathbf{x}}_2,t)\bigg)d\mathbf{y}\\
 \nonumber=&Q^c_\delta(\xb)f(\mathbf{x},t)-\alpha_{nl}D_2\cot \theta \left(\dfrac{\partial g_1(\overline{\mathbf{x}}_1,t)}{\partial \mathbf{p}_1}-\dfrac{\partial g_2(\overline{\mathbf{x}}_2,t)}{\partial \mathbf{p}_2}\right)+\beta\alpha_{nl}D_2\cot\theta\left(\cot\theta+\dfrac{1}{\sin\theta}\right)(g_1(\overline{\mathbf{x}}_1,t)+g_2(\overline{\mathbf{x}}_2,t))\\
\nonumber&+2\alpha_{nl}\int_{\mathbb{R}^2\backslash\Omega_{nl}} J_{\delta}(|\mathbf{x}-\mathbf{y}|)
 \bigg( d_1g_1(\overline{\mathbf{x}}_1,t) +d_2g_2(\overline{\mathbf{x}}_2,t)+\dfrac{d_1d_2}{2\sin\theta} \left(\dfrac{\partial g_1(\overline{\mathbf{x}}_1,t)}{\partial \mathbf{p}_1}-\dfrac{\partial g_2(\overline{\mathbf{x}}_2,t)}{\partial \mathbf{p}_2}\right)\\
 &-\beta\dfrac{d_1d_2}{2\sin\theta}\cot\theta\left(\cot\theta+\dfrac{1}{\sin\theta}\right)(g_1(\overline{\mathbf{x}}_1,t)+g_2(\overline{\mathbf{x}}_2,t))\bigg) d\mathbf{y}\label{eqn:cornerf1}
\end{align}
where
\[Q^c_\delta(\xb)=1-D_1+\int_{\mathbb{R}^2\backslash\Omega_{nl}} J_{\delta}(|\mathbf{x}-\mathbf{y}|)d_1d_2\cos\theta d\mathbf{y}.\]
Else, we take $\delta_2$ as the arc length from $\mathbf{x}$ to $\Gamma_{i}$ following the contour parallel to $\Gamma_{i2}$ and use $2\int_{-\delta_2}^{\delta_2} H_{\delta_2}(|l|)(u(\mathbf{x}_{l2},t)-u(\mathbf{x},t))d\mathbf{x}_{l2}$ to denote the integral on this contour which approximates $[u(\mathbf{x},t)]_{p_2p_2}$. A similar formulation is obtained.

%{\color{blue}[HY]: please add a remark regarding convex/concave corners here.}
\begin{remark}
When the domain is concave and $\theta > \pi$ on the corner, it is possible that the projection points $\overline{\mathbf{x}}_1$ and $\overline{\mathbf{x}}_2$ are on the extended lines of $\Gamma_{i1}$ and $\Gamma_{i2}$. In this case, we project $\mathbf{x}$
onto the corner point $\mathbf{c}$ and evaluate $g_1$, $g_2$ on $\mathbf{c}$. The derivation is very similar as above. %Without loss of generality, we assume only $\overline{\mathbf{x}}_1$ is not on the boundary, then with Taylor expansion, we have 
\end{remark}

\subsection{Numerical Results for Nonlocal Boundary Conditions}\label{sec:Robintest}

\begin{figure}
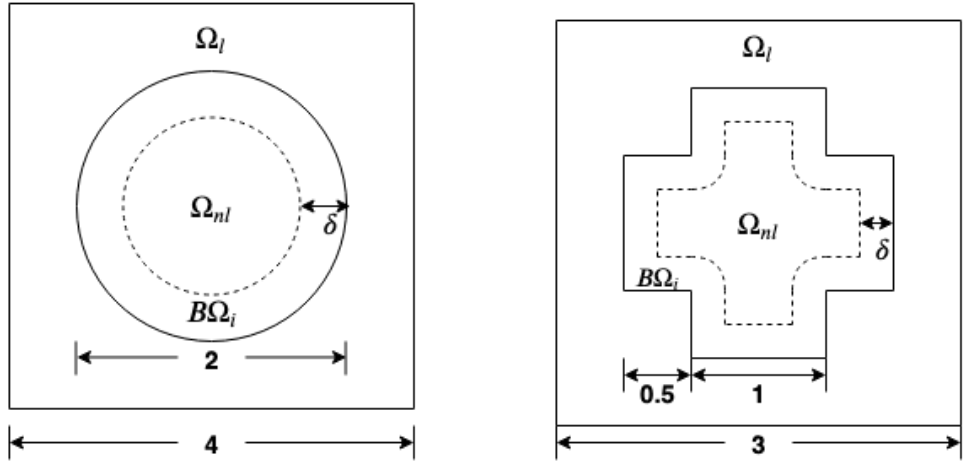

\centering
\subfigure{\includegraphics[width=0.35\textwidth]{./circlecouple_domain.png}}$\qquad\qquad$
\subfigure{\includegraphics[width=0.35\textwidth]{./crosscouple_domain.png}}
\caption{Illustration of the domains employed in numerical tests. Left: problem with a circular nonlocal subdomain and a surrounding local subdomain, corresponding to test 2 in Sections \ref{sec:test2} and \ref{sec:test2couple}. Right: test with a cross-shape nonlocal subdomain and a surrounding local subdomain, corresponding to test 3 in Sections \ref{sec:test2} and \ref{sec:test2couple}.}
\label{fig:couplingdomain}
\end{figure}

In this Section, we present numerical tests of the proposed nonlocal boundary treatment on $\Gamma_i$, by considering three types of representative domains $\Omega_{nl}$: a square domain in Section \ref{sec:test1} which represents the case with $0$ curvature on $\Gamma_i$; a circular domain in Section \ref{sec:test2}, which represents a case with nonzero curvature on $\Gamma_i$; and a cross-shape domain in Section \ref{sec:test3} which is a non-convex domain with corners and therefore it is outside the scope of the model problem analysis presented earlier in Section \ref{sec:neumann}. Illustration of the square domain can be found in the left plot of Figure \ref{fig2}, while the circular domain and the cross-shape domain are shown as the $\Omega_{nl}$ region in Figure \ref{fig:couplingdomain}. With the tests we aim to investigate the performance of the proposed nonlocal Neumann and Robin-type constraint formulation on patch tests, and to demonstrate the asymptotic convergence of the meshfree quadrature rule \eqref{eqn:nonlocaldxdt}. To maintain an easily scalable implementation, it is often desirable that the ratio $C_1\leq \dfrac{\delta}{h}\leq C_2$ as $\delta \rightarrow 0$. This so-called ``M-convergence'' results in a sparse linear system with bounded bandwidth that may be solved efficiently with standard preconditioning techniques \cite{bobaru2009convergence}. Therefore, in the asymptotic compatibility tests we focus on the case with $\delta/h=C$. For simplicity, we set $\alpha_{nl}=1$ in this section. Although the discussions and the proposed formulations in this paper are not tied to a specific kernel, in numerical tests we demonstrate the numerical performances with $J_\delta(r)=J^1_\delta(r)$.

\subsubsection{Test 1: a square domain with a straight line boundary}\label{sec:test1}

We first consider the nonlocal heat problem when $\Gamma_i=\{(1,y)|y\in[0,1]\}$, $\Omega_{nl}=[0,1] \times [0,1]$. Dirichlet-type boundary condition are imposed on the other three sides of $\Omega_{nl}$ in a collar with width $\delta$. An illustration of the domain can be found in the left plot of Figure \ref{fig2}. 

We first demonstrate the asymptotic compatibility. In this test, we set the initial condition $u^{IC}=0$ and external loading $f^{nl}(x,y,t)=(2t+2t^2)\sin(x)\cos(y)$. On $B\Omega_d$, a Dirichlet-type boundary condition $u^D_{nl}(x,y,t) := t^2\sin(x)\cos(y)$ is applied, and a Robin-type boundary condition $g(x,y,t)=\beta t^2\sin(1)\cos(y)+\pi \cos(1)\cos(y)$ is applied on the sharp interface $\Gamma_i$. Here we note that when $\beta=0$, the Robin-type boundary condition is equivalent to the Neumann-type boundary condition. The local limit of this problem has an analytical solution $u_0(x,y,t) = t^2\sin(x)\cos(y)$. To investigate the asymptotic compatibility when $\delta/h=C$, we refine $\delta$ and $h$ simultaneously keeping the ratio $\delta/h = 3.9$. For time discretization, we integrate until {$T=1$} using the backward-Euler method and $\Delta t= 100h^2$. The convergence results are presented in Figure \ref{fig:NLline}, where we demonstrate the difference between the numerical nonlocal solution and the analytical local limit $||u^{M,h}_{nl,\delta}-u_0(\xb,T)||$. Three different sets of Robin coefficients are employed here: (1) $\beta=0$ which is equivalent to the Neumann-type boundary condition; (2) $\beta$ is a non-zero constant, and (3) $\beta=C/h$. Note that the case (3) is tested here since $\beta=C/h$ is the most robust Robin coefficient for local-to-nonlocal coupling framework, as will be further discussed in Section \ref{sec:couple}. It is observed from Figure \ref{fig:NLline} that the second-order convergence $\mcO(\delta^2)$ is achieved from all three sets of Robin coefficients, which therefore verifies the analysis in Section \ref{sec:neumann} for the Neumann-type boundary condition and demonstrates the asymptotic compatibility of the numerical solver. The results on cases (2) and (3) illustrate that the second order convergence $\mcO(\delta^2)$ is also achieved on the nonlocal problem with Robin-type boundary condition, which can be seen as a generalization of the nonlocal Neumann-type boundary condition.

Moreover, we investigate the linear patch test problem with analytical linear solution $u_{nl}=u_0=x$ and the quadratic patch test problem with analytical quadratic solution $u_{nl}=u_0=x^2$. In the absence of forcing terms and with consistent boundary conditions on $B\Omega_d$ and $\Gamma_i$, we investigate if the nonlocal Robin-type constraint problem returns the accurate analytical nonlocal solution. 
%Note that in these two cases, the analytical nonlocal solution coincides with the analytical local limit. 
The numerical results along the domain center line $y=1/2$ are reported in Figure \ref{fig:patchtest}. We observe that the numerical solution from the proposed Robin-type boundary condition passes both the linear and quadratic patch tests within machine precision accuracy 
%(i.e. as accurate as the machine used for the simulations can be) 
and for several values of $h$ and $\beta$.

\begin{figure}
    \centering
     \subfigure{\includegraphics[width=0.48\textwidth]{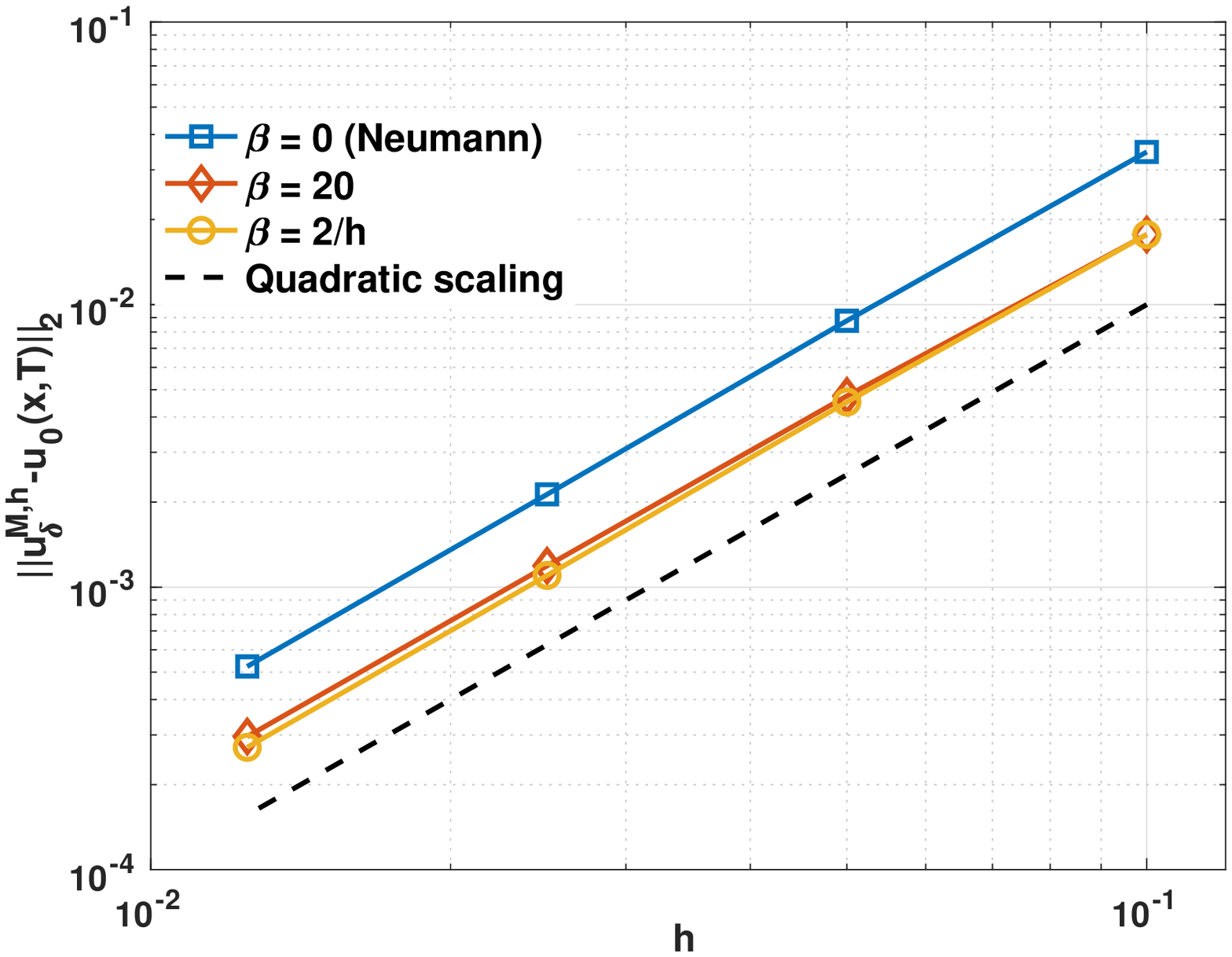}}
     \subfigure{\includegraphics[width = 0.48\textwidth]{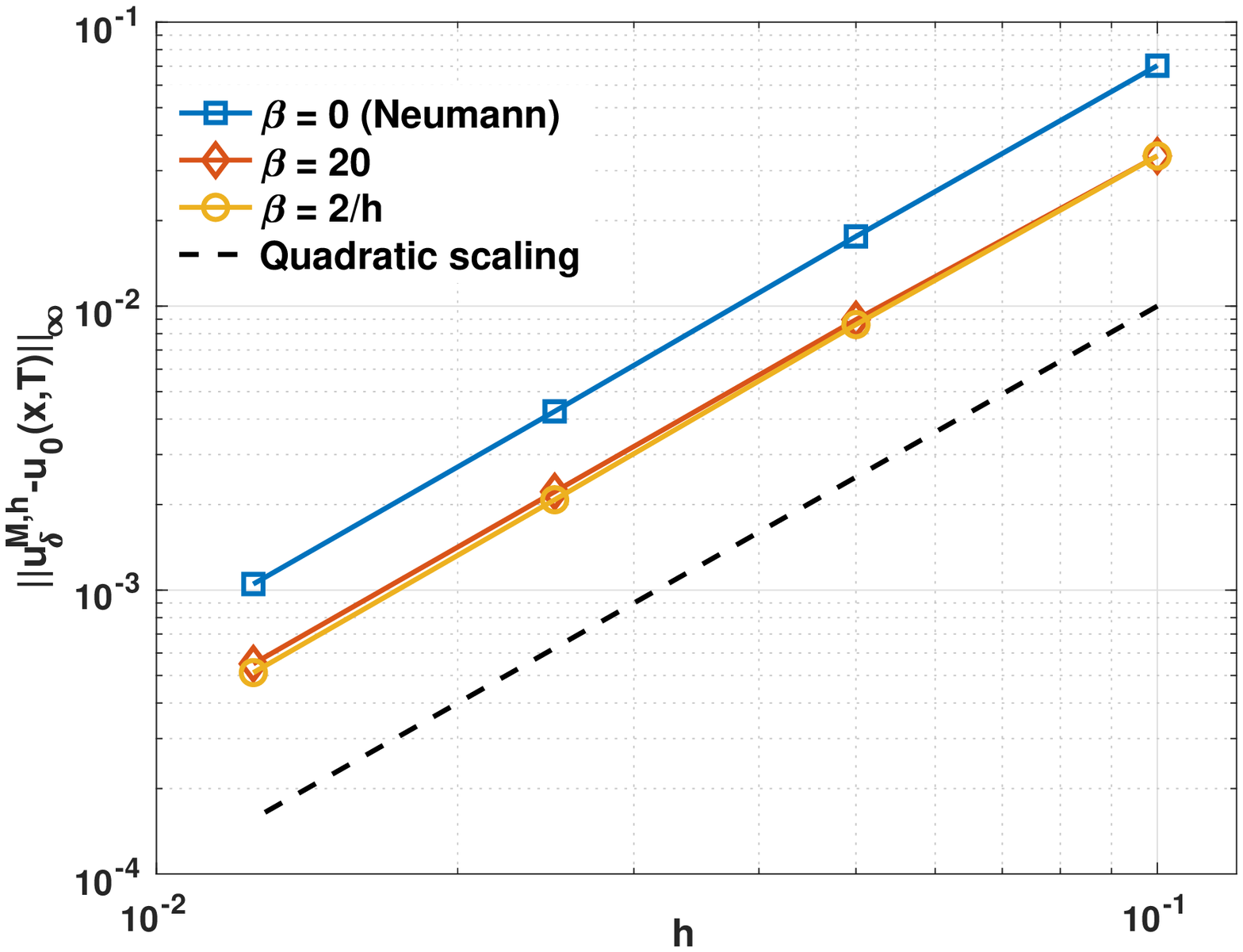}}
    \caption{Test 1 results for nonlocal boundary condition on a square domain. Left: convergence of the numerical nonlocal solution to the local limit with different Robin coefficients, in the $L^2$ norm. Right: convergence of the numerical nonlocal solution to the local limit with different Robin coefficients, in the $L^\infty$ norm.}
    \label{fig:NLline}
\end{figure}

\begin{figure}
\centering
 \subfigure{\includegraphics[width = 0.48\textwidth]{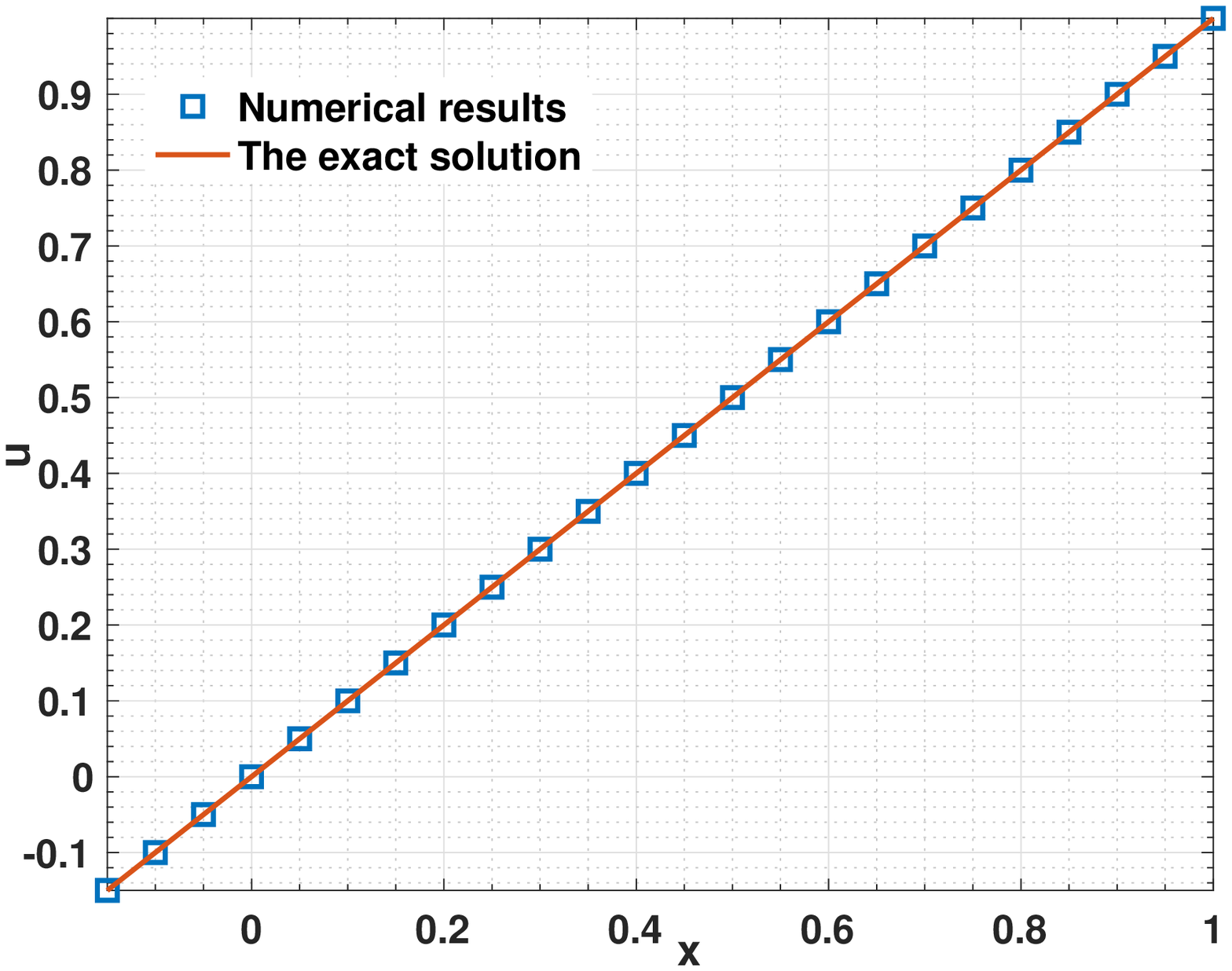}}
 \subfigure{\includegraphics[width =0.48 \textwidth]{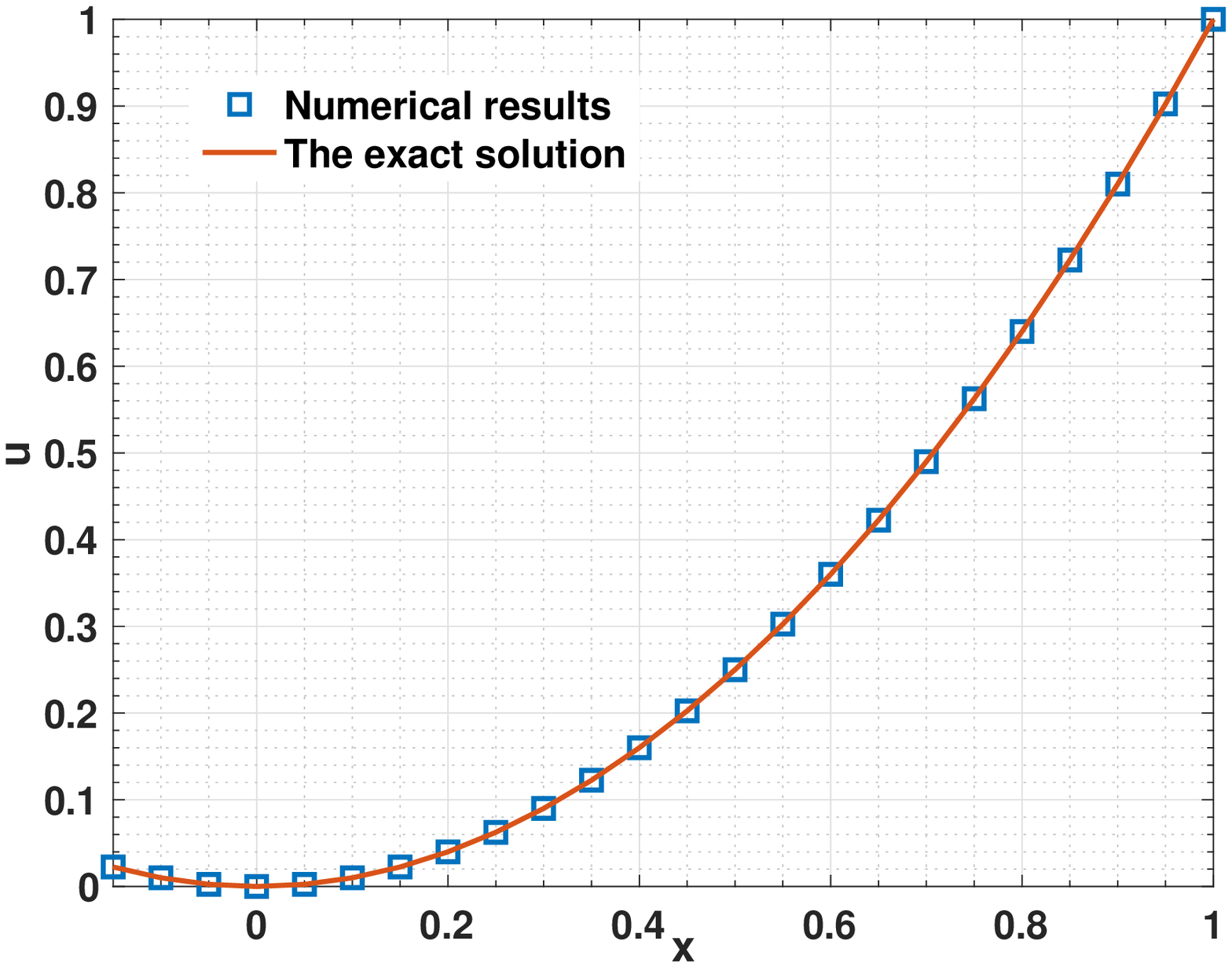}}
\caption{Test 1 results for nonlocal boundary condition on a square domain. Left: linear patch test. The comparison of numerical results with $h=1/20$, $\beta=10$ and the analytical solution $u=x$. Right: quadratic patch test. The comparison of numerical results with $h=1/20$, $\beta=10$ and the analytical  solution $u=x^2$.}
\label{fig:patchtest}
\end{figure}

\subsubsection{Test 2: a circular domain}\label{sec:test2}

%\begin{figure}
%    \centering
%    \begin{minipage}{.5\textwidth}
%    \includegraphics[width=\textwidth]{pics/circle_domain.png}
%    \end{minipage}%
%    \begin{minipage}{.5\textwidth}    \includegraphics[width = \textwidth]{pics/cross_domain.png}
%    \end{minipage}
%    \caption{Left: the circular domain with full Robin type boundary. Right: the cross-type domain with full Robin boundary.}
%    \label{fig:2domains}
%\end{figure}

We now consider a domain which has boundaries with non-zero curvature. As shown in the left side of Figure \ref{fig:couplingdomain}, we employ $\Omega_{nl} = \{(x,y)|x^2+y^2 \leq 1\}$ and $\Gamma_i=\partial\Omega_{nl}\backslash \{(0,-1)\}$, with a similar problem setting for initial condition and external loading as in test 1, namely, $u^{IC}=0$ and  $f^{nl}(x,y,t)=(2t+2t^2)\sin(x)\cos(y)$. A Robin-type boundary condition $g(x,y,t)=\beta t^2\sin(x)\cos(y)+t^2(x\cos(x)\cos(y)-y\sin(x)\sin(y))$ is applied on the sharp interface $\Gamma_i$. To make the problem well-posed  in the $\beta=0$ case, we set $u^D_{nl}(x,y,t) := t^2\sin(x)\cos(y)$ on $(x,y)=(0,-1)$. Similarly as in test 1, this problem setting has an analytical local limit $u_0(x,y,t) = t^2\sin(x)\cos(y)$. Considering $\delta/h = 3.9$, $\Delta t = 100h^2$ while decreasing the spatial discretization size $h$, the comparison of numerical nonlocal solution and the analytical local limit $u_0$ at {$T=1$} are presented in Figure \ref{fig:NLcircle}, again with three sets of Robin coefficients: (1) $\beta=0$; (2) $\beta$ is a non-zero constant; and (3) $\beta=C/h$. It can be observed that with all three sets of Robin-coefficients we have achieved $\mcO(\delta^2)$ convergence rate to the corresponding local limit. Therefore, the proposed Robin-type boundary formulation is asymptotically compatible with boundaries that have a non-zero curvature and $\delta/h=C$.

\begin{figure}
    \centering
    \subfigure{\includegraphics[width=0.48\textwidth]{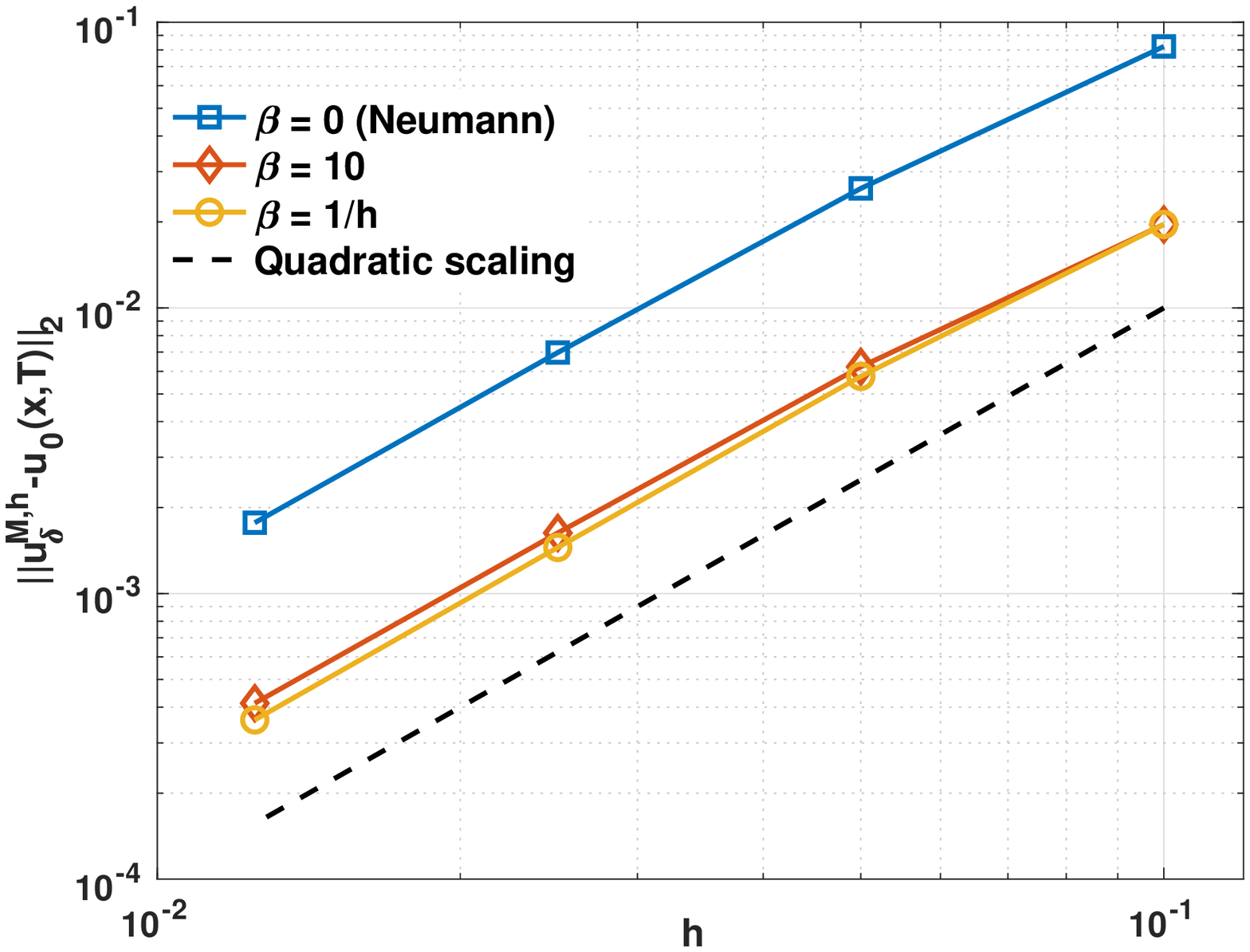}}
    \subfigure{\includegraphics[width =0.48 \textwidth]{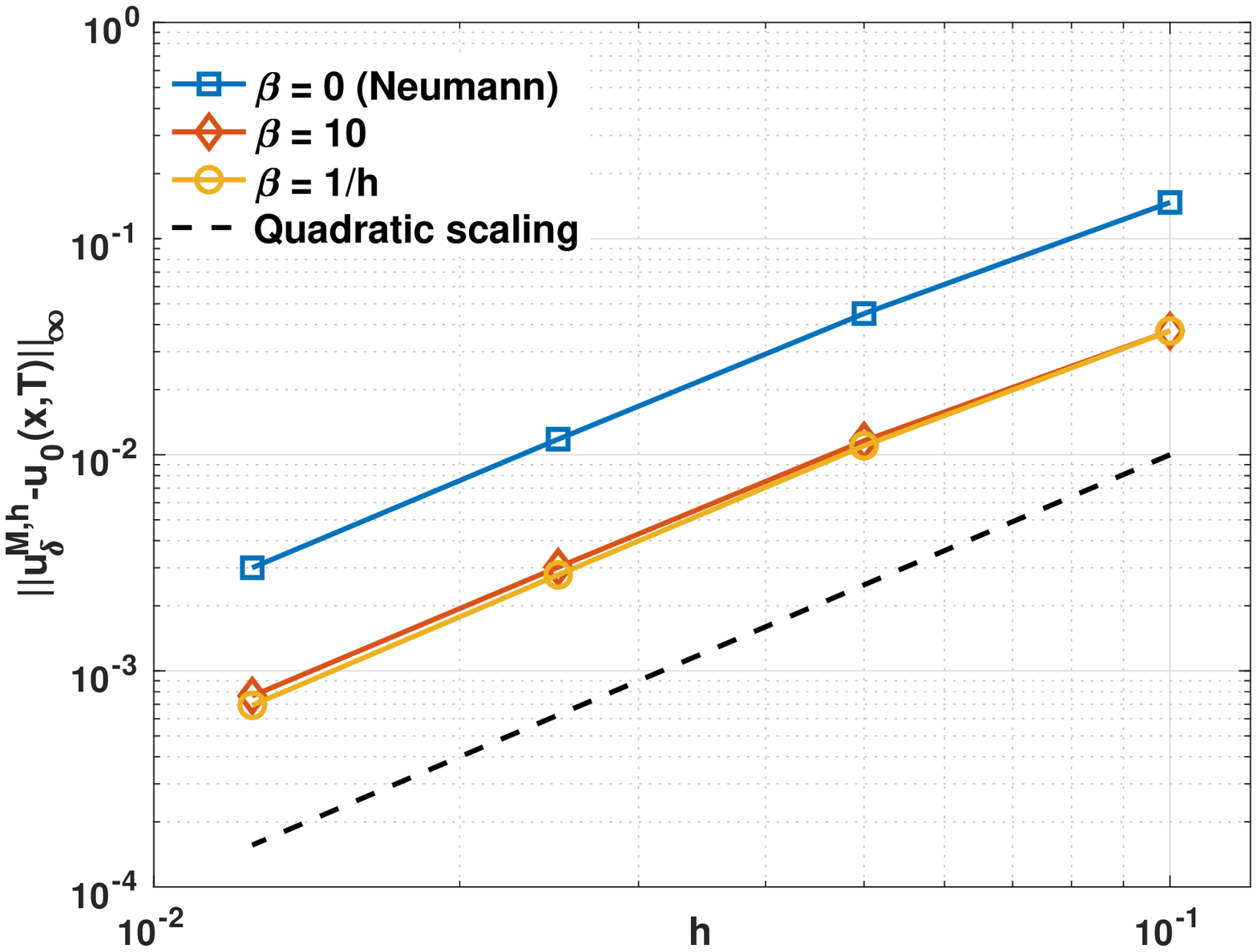}}
    \caption{Test 2 results for nonlocal boundary condition on a circular domain. Left: convergence of the numerical nonlocal solution to the local limit with different Robin coefficients, in the $L^2$ norm. Right: convergence of the numerical nonlocal solution to the local limit with different Robin coefficients, in the $L^\infty$ norm.}
    \label{fig:NLcircle}
\end{figure}

\subsubsection{Test 3: a cross-shape domain}\label{sec:test3}

We now consider a more complicated domain which doesn't satisfy the convex and $C^3$ regularity requirements in the convergence analysis of Section \ref{sec:neumann}. The domain is of cross-shape, presented as $\Omega_{nl}$ in the right plot of Figure \ref{fig:couplingdomain}. Neumann or Robin-type boundary conditions are applied everywhere over the boundary except on point $(-1,-0.5)$ where we set $u^D_{nl}(x,y,t) = t^2\sin(x)\cos(y)$, in order to make the problem well-posed on $\beta=0$ case. Note that this domain is non-convex and the boundary include corners. Therefore, for $\xb\in B\Omega_i$ within distance $\delta$ to the corner, we employ the corner formulation developed in \eqref{eqn:cornerf1}. In this test we set $u^{IC}=0$, $f^{nl}(x,y,t)=(2t+2t^2)\sin(x)\cos(y)$ and a Robin-type boundary condition $g(x,y,t)=\beta t^2\sin(x)\cos(y)+t^2 (\cos(x)\cos(y)n_x-\sin(x)\sin(y)n_y)$ is applied on the sharp interface $\Gamma_i$, where $\mathbf{n}=(n_x,n_y)$ is the outward-pointing unit normal vector on $\Gamma_i$. The analytical local limit solution for above problem setting is also $u_0(x,y,t) = t^2\sin(x)\cos(y)$. Keep a fixed ratio $\delta/h = 3.5$, $\Delta t = 100h^2$ while refining the spatial discretization length scale $h$, the $L^2$ and $L^{\infty}$ norm for the difference between numerical nonlocal solution and the analytical local limit {at $T=1$} are presented in Figure \ref{fig:NLcross}, from which a second-order convergence rate $\mcO(\delta^2)$ is observed. This example verifies the proposed corner formulation and illustrates that the proposed nonlocal Robin-type formulation also achieves asymptotic compatibility on a non-convex domain consisting of line segments and corners, which greatly improves the applicability of the proposed formulation for more complicated scenarios.

\begin{figure}
    \centering
    \subfigure{\includegraphics[width=0.48\textwidth]{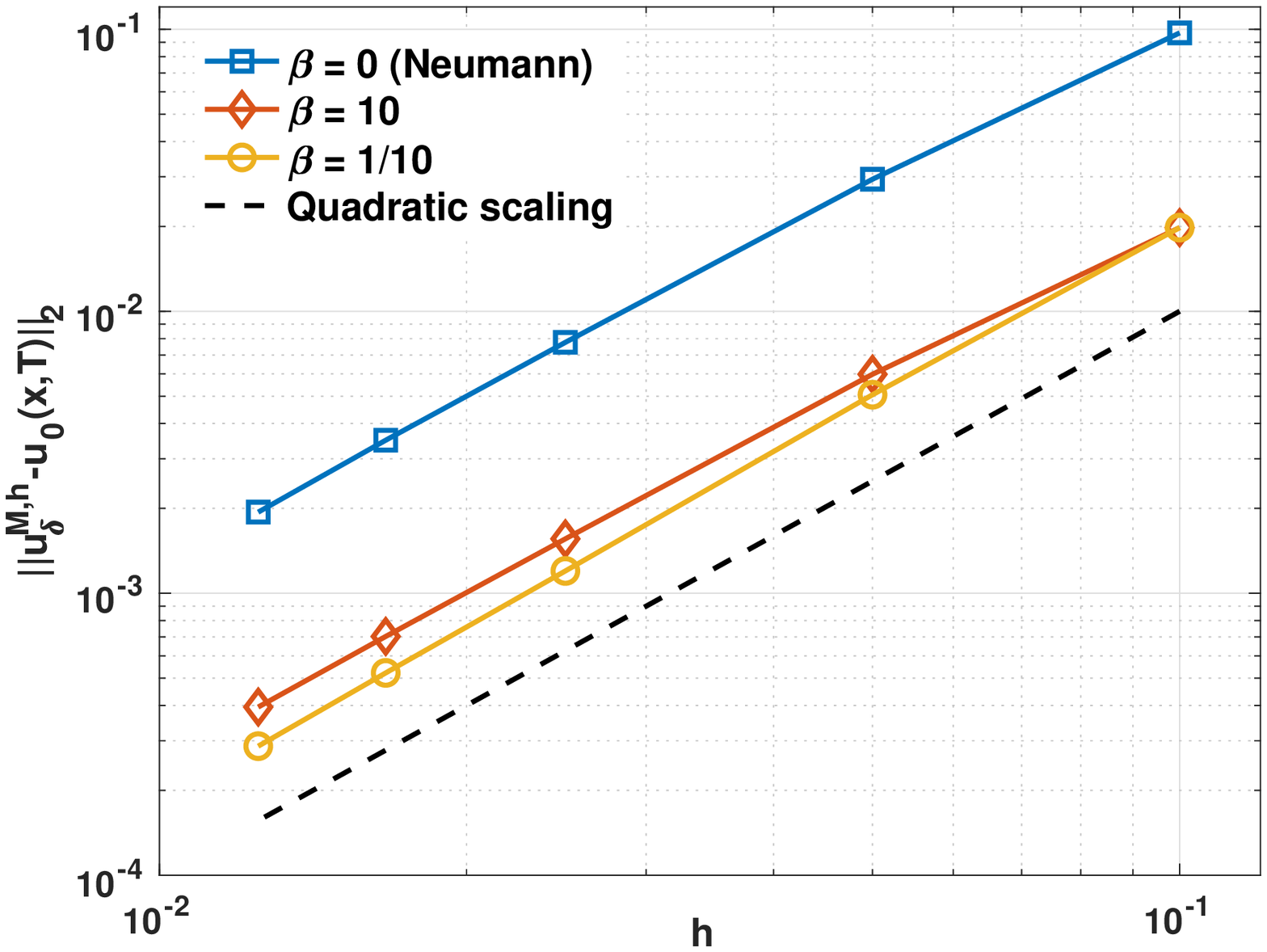}}
    \subfigure{\includegraphics[width =0.48 \textwidth]{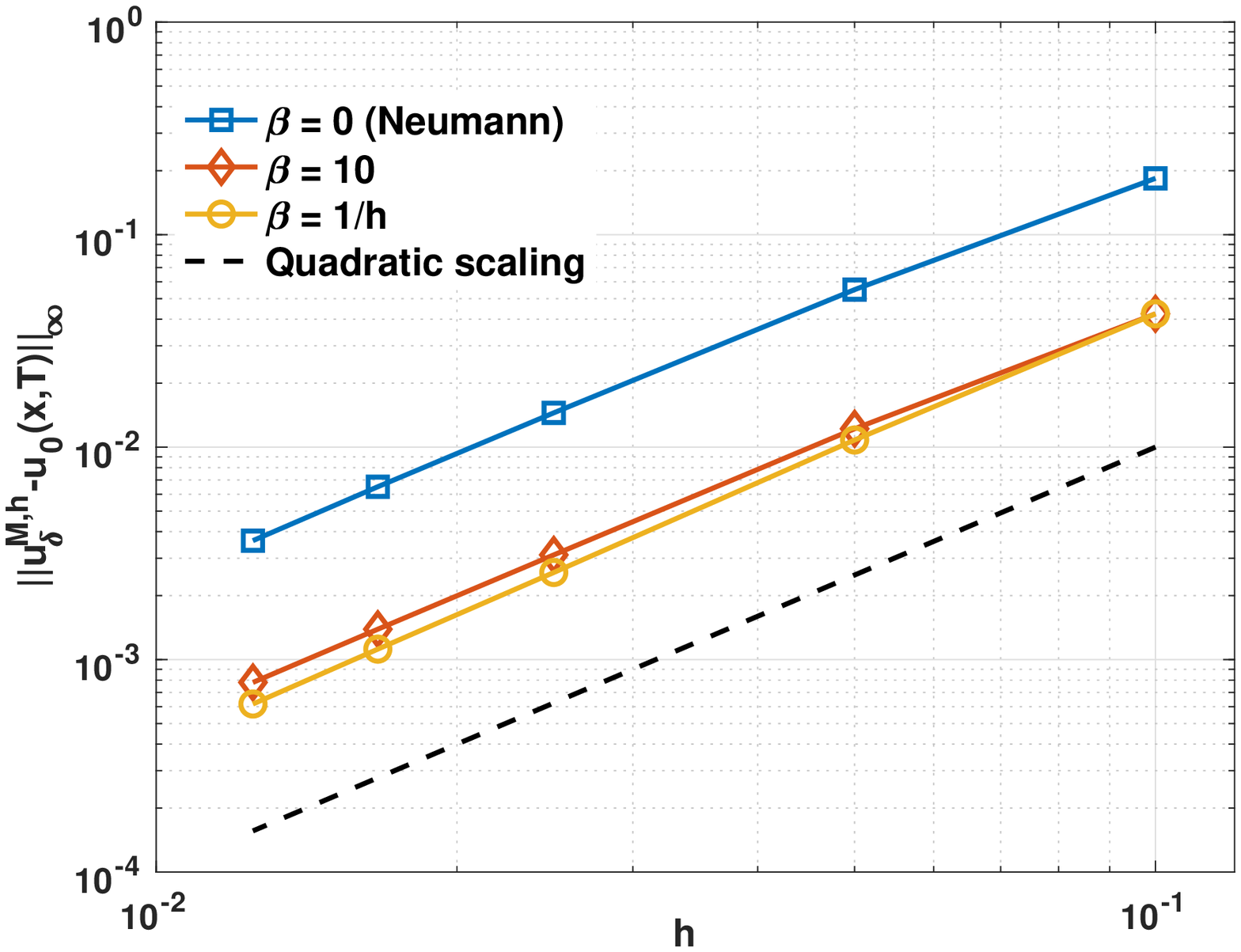}}
    \caption{Test 3 results for nonlocal boundary condition on a cross-shape domain. Left: convergence of the numerical nonlocal solution to the local limit with different Robin coefficients, in the $L^2$ norm. Right: convergence of the numerical nonlocal solution to the local limit with different Robin coefficients, in the $L^\infty$ norm.}
    \label{fig:NLcross}
\end{figure}

\section{Nonoverlapping Local-to-Nonlocal (LtN) Coupling Framework}\label{sec:couple}

In this section, we present an explicit coupling procedure for the local-to-nonlocal coupling problem without overlapping regions. In this coupling framework, a partitioned procedure is employed such that the nonlocal and local subproblems are solved separately, which allows for the reuse of existing codes/methods. {This is of particular value in the case of local-to-nonlocal coupling, since the different model classes are often solved with radically-different code architectures (although similarities have recently been pointed out between certain meshfree discretizations of local and nonlocal problems \cite{Hillman2019}).} The local and nonlocal solvers exchange transmission conditions on the sharp interface $\Gamma_i$, to enforce the continuity of solutions and the energy balance of the whole system. The partitioned procedure can be broadly classified as either explicit or implicit (see, e.g., \cite{badia2008fluid,degroote2009performance,quarteroni1999domain,mathew2008domain,toselli2006domain}). In explicit coupling strategies, {the solution of each sub-problem and the exchange of interface data are performed only once} (or a few times) per time step, while in the implicit coupling strategies an additional sub-iteration is employed at each time step and each sub-problem is solved in a partitioned way via sub-iterations until convergence. For dynamic problems, the explicit coupling strategy is generally more efficient than the implicit coupling strategy, although the former is more likely to be unstable. To develop a stable explicit coupling strategy, proper transmission conditions are required. In the current paper we propose to employ the Robin-type transmission condition, which was proven to be effective in stabilizing the explicit coupling strategy on domain decomposition problems \cite{Burman_2014,fernandez2015generalized}. Specifically, for the nonlocal subproblem, we solve for the nonlocal heat equation \eqref{eqn:nonlocaleqnRobin} with the Robin-type boundary condition applied on the interface $\Gamma_i$. For the local subproblem, we solve a classical heat equation \eqref{eqn:local} with the classical Dirichlet boundary condition applied on the interface $\Gamma_i$. To improve the stability and efficiency of the coupling framework, we propose a numerical approach to choose the Robin coefficient $\beta$.

This section is organized as follows: In Section \ref{41}, we present an explicit coupling procedure for the proposed local-to-nonlocal coupling approach, then in Section \ref{sec:Robincoef} we introduce an approach to numerically obtain the optimal Robin coefficient by minimizing the amplification factor in the discretized coupling system; To numerically verify the proposed local-to-nonlocal coupling approach, in Section \ref{42}, we investigate its performance on three different numerical tests with various configuration settings. %For simplicity, we let both $\alpha_l$ and $\alpha_{nl}$ be $1$ in this section.

\subsection{An Explicit Coupling Approach with Robin Transmission Conditions}\label{41}

In this section, we propose an explicit nonoverlapping local-to-nonlocal coupling framework, by employing the Robin transmission condition developed in Section \ref{section:robin}. Note that here we present the nonlocal model for the case without corners only, since for the case with corners one can simply replace the nonlocal formulation \eqref{eqn:formula4} with \eqref{eqn:cornerf1}.

To introduce the partitioned procedure, we consider a semi-discretized system where the backward Euler method is employed for time discretization. At time step $k$, we solve for the nonlocal solution $u_{nl,\delta}^k$ in $\Omega_{nl}$ and the local solution $u_{l}^{k}$ in $\Omega_l$ using the solutions at the previous time step, $u_{nl,\delta}^{k-1}$ and $u_{l}^{k-1}$: we first solve for $u_{nl,\delta}^{k}$ with
\begin{subequations}\label{eqn:nonlocaleqnRobinCouple}
\begin{align}
&\frac{1}{\Delta t}({u}_{nl,\delta}^k(\xb)-{u}_{nl,\delta}^{k-1}(\xb))-\alpha_{nl}\mcL_{\delta} {u}_{nl,\delta}^k(\xb)=f_{nl}(\xb,t^k),\quad \mathbf{x}\in\Omega_{nl}\backslash B\Omega_{i}\\
\nonumber&\frac{1}{\Delta t}Q_\delta(\xb)({u}_{nl,\delta}^k(\xb)-{u}_{nl,\delta}^{k-1}(\xb))-\alpha_{nl}\mcL_{N\delta} {u}_{nl,\delta}^k(\xb)+\alpha_{nl}\beta V_\delta (\xb){u}_{nl,\delta}^k(\overline{\mathbf{x}})\\
&~~~~=Q_\delta(\xb)f_{nl}(\xb,t^k)+\alpha_{nl}V_\delta (\xb)\left[\dfrac{\partial u_l^{k-1}(\overline{\xb})}{\partial \nb}+\beta u_l^{k-1}(\overline{\mathbf{x}})\right],\quad \mathbf{x}\in B\Omega_{i}\label{eqn:nonlocalbc}\\
&{u}_{nl,\delta}^k(\xb)=u^D_{nl}(\xb,t^k),\quad \mathbf{x}\in B\Omega_{d}
\end{align}
\end{subequations}
then solve for $u_{l}^{k}$ with
\begin{subequations}\label{eqn:localCouple}
\begin{align}
&\frac{1}{\Delta t}(u_{l}^{k}(\xb)-u_{l}^{k-1}(\xb))-\alpha_{l}\Delta u_{l}^{k}(\xb)=f_l(\xb,t^k),\quad \mathbf{x}\in\Omega_{l}\\
&u_{l}^{k}(\xb)=u_{nl,\delta}^k(\xb),\quad \mathbf{x}\in \Gamma_{i}\\
&u_{l}^{k}(\xb)=u_l^D(\xb,t^k).\quad \mathbf{x}\in \Gamma_{d}
\end{align}
\end{subequations}
%The governing equations for this coupling system are 
%\begin{align*}
%    \dot{u}_{\delta}(\bm{x},t)- \mathcal{L}_{\delta} u(\bm{x},t) &=     f_{nl}(\bm{x},t) ,     \quad \forall \bm{x} \in \Omega_{nl}, \\
%   \dot{u}_l(\bm{x},t) - \Delta u_l(\bm{x},t) &= f_l(\bm{x},t), \quad    \forall \bm{x} \in \Omega_{l}.
%\end{align*}
Here $\beta$ is the Robin coefficient which is to be determined in Section \ref{sec:Robincoef} to achieve the optimal coupling performance. $\mathbf{n}$ is the normal vector on interface $\Gamma_i$ pointing from the nonlocal subdomain to the local subdomain. $Q_\delta$, $V_\delta$ are functions depending on the position of $\xb$ and the nonlocal domain geometry, with formulations given in \eqref{eqn:Q}-\eqref{eqn:S}. The nonlocal operator $L_{N\delta}$ for $\xb\in B\Omega_i$ is defined in \eqref{eqn:LN}. In this coupling problem we employ Dirichlet transmission conditions for the local problem and nonlocal Neumann or Robin transmission conditions for the nonlocal problem on the sharp local-nonlocal coupling interface $\Gamma_i$. For presentation simplicity, in the following we neglect the Dirichlet boundary conditions on $B\Omega_d$ and $\Gamma_d$ when presenting the fully-discretized formulation, and focus on the interface transmission conditions.

In the coupling formulation introduced in \eqref{eqn:nonlocaleqnRobinCouple}-\eqref{eqn:localCouple}, since different solvers are employed for the two sub-problems, the local and nonlocal grid points on $\Gamma_i$ is possibly non-conforming. Therefore, to impose the interface transmission conditions one can not simply pass the nodal values on the interface between the local and nonlocal solvers. To obtain the nonlocal Robin-type interface condition, for $\xb_j\in \chi_h\cap B\Omega_i$ we approximate the Robin condition on its projection of $\Gamma_i$ as the interpolation with the solution on local nodes. Denoting $\mathbf{U}_{nl,B\Omega_i}$ as the vector of nodal values of the nonlocal solution when $\xb_j\in \chi_h\cap B\Omega_i$, $\mathbf{U}_{nl,in}$ as the vector of the nonlocal solution on nodes $\xb_j\in \chi_h\cap (\Omega_{nl}\backslash B\Omega_i)$, $\mathbf{U}_{l,\Gamma_i}$ as the vector of nodal values of the local solution on interface $\Gamma_i$ and $\mathbf{U}_{l,in}$ as the vector of nodal values of the local solution on the interior nodes, for each $\xb_j\in \chi_h\cap B\Omega_i$ we obtain $g(\overline{\xb}_j)$ {from nodal values of the numerical local solution}:
\begin{equation}\label{eqn:interRobin}
\dfrac{\partial u_l^{k-1}(\overline{\xb}_j)}{\partial \nb}+\beta u_l^{k-1}(\overline{\mathbf{x}}_j)\approx \left[T_1\mathbf{U}^{k-1}_{l,in}+T_2\mathbf{U}^{k-1}_{l,\Gamma_i}\right]_j.
\end{equation}
Note here $T_1$ and $T_2$ are matrices formed by interpolation coefficients, and their elements linearly depend on the Robin coefficient $\beta$. With the above interpolation formulation, we can then substitute the Robin transmission condition condition $\dfrac{\partial u_l^{k-1}(\overline{\xb})}{\partial \nb}+\beta u_l^{k-1}(\overline{\mathbf{x}})$ applied on the sharp interface $\Gamma_i$ into \eqref{eqn:formula4}, and formulate the fully discretized nonlocal subproblem as the following linear system:
\begin{align}
\nonumber&\dfrac{M_{nl,in}}{\Delta t}\Ub^{k}_{nl,in}+A_{11}\Ub^{k}_{nl,in}+A_{12}\Ub^{k}_{nl,B\Omega_i}=\Fb^k_{nl,in}+\dfrac{M_{nl,in}}{\Delta t}\Ub^{k-1}_{nl,in},\\
&\dfrac{M_{nl,B\Omega_i}}{\Delta t}\Ub^{k}_{nl,B\Omega_i}+A_{21}\Ub^{k}_{nl,in}+A_{22}\Ub^{k}_{nl,B\Omega_i}+\beta\Sigma_1\Ub^{k}_{nl,B\Omega_i}=\Fb^k_{nl,B\Omega_i}+\dfrac{M_{nl,B\Omega_i}}{\Delta t}\Ub^{k-1}_{nl,B\Omega_i}+\Sigma_1T_1 \Ub^{k-1}_{l,in}+\Sigma_1T_2 \Ub^{k-1}_{l,\Gamma_i}.\label{eqn:fulldiscnl}
\end{align}
Here $M_{nl,in}$, $M_{nl,B\Omega_i}$ are the mass matrices corresponding to the nodes in $\Omega_{nl}\backslash B\Omega_i$ and in $B\Omega_i$, respectively, $\Fb^k_{nl,in}$ and $\Fb^k_{nl,B\Omega_i}$ are the external loading terms for nodes in $\Omega_{nl}\backslash B\Omega_i$ and in $B\Omega_i$, respectively, $A_{ij}$, $i,j\in\{1,2\}$, are parts of the stiffness matrices, and $\Sigma_1$ handles the $V_\delta(\xb)$ term and the mapping of each $\overline{\xb}_i$ onto the vector $\Ub_{nl,B\Omega_i}$. On the other hand, to apply the Dirichlet boundary condition on the local side, we need to interpolate the nonlocal numerical solution $\mathbf{U}^k_{nl}$ to obtain an approximation for each $u_{nl,\delta}^k(\xb_j)$ where $\xb_j\in\Gamma_i$ is a node in the local subdomain mesh. Employing the moving least square method \cite{wendland2004scattered,levin1998approximation} with support radius of size $\delta$ and quadratic basis, we reconstruct $u_{nl,\delta}^k(\xb_j)$ as a linear combination from nodal values of the nonlocal solution in $B\Omega_i$:
\begin{equation}\label{eqn:interDiri}
    u_{nl,\delta}^{k}(\xb_j) 
     \approx \left[\Sigma_2 \Ub_{nl,B\Omega_i}^{k}\right]_j.
\end{equation}
Substituting the Dirichlet transmission condition into the local subproblem \eqref{eqn:localdisc}, we then obtain a linear system for the fully discretized local subproblem:
\begin{align}
\nonumber&\dfrac{M_l}{\Delta t}{\Ub}^{k}_{l,in}+B_{11}\Ub_{l,in}^{k}+B_{12}\Ub_{l,\Gamma_i}^{k}=\Fb^{k}_{l}+\dfrac{M_l}{\Delta t}{\Ub}^{k-1}_{l,in},\\
&\Ub_{l,\Gamma_i}^{k}=\Sigma_2 \Ub_{nl,B\Omega_i}^{k}.\label{eqn:fulldiscl}
\end{align}
Here $M_l$ is the local mass matrix, $\Fb_l$ is the global vector of the external loads, and $B_{11}$, $B_{12}$ together forms the local stiffness matrix $[B_{11},B_{12}]=B_l$.

%For the nonlocal problem, we need to consider all  the projection points of the nodes within a distance of $\delta$ from the interface instead of the nodes on the interface only. Denote $\bar{\bm{x}}$ as the projection points, we extract the Robin boundary condition for the nonlocal problem by evaluating the local solution and the gradient of the local solution from the previous step at the projection points $\bar{\bm{x}}$. 

In summary, we obtain the following fully-discretized explicit local-to-nonlocal coupling algorithm: 
\begin{enumerate}
    \setlength{\itemsep}{1pt}
    \item (Both Solvers): Set initial values for $\Ub^0_{nl}$ and $\Ub^0_l$ with the given initial condition $u^{IC}(\xb)$.
    \item for $k = 1,\cdots,M=T/\Delta t$, do 
    \begin{enumerate}
        \setlength{\itemsep}{1pt}
        \item (Local Solver): Calculate the nonlocal transmission condition from the local solution $\Ub^{k-1}_l$ by perform interpolation for each $\xb_j\in \chi_h\cap B\Omega_i$ via \eqref{eqn:interRobin}. Pass the results to the nonlocal solver.
        \item (Nonlocal Solver): Solve the linear system \eqref{eqn:fulldiscnl} of the nonlocal subproblem for the vector of nodal values of the nonlocal solution $\Ub_{nl,in}^{k}$ and $\Ub_{nl,B\Omega_i}^{k}$.
        \item (Nonlocal Solver): Calculate the local transmission condition from the nonlocal solution $\Ub_{nl,B\Omega_i}^{k}$ via the interpolation formulation in \eqref{eqn:interDiri}. Pass the results to the local solver.
        \item (Local Solver): Solve the linear system \eqref{eqn:fulldiscl} of the local subproblem for the vector of nodal values of the local solution $\Ub_{l,in}^{k}$ and $\Ub_{l,\Gamma_i}^{k}$.
        \item Go to time step $k+1$.
        \end{enumerate}
\end{enumerate}

\subsection{Estimates for the Optimal Robin Coefficient}\label{sec:Robincoef}

As will be observed from the numerical tests in Section \ref{42}, the explicit coupling strategy may suffer from slow convergence or even divergence, and therefore a good choice of the Robin coefficient is a necessity. The optimal Robin coefficients can be estimated either theoretically or numerically. In problems with relatively simple and/or structured domain settings, one can perform Fourier decomposition to the analytical solution and obtain the optimal Robin coefficient by minimizing the analytic reduction factor, as shown in \cite{yu2018partitioned}. However, for the coupling problem with general geometry, deriving the analytic expression of the optimal Robin coefficient is often not straightforward, and therefore in this paper we propose a numerical approach to approximate the optimal Robin coefficient $\beta$.

To perform a stability analysis, we consider the homogeneous local-to-nonlocal coupling problem. At the $k$-th time step, the fully discretized coupling system is written as 
\begin{equation*}
    \begin{bmatrix}
    \dfrac{M_{nl,in}}{\Delta t}+A_{11} & A_{12} & 0 & 0 \\
    A_{21} & \dfrac{M_{nl,B\Omega_i}}{\Delta t}+A_{22}+\beta \Sigma_1 & 0 & 0 \\
    0 & 0 & \dfrac{M_l}{\Delta t}+B_{11} & B_{12} \\    
    0 & -\Sigma_2 & 0 & I \\
    \end{bmatrix}
    \begin{bmatrix}
    \Ub_{nl,in}^{k} \\
    \Ub_{nl,B\Omega_i}^{k}\\
    \Ub_{l,in}^{k} \\    
    \Ub_{l,\Gamma_i}^{k} \\
    \end{bmatrix}
    = \begin{bmatrix}
    \dfrac{M_{nl,in}}{\Delta t} & 0 & 0 & 0 \\
    0 & \dfrac{M_{nl,B\Omega_i}}{\Delta t} & \Sigma_1T_1 & \Sigma_1T_2 \\
    0 & 0 & \dfrac{M_l}{\Delta t} & 0 \\
    0 & 0 & 0 & 0 \\
    \end{bmatrix}
    \begin{bmatrix}
    \Ub_{nl,in}^{k-1} \\
    \Ub_{nl,B\Omega_i}^{k-1}\\
    \Ub_{l,in}^{k-1} \\    
    \Ub_{l,\Gamma_i}^{k-1} \\
    \end{bmatrix}.
\end{equation*}
Here the first row corresponds to the discretized nonlocal equation of the interior region, the second row represents the modified nonlocal formulation in $B\Omega_i$ with the nonlocal Robin-type transmission condition, the third row corresponds to the discretized local equation of the interior local nodes, and the last row applies the Dirichlet transmission condition at the interface on the local side. Denoting the matrix $\Lambda$ as 
\begin{equation}\label{eqn:lambda}
    \Lambda = \begin{bmatrix}
    \dfrac{M_{nl,in}}{\Delta t}+A_{11} & A_{12} & 0 & 0 \\
    A_{21} & \dfrac{M_{nl,B\Omega_i}}{\Delta t}+A_{22}+\beta \Sigma_1 & 0 & 0 \\
    0 & 0 & \dfrac{M_l}{\Delta t}+B_{11} & B_{12} \\    
    0 & -\Sigma_2 & 0 & I \\
    \end{bmatrix}^{-1}
    \begin{bmatrix}
    \dfrac{M_{nl,in}}{\Delta t} & 0 & 0 & 0 \\
    0 & \dfrac{M_{nl,B\Omega_i}}{\Delta t} & \Sigma_1T_1 & \Sigma_1T_2 \\
    0 & 0 & \dfrac{M_l}{\Delta t} & 0 \\
    0 & 0 & 0 & 0 \\
    \end{bmatrix},
\end{equation}
and $\{\lambda_i\}$ as the eigenvalues of the matrix $\Lambda$, in the homogeneous coupling system the magnitude of $\lambda_i$ characterizes the convergence rate of the error component along the $i$-th eigenvector, and the fully discretized coupling system is stable when the magnitudes of all $\lambda_i$ are bounded by $1$. Therefore, we define an amplification factor as $\underset{i}\max |\lambda_i|$, then numerically obtain the optimal Robin coefficient $\beta$ by minimizing the reduction factor:
\begin{equation}\label{eqn:beta}
    \beta = \underset{\tilde{\beta}\geq 0}{\operatorname{arg\,min}}\left(\max_i\vert\lambda_i(\tilde{\beta})\vert\right).
\end{equation}
%\DK{Maybe a better notation would be
%\begin{equation}
%    \beta = \underset{\tilde{\beta}\geq 0}{\operatorname{arg\,min}}\left(\max_i\vert\lambda_i(\tilde{\beta})\vert\right) .
%\end{equation}
%}\YY{Sounds good! Equation (4.8) modified.}

\begin{remark}\label{remark:3}
The expression of $\Lambda$ in  \eqref{eqn:lambda} indicates that the optimal Robin coefficient may depend on the local and nonlocal subdomains, the time step size $\Delta t$, the local and nonlocal discretization methods and the spatial discretization length scale $h$, and the diffusivity parameters $\alpha_l$ and $\alpha_{nl}$. For systems with large degree of freedoms, the matrix $\Lambda$ is of size $(DOF_{nl}+DOF_l)^2$ which might make the calculation of eigenvalues unfeasible. However, two observations make the proposed numerical approach applicable for large local-to-nonlocal coupling systems:
\begin{itemize}
\item The matrix $\Lambda$ is independent of time and therefore the analysis on amplification factor only needs to be performed once.
\item In numerical tests of Section \ref{42}, we have observed that when taking the CFL-like condition $\Delta t=\mcO(h^2)$ the optimal Robin coefficient $\beta$ scales with the spatial discretization length scale $h$ as $\beta=\mcO(1/h)$. This finding was also suggested in literatures on applying Robin transmission conditions in other domain-decomposition problems, such as in fluid--structure interaction problems (see, e.g., \cite{Burman_2014}).
\end{itemize}
Therefore, for a dynamic local-to-nonlocal coupling problem, one only need to calculate the optimal Robin coefficient $\beta_0$ once on a coarse grid with spatial discretization length scale $h_0$ with the same domain settings. A scaled Robin coefficient $\beta=\dfrac{\beta_0 h_0}{h}$ can then be employed in the final simulation with spatial discretization length scale $h$.
\end{remark}

%we can analyze the optimal Robin coefficient from the eigenvalues of the matrix $\Lambda$. In order to keep the coupling system stable, all the eigenvalue $\{\lambda_i\}$ need to satisfy the condition that $|\lambda_i| < 1$. In Figure \ref{fig:line_infconv}, \ref{fig:circle_infconv}, and \ref{fig:cross_infconv}, the horizontal axis is the Robin coefficient, and the vertical axis represents the eigenvalue that has the largest norm(the eigenvalues could be complex numbers). The figures show how the largest eigenvalue changes with respect to Robin coefficients when $h = \{1/10, 1/20\}$. The lowest point in the graph tells the optimal Robin coefficient.

\begin{remark}\label{remark:cfl}
In the explicit coupling strategy \eqref{eqn:nonlocaleqnRobinCouple}-\eqref{eqn:localCouple}, since the transmission condition $\left[\dfrac{\partial u_l^{k-1}(\overline{\xb})}{\partial \nb}+\beta u_l^{k-1}(\overline{\mathbf{x}})\right]$ from the local side is generated by the local solution from the last time step, the Robin-type transmission condition \eqref{eqn:nonlocalbc} results in a splitting error. When employing the Robin coefficient $\beta=\mcO(1/h)$ and considering piecewise linear finite elements in the explicit coupling strategy of two classical local heat equations, this splitting error was reported to be of order $\dfrac{\Delta t}{h}$ (see, e.g., \cite{Burman_2014,burman2019stability}) in $L^2$ error estimates. Therefore, the time step $\Delta t$ has to be chosen small enough compared to the spatial discretization size $h$. For instance, under a ``CFL-like'' condition $\Delta t=\mcO(h^2)$ the splitting error is expected to be of order $\mcO(h)$.
\end{remark}

\subsection{Numerical Results for Local-to-Nonlocal Coupling Framework}\label{42}

In this section, we present a series of numerical tests using the proposed local-to-nonlocal coupling framework, where the nonlocal subdomain is either adjacent to the local subdomain (as shown in the left plot of Figure \ref{fig1}) or embedded in the local subdomain (as shown in the right plot of Figure \ref{fig1}). Specifically, three types of representative domain decomposition settings are employed: (1) In Section \ref{sec:test1couple},  we consider a square nonlocal subdomain which is adjacent to a square local subdomain on one side. The coupling configuration is illustrated in the left plot of Figure \ref{fig1}. (2) In Section \ref{sec:test2couple}, we demonstrate the case with a circular nonlocal subdomain fully embedded in a square local subdomain, as shown in the left plot of Figure \ref{fig:couplingdomain}. (3) To investigate the coupling framework performance on complicated domain settings we consider a cross-shape nonlocal subdomain embedded in a square local subdomain in Section \ref{sec:test3couple}. An illustration of the settings is shown in the right plot of Figure \ref{fig:couplingdomain}. With these tests, we aim to provide a validation for our analysis of the optimal Robin coefficient $\beta$, and to demonstrate the capability of our coupling framework in handling both homogeneous ($\alpha_l=\alpha_{nl}$) and heterogeneous ($\alpha_l\neq\alpha_{nl}$) local-to-nonlocal coupling systems with non-trivial domain configuration settings. Moreover, to demonstrate the asymptotic convergence of the propose coupling approach when the nonlocal interaction region $\delta\rightarrow 0$, in this section we also demonstrate the ``M-convergence'' of the coupling framework by fixing the ratio of $\delta$ and the spatial discretization length $h$ and take $h\rightarrow 0$. As discussed in Remark \ref{remark:cfl}, to provide an $\mcO(h)$ bound for the splitting error introduced in the explicit coupling strategy, in all tests we choose the time step size $\Delta t$ following a ``CFL-like'' condition $\Delta t=\mcO(h^2)$.

\subsubsection{LtN Test 1: coupling problem with a straight line interface}\label{sec:test1couple}

As the first local-to-nonlocal coupling (LtN) test, we consider a local-to-nonlocal coupling problem where the local and nonlocal subdomains are adjacent, as demonstrated in the left plot of Figure \ref{fig1}. Specifically, in this test we set $\Omega_{nl}=[0,1] \times [0,1]$, $\Omega_{nl}=[1,2] \times [0,1]$ and demonstrate the numerical performance of the coupling framework on both $\alpha_l=\alpha_{nl}$ and $\alpha_l\neq\alpha_{nl}$ cases. For each case we first investigate the optimal Robin coefficient $\beta$ following the numerical approach introduced in \eqref{eqn:beta}, then employ the optimal $\beta$ to study the asymptotic convergence of the numerical solution to the local limit when $\Delta t,\delta,h\rightarrow 0$. We also employ linear/quadratic analytical solutions on both subdomains to investigate the patch-test consistency of the proposed coupling method.

We first consider the $\alpha_l=\alpha_{nl}$ case by assuming $\alpha_l=\alpha_{nl}=1$, without loss of generality. In this test, the initial temperature $u^{IC}=0$ in the whole domain. In the nonlocal exterior boundary $B\Omega_d$ and the local exterior boundary $\Gamma_d$, prescribed Dirichlet boundary conditions 
\begin{equation}\label{eqn:caseAdiri}
u^{D}_{nl}(x,y,t)=t\sin(x)\cos(y),\quad u^{D}_{l}(x,y,t)=t\sin(x)\cos(y),
\end{equation}
are applied. The external loadings are set as
\begin{equation}\label{eqn:caseAf}
f_{nl}(x,y,t)=(1+2t)\sin(x)\cos(y), \text{ for } (x,y)\in \Omega_{nl},\quad f_l(x,y,t)=(1+2t)\sin(x)\cos(y), \text{ for }(x,y)\in \Omega_{l}.
\end{equation}
This problem has the following analytic solution for the local subproblem:
\begin{equation}\label{eqn:caseul}
u_l(x,y,t) = t\sin(x)\cos(y),\; (x,y)\in \Omega_{l},
\end{equation}
which coincides with the analytical expression of local limit of the nonlocal solution, i.e.,
\begin{equation}\label{eqn:caseu0}
\lim_{\delta\rightarrow 0} u_{nl,\delta}(x,y,t)=u_0(x,y,t)=t\sin(x)\cos(y),\;(x,y)\in \Omega_{nl}.
\end{equation}
Taking $\Delta t = 10 h^2$ and $\delta /h = 3.9$, in the left plot of Figure \ref{fig:line_infconv} the amplification factor $\underset{i}\max |\lambda_i|$ for the discretized coupling system is plotted versus the Robin coefficient $\beta$, for two different spatial discretization length scales $h=1/10$ and $h=1/20$. It can be observed that when $h=1/10$, $\underset{i}\max |\lambda_i|$ achieves the minimum when $\beta=3$; when the spatial discretization size $h$ is decreased to $1/20$, the minimum of $\underset{i}\max |\lambda_i|$ occurs at $\beta=6$. Therefore, the amplification factor analysis suggests $\beta=\dfrac{3}{10h}$ for this problem setting with spatial discretization length scale $h$. To verify the analysis of $\beta$ and investigate the asymptotic compatibility of the coupling framework, in the right plot of Figure \ref{fig:line_infconv} we demonstrate the convergence of the numerical solution to the local limit, i.e., to $u_0$ and $u_l$, in the $L^{\infty}$ norm at $T=1$. Five difference Robin coefficients $\beta=0$, $\dfrac{3}{10h}$, $\dfrac{7}{10h}$, $\dfrac{10}{h}$ and $\dfrac{100}{h}$ are considered. We can see that when $\beta=\dfrac{3}{10h}$, the convergence rate $\mcO(h)=\mcO(\delta)$ is achieved and the numerical solution has the fastest convergence to the local limit. On the other hand, when taking the Neumann-type transmission condition $\beta=0$, the coupling framework is unstable for small $h$. This is consistent with the amplification factor analysis on the left plot of Figure \ref{fig:line_infconv} where $\underset{i}\max |\lambda_i|$ exceeds $1$ for $h=1/20$ and $\beta=0$. When taking large values of $\beta$, the amplification factor $\underset{i}\max |\lambda_i|$ gets close to $1$, which is also consistent with the slow convergence observed in the $\beta=\dfrac{100}{h}$ case in the right plot of Figure \ref{fig:line_infconv}.

%\begin{figure}
%    \centering
%    \includegraphics[width= .8\textwidth]{pics/Linecouple_domain.png}
%    \caption{The domain of the coupling problem with a straight line interface}
%    \label{fig:linedomain}
%\end{figure}
\begin{figure}[ht!]
    \centering
    \begin{minipage}{0.5\textwidth}
    \includegraphics[width = \textwidth]{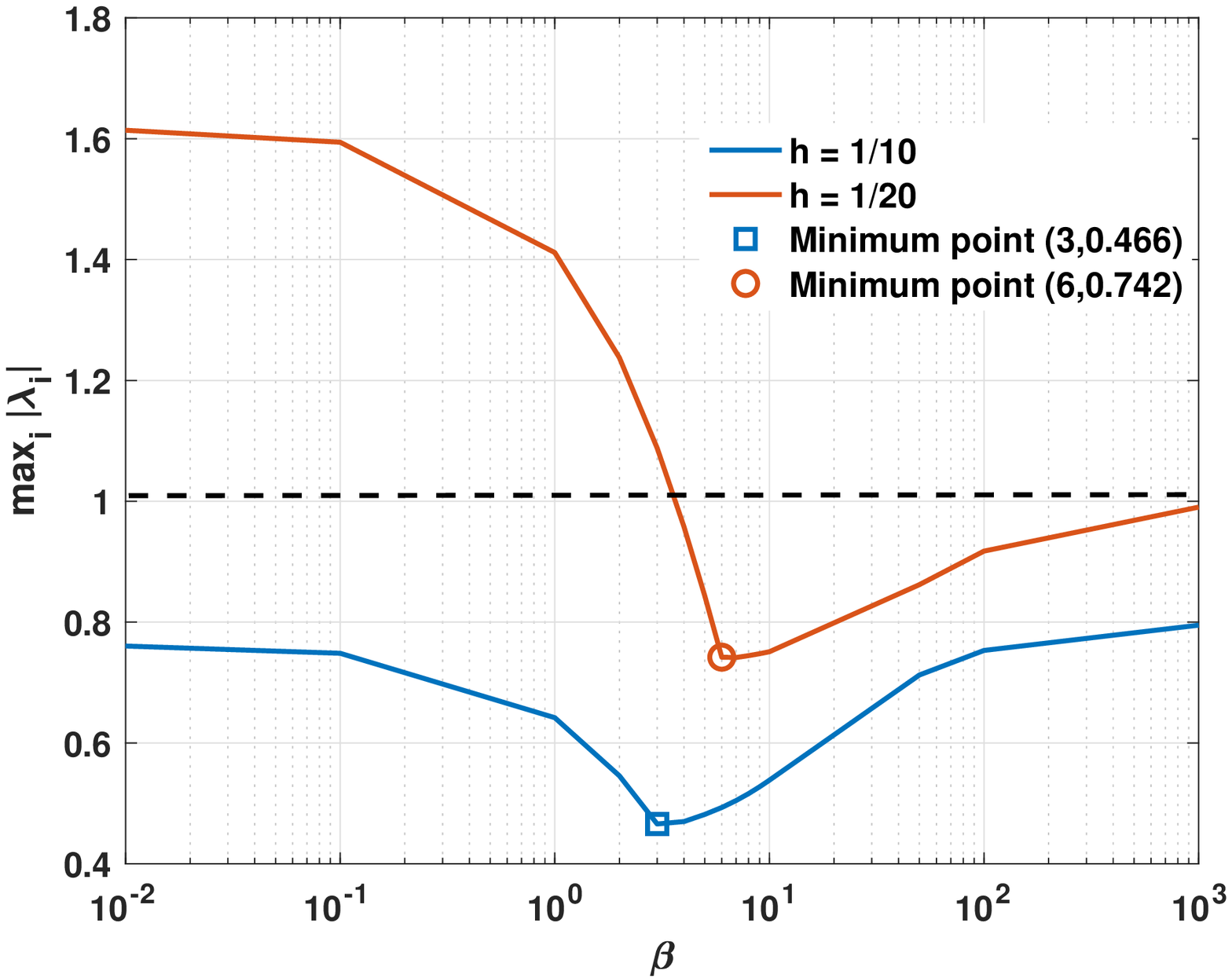}
    \end{minipage}%
    \begin{minipage}{0.5\textwidth}
    \includegraphics[width = \textwidth]{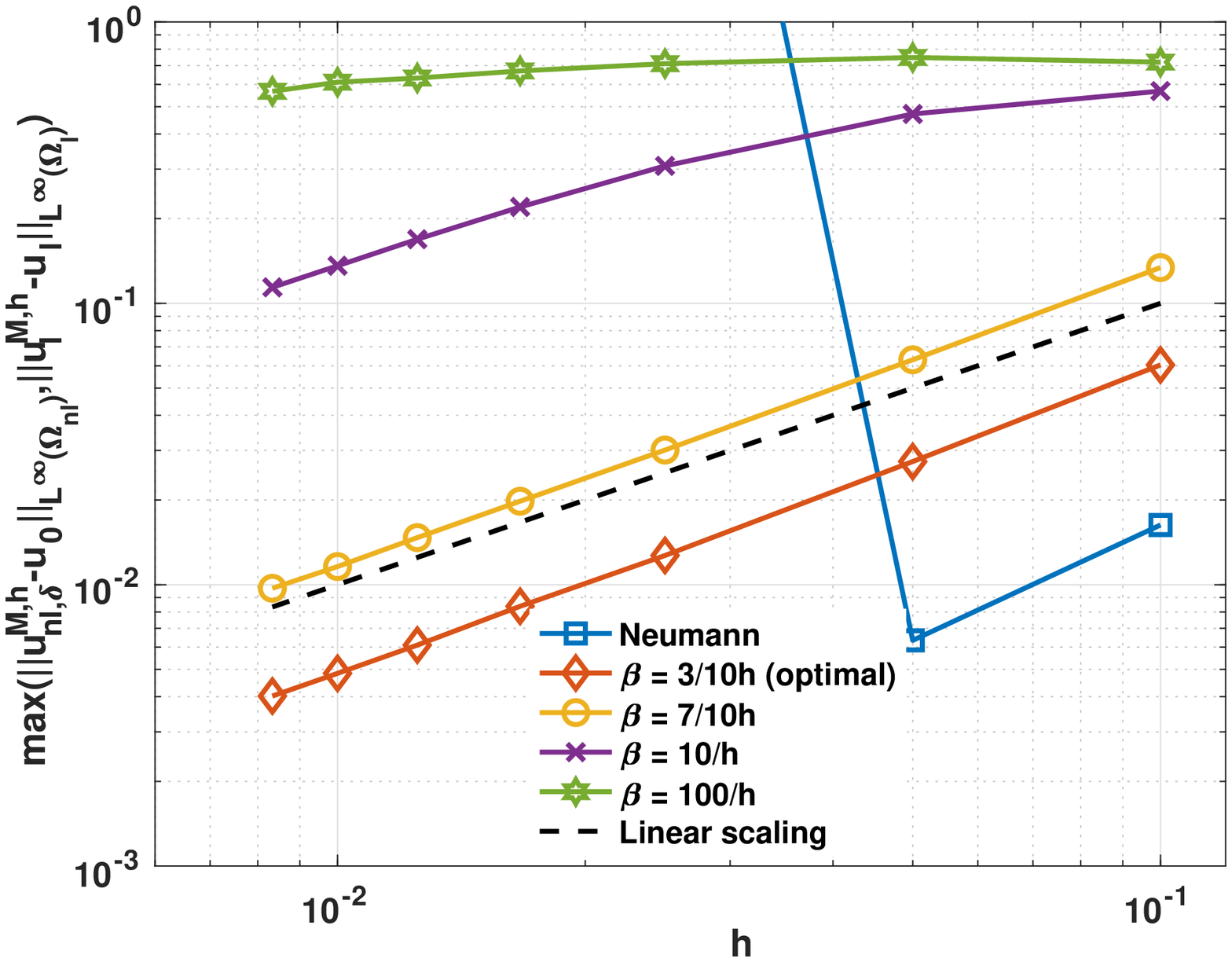}
    \end{minipage}
    \caption{LtN test 1 results for coupling problem with a straight line interface when $\alpha_{nl}=\alpha_l=1$. Left: the amplification factor $\underset{i}\max|\lambda_i|$ as a function of Robin coefficient $\beta$, when $h=\{1/10,1/20\}$. Right: convergence of the numerical solution to the local limit with different Robin coefficients, in the $L^\infty$ norm.}
    \label{fig:line_infconv}
\end{figure}

To illustrate the performance of the non-overlapping coupling framework in handling heterogeneous domains with jumps in physical properties, we now consider a coupling problem with different diffusivities in two subdomains. Specifically, we set $\alpha_{nl} = 1$ and $\alpha_l = 2$, and consider two problem settings:
\begin{itemize}
    \item \textbf{Heterogeneous domain setting A:} 
%\begin{subequations}\label{eqn:heteroA}
\begin{align*}
&u^{IC}(x,y)=0, \text{ for } (x,y)\in \Omega_{nl}\cup \Omega_l,\\
&u^{D}_{nl}(x,y,t)=t\sin(x)\cos(y) \text{ for } (x,y) \in B\Omega_{d}, \quad u^D_{l}(x,y,t) = t\sin(x)\cos(y), \text{ for } (x,y)\in \Gamma_{d},\\
&f_{nl}(x,y,t) = (1+2t)\sin(x)\cos(y), \text{ for } (x,y) \in \Omega_{nl}, \quad f_{l}(x,y,t) = (1+4t)\sin(x)\cos(y), \text{ for } (x,y)\in \Omega_{l},\\
&\lim_{\delta\rightarrow 0}u_{nl,\delta}(x,y,t)=u_{0}(x,y,t)=t\sin(x)\cos(y) \text{ for } (x,y) \in \Omega_{nl}, \quad u_{l}(x,y,t) = t\sin(x)\cos(y), \text{ for } (x,y)\in \Omega_{l}.
\end{align*}
%\end{subequations}
    \item \textbf{Heterogeneous domain setting B:} 
%\begin{subequations}\label{eqn:heteroB}
\begin{align*}
&u^{IC}(x,y)=0, \text{ for } (x,y)\in \Omega_{nl}\cup \Omega_l,\\
&u^{D}_{nl}(x,y,t)=tx^4 \text{ for } (x,y) \in B\Omega_{d}, \quad u^D_{l}(x,y,t) = t(3x^2-2x), \text{ for } (x,y)\in \Gamma_{d},\\
&f_{nl}(x,y,t) = x^2(x^2-12t), \text{ for } (x,y) \in \Omega_{nl}, \quad f_{l}(x,y,t) = 3x^2-2x-12t, \text{ for } (x,y)\in \Omega_{l},\\
&\lim_{\delta\rightarrow 0}u_{nl,\delta}(x,y,t)=u_{0}(x,y,t)=tx^4 \text{ for } (x,y) \in \Omega_{nl}, \quad u_{l}(x,y,t) = t(3x^2-2x), \text{ for } (x,y)\in \Omega_{l}.
\end{align*}
%\end{subequations}
\end{itemize}
In both settings we keep a fixed ratio $\delta/h = 3.9$ and take the time step size $\Delta t = 10h^2$. Note here in setting A, the local limit of $u_{nl,\delta}$ coincides with the analytical local solution $u_l$, although there is a discontinuity of the external loading across the interface $\Gamma_i$. In setting B, besides the discontinuous external loading, when $\delta\rightarrow 0$ the analytical local limit is also not smooth on the interface $\Gamma_i$. However, since $u_0=u_l$ and $\dfrac{\partial u_0}{\partial \mathbf{n}}=\dfrac{\partial u_l}{\partial \mathbf{n}}$ on $\Gamma_i$, the analytical local limit in setting B still satisfies the classical Dirichlet, Neumann and Robin transmission conditions. 

As discussed in Remark \ref{remark:3}, since the eigenvalues of $\Lambda$ depend on $\Delta t$, $h$, $\Omega_{nl}$, $\Omega_l$, $\alpha_l$ and $\alpha_{nl}$, setting A and setting B should have the same optimal Robin coefficient $\beta$, and this optimal $\beta$ differs from the optimal $\beta$ we have obtained in the test on homogeneous domain \eqref{eqn:caseAdiri}-\eqref{eqn:caseu0}. In Figure \ref{fig:line2robin}, we investigate the optimal Robin coefficient $\beta$ for heterogeneous domain problem by plotting the amplification factor $\underset{i}\max |\lambda_i|$ as a function of $\beta$ for fixed spatial discretization size $h=1/10$ and $h=1/20$. It can be seen that the minimum of $\underset{i}\max |\lambda_i|$ occurs at $\beta=\dfrac{4}{10h}$. To numerically verify the choice of optimal Robin coefficient and to study the asymptotic convergence of the analytical solution, in Figure \ref{fig:line2upoly} we demonstrate the convergence results of numerical solution to the analytical local limit at time $T =1$, for both problem setting A (in the left plot) and problem setting B (in the right plot). Among five different values of Robin coefficient $\beta$, it is observed that the optimal convergence $\mcO(h)=\mcO(\delta)$ is achieved at $\beta=\dfrac{4}{10h}$ -- the optimal coefficient suggested in the amplification factor analysis. Slow convergence and divergent results are also observed when taking large $\beta$ and $\beta=0$, respectively. This observation is also consistent with the amplification factor analysis in Figure \ref{fig:line2robin}, and it further indicates that choosing a proper Robin coefficient is of critical for the numerical stability and the asymptotic convergence rate to the local limit in the proposed coupling framework.

\begin{figure}[ht!]
    \centering
    \begin{minipage}{0.5\textwidth}
    \includegraphics[width = \textwidth]{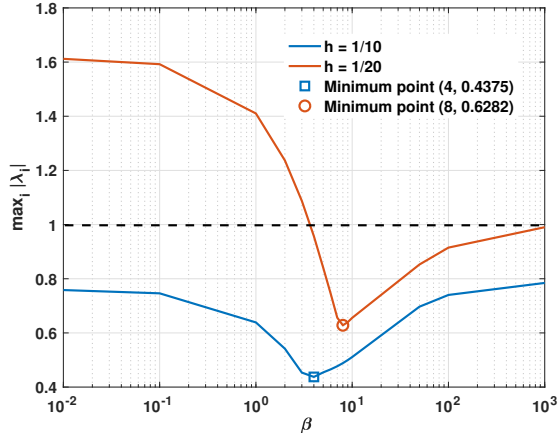}
    \end{minipage}
    \caption{LtN test 1 results with the heterogeneous domain setting ($\alpha_{nl}=1$, $\alpha_l=2$):  the amplification factor $\underset{i}\max  |\lambda_i|$ as a function of Robin coefficient $\beta$ when $h=\{1/10,1/20\}$.}
    \label{fig:line2robin}
\end{figure}

\begin{figure}[ht!]
    \centering
    \begin{minipage}{0.49\textwidth}
    \includegraphics[width = \textwidth]{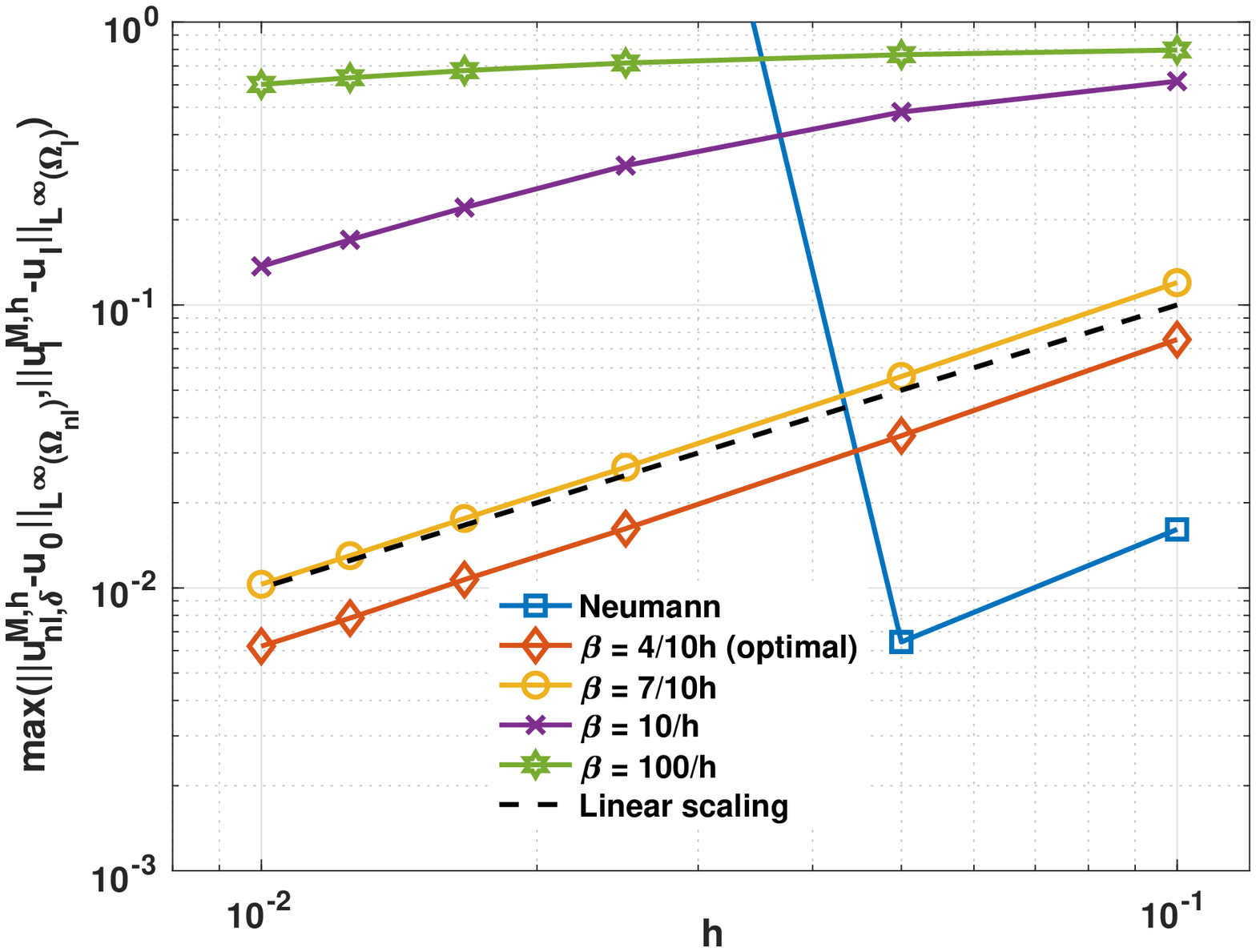}
    \end{minipage}
    \begin{minipage}{0.49\textwidth}
    \includegraphics[width = \textwidth]{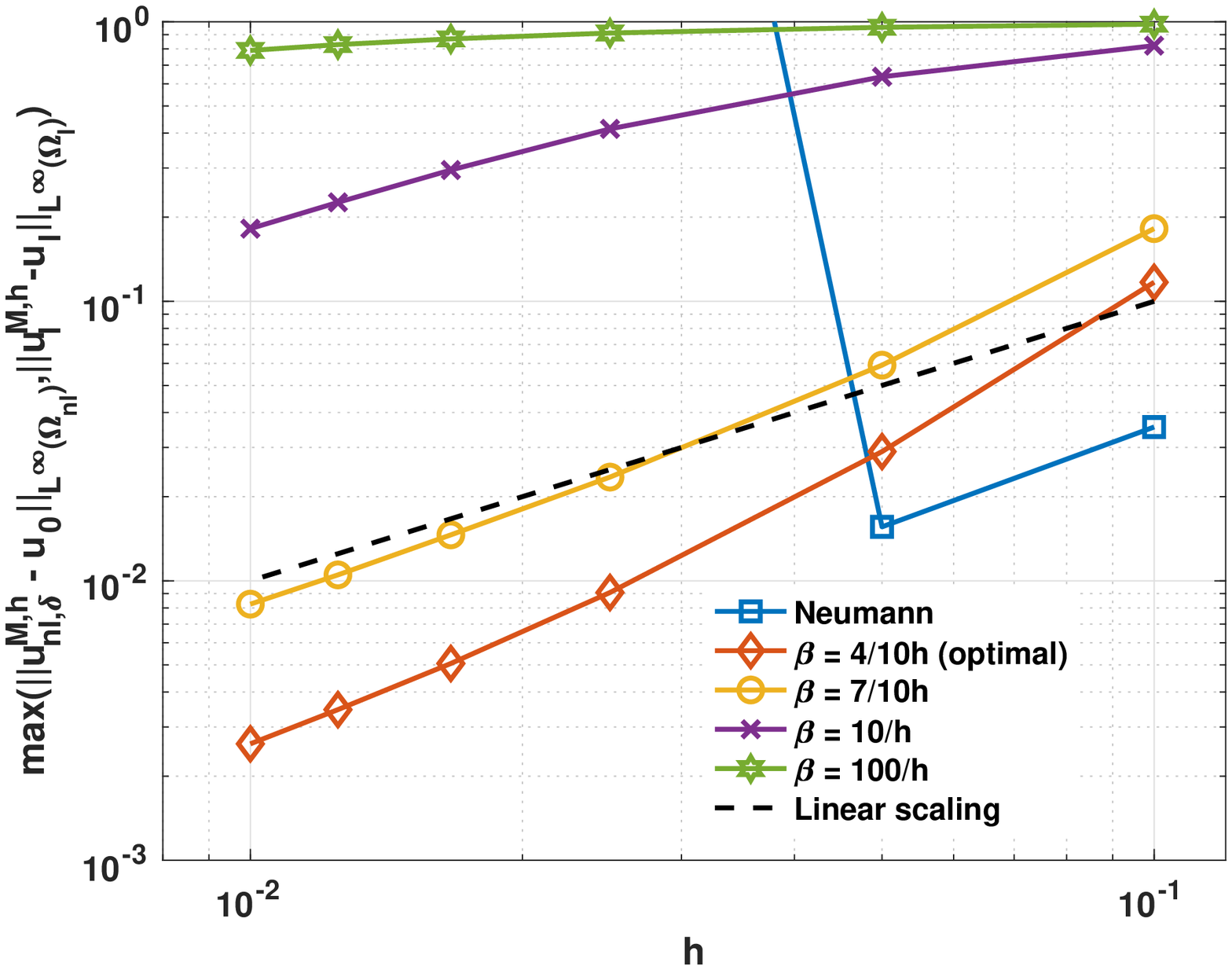}
    \end{minipage}
    \caption{LtN test 1 results for convergence of the numerical solution to the local limit with different Robin coefficients, on heterogeneous domain setting ($\alpha_{nl}=1$, $\alpha_l=2$). Left: convergence in the $L^{\infty}$ norm with problem setting A. Right: convergence in the $L^{\infty}$ norm with problem setting B.}
    \label{fig:line2upoly}
\end{figure}

Lastly, we study the patch-test consistency of the proposed local-to-nonlocal coupling framework, by employing fabricated analytical solutions such that the local and nonlocal analytical solutions, $u_l$ and $u_{nl,\delta}$, coincide. In a patch-test consistent coupling framework, the local and nonlocal subproblems by coupling the corresponding models should still return the same problem solution. We first take a linear analytical solution $u_{nl,\delta}(x,y,t)=u_{nl}(x,y,t)= x$ and plot the numerical solution along the middle line $y=1/2$ in the left plot of Figure \ref{fig:linearpatch}. In this test we take $\alpha_l=\alpha_{nl}=1$, $h=1/20$, $\delta=3.9h$, $\Delta t=10h^2=1/40$ and the optimal Robin coefficient $\beta=6$. It is observed that the linear patch test results are in good agreement with the expected linear solution, and the numerical solution is of machine accurate. To further check the quadratic patch test consistency, we take a quadratic analytical solution $u_{nl,\delta}(x,y,t)=u_{nl}(x,y,t)= x^2$ and plot the numerical solution along the middle line $y=1/2$ in the right plot of Figure \ref{fig:linearpatch}. Although the numerical solution visually agrees well with the analytical solution, we observe a numerical error since $x^2$ doesn't belong to the space of piecewise linear finite elements, and therefore the discretization method for the local subproblem introduces a numerical error. In Table \ref{tab:quapatch} we demonstrate the numerical errors in both the $L^2$ norm and $L^\infty$ norm with different combinations of $h$ and $\beta$. In all tests we take time step $\Delta t=10h^2$. The results show that the numerical error is almost independent of $\beta$ and it converges linearly with decreasing $h$. To further confirms that the numerical error in quadratic patch test is introduced by the linear finite element method, we employ quadratic finite elements for the local subproblem solver and provide the results in Table \ref{tab:sopatch}. The numerical results show that the coupling framework achieves machine accuracy. Therefore, when $x^2$ is in the space of finite elements, the proposed coupling framework passes both linear and quadratic patch tests.

\begin{figure}[ht!]
    \centering
    \begin{minipage}{.5\textwidth}
    \includegraphics[width = \textwidth]{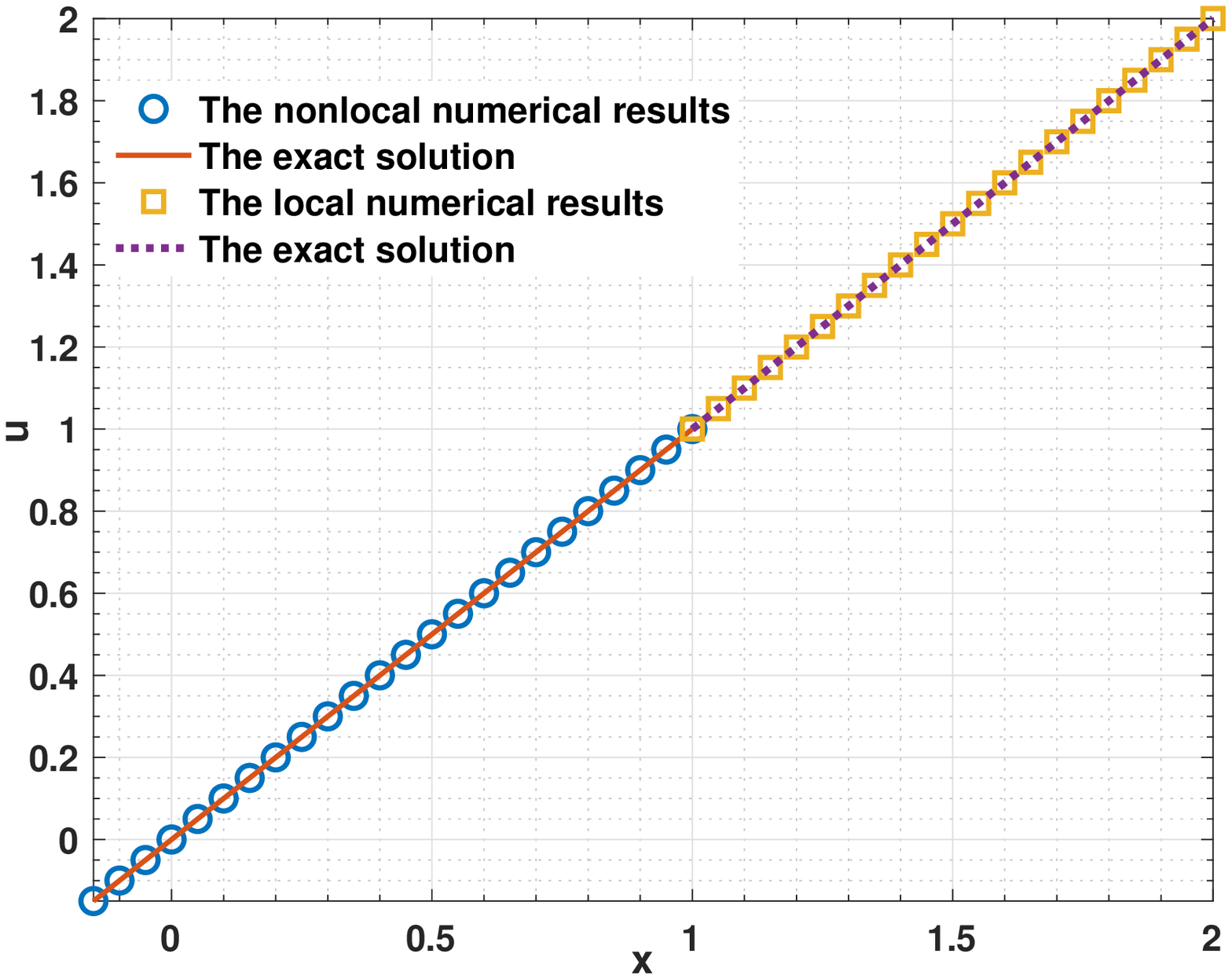}
    \end{minipage}%
    \begin{minipage}{.5\textwidth}
    \includegraphics[width = \textwidth]{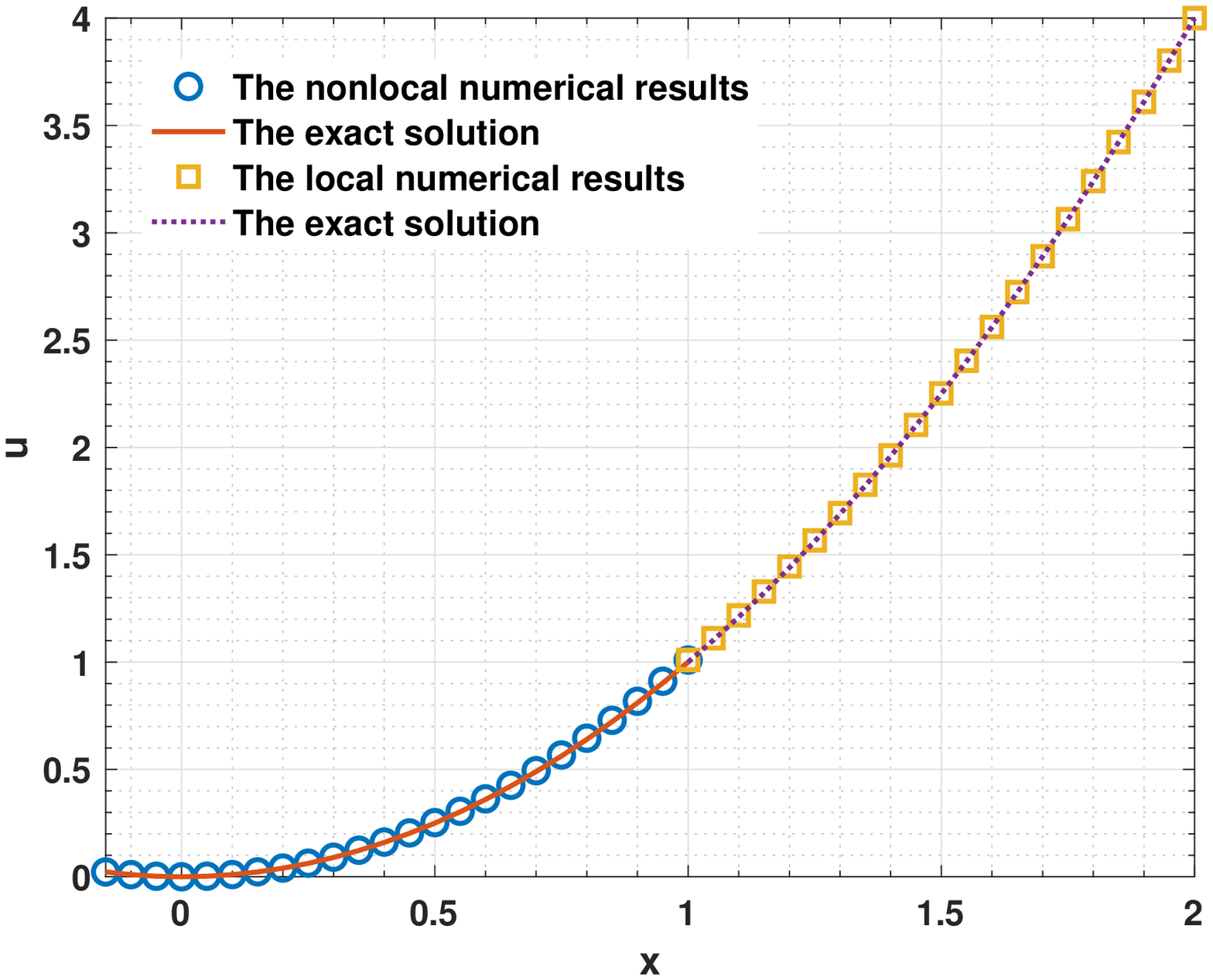}
    \end{minipage}%
    \caption{LtN test 1 results for coupling problem with a straight line interface. Left: linear patch test. The comparison of the numerical results with $h = 1/20$, $\beta = 6$ and the analytical solution $u_{nl,\delta}=u_l = x$. Right: quadratic patch test. The comparison of numerical results with $h=1/20$, $\beta=6$ and the analytical solution $u_{nl,\delta}=u_l = x^2$.}
    \label{fig:linearpatch}
\end{figure}

\begin{table}[ht!]
    \centering
    \begin{tabular}{ccccccccccc}
    \hline
    \multirow{2}{*}{Local Solver} &\multirow{2}{*}{$h$} & \multirow{2}{*}{$\beta$} & \multicolumn{4}{c}{Nonlocal problem} & \multicolumn{4}{c}{Local problem}\\
    \cline{4-11}
    && & $||u_{nl}-u^{ex}||_2$& rate & $||u_{nl}-u^{ex}||_{\infty}$ & rate & $||u_{l}-u^{ex}||_2$ & rate & $||u_{l}-u^{ex}||_{\infty}$ & rate \\
    \hline
    linear FEM &1/10 &3& $6.86 \times 10^{-3}$ & -- & $1.98\times 10^{-2}$
    & -- & $6.18 \times 10^{-3}$ & -- & $1.98 \times 10^{-2}$
    & --\\
    linear FEM &1/20 & 6&$3.17 \times 10^{-3}$ & 1.11 & $9.72 \times 10^{-3}$
    & 1.03 & $3.00 \times 10^{-3}$ & 1.04 & $9.72 \times 10^{-3}$
    & 1.03\\
    linear FEM &1/40 & 12&$1.50 \times 10^{-3}$ & 1.08 & $4.75 \times 10^{-3}$
    & 1.04 & $1.46 \times 10^{-3}$ & 1.03 & $4.75 \times 10^{-3}$
    & 1.04\\
%    1/60 & 18 &$9.80 \times 10^{-4}$ & 1.05 & $3.14 \times 10^{-3}$
%    & 1.02 & $9.64 \times 10^{-4}$ & 1.02 & $3.14 \times 10^{-3}$
%    & 1.02\\
    linear FEM &1/80 & 24 &$7.28 \times 10^{-4}$ & 1.04 & $2.34 \times 10^{-3}$
    & 1.02 & $7.19 \times 10^{-4}$ & 1.02 & $2.34 \times 10^{-3}$
    & 1.02\\
%    1/100 & 30 &$5.78 \times 10^{-4}$ & 1.03 & $1.87 \times 10^{-3}$    & 1.00 & $5.73 \times 10^{-4}$ & 1.02 & $1.87 \times 10^{-3}$
%    & 1.00\\
    \hline
    linear FEM &1/10 &7& $6.74 \times 10^{-3}$ & -- & $2.00\times 10^{-2}$
    & -- & $6.22 \times 10^{-3}$ & -- & $2.00 \times 10^{-2}$
    & --\\
    linear FEM &1/20 & 14&$3.10 \times 10^{-3}$ & 1.12 & $9.65 \times 10^{-3}$
    & 1.05 & $3.00 \times 10^{-3}$ & 1.05 & $9.65 \times 10^{-3}$
    & 1.05\\
    linear FEM &1/40 & 28&$1.48 \times 10^{-3}$ & 1.07 & $4.72 \times 10^{-3}$
    & 1.03 & $1.45 \times 10^{-3}$ & 1.05 & $4.72 \times 10^{-3}$
    & 1.03\\
    linear FEM &1/80 & 56 &$7.28 \times 10^{-4}$ & 1.02 & $2.34 \times 10^{-3}$
    & 1.01 & $7.19 \times 10^{-4}$ & 1.04 & $2.34 \times 10^{-3}$
    & 1.01\\    
    
    \hline
    \end{tabular}
    \caption{LtN test 1 quadratic patch test results for coupling problem with a straight line interface, using linear finite element method in the local solver. Here $u^{ex}=x^2$ represents the analytical solution.}
    \label{tab:quapatch}
\end{table}

\begin{table}[ht!]
    \centering
    \begin{tabular}{cccccccc}
    \hline
        \multirow{2}{*}{Local Solver} &\multirow{2}{*}{$\beta$} &
     \multicolumn{3}{c}{Nonlocal problem} & \multicolumn{3}{c}{Local problem}\\
    \cline{3-8}
    && $h_{nl}$ & $||u_{nl}-u^{ex}||_2$&  $||u_{nl}-u^{ex}||_{\infty}$ &$ h_{l}$ & $||u_{l}-u^{ex}||_2$ &  $||u_{l}-u^{ex}||_{\infty}$ \\
    \hline
    quadratic FEM & 6 & 1/20 & $6.17 \times 10^{-15}$ & 
    $2.26 \times 10^{-14}$ & 1/10 &
    $1.25 \times 10^{-14}$ & $2.31 \times 10^{-14}$ \\
    quadratic FEM & 12 & 1/40 & $1.88 \times 10^{-14}$ & 
    $6.44 \times 10^{-14}$ & 1/20 &
    $5.42 \times 10^{-14}$ & $9.59 \times 10^{-14}$ \\
    \hline
    \end{tabular}
    \caption{LtN test 1 quadratic patch test results for coupling problem with a straight line interface, using quadratic finite element method in the local solver. Here we take $\delta/h_{nl} = 3.0$, and $u^{ex}=x^2$ represents the analytical solution.}
    \label{tab:sopatch}
\end{table}

\subsubsection{LtN Test 2: coupling problem with a circular interface}\label{sec:test2couple}

We now consider the coupling LtN problem on a circular interface, with the domain settings illustrated in the left side of Figure \ref{fig:couplingdomain}. The nonlocal subdomain is set as a unit circle $\Omega_{nl}=\{(x,y)|x^2+y^2\leq 1\}$ and the local subdomain is a $4\times 4$ square region surrounding the unit circle. The local-to-nonlocal interface $\Gamma_i = \{(x,y)|x^2+y^2 = 1\}$. With this test, we aim to investigate the performance of the proposed coupling framework on local-to-nonlocal coupling problems with curved interfaces. Note that in the finite element solver generated by FEniCS, the circular interface is approximated by a polygon, which introduces an $\mcO(h)$ discretization error in the coupling framework. However, when $\Delta t=\mcO(h^2)$, this discretization error is in the same order as the optimal splitting error $\Delta t/h=\mcO(h)$, so we therefore expect no deterioration on the convergence rate.

\begin{figure}
    \centering
\includegraphics[width =0.7 \textwidth]{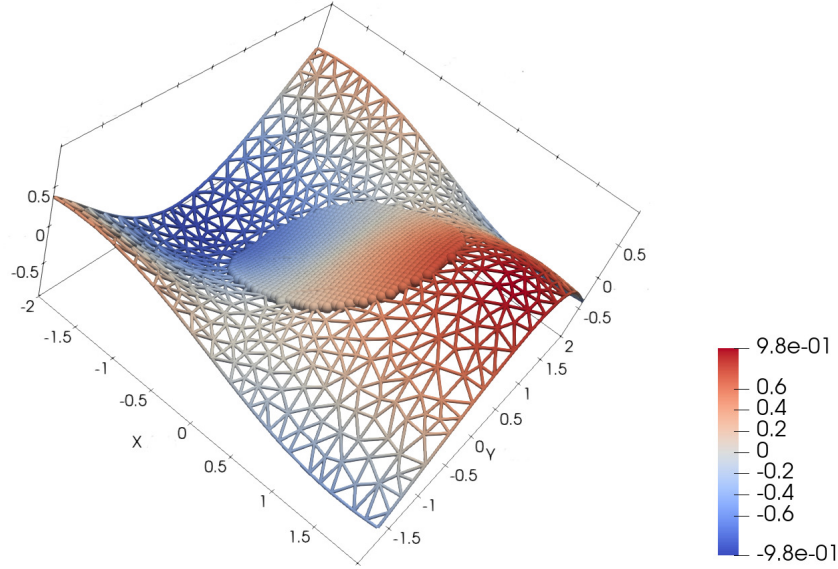}
    \caption{LtN test 2 simulation results for coupling problem with a circular interface when $\alpha_{nl}=\alpha_l=1$. Here the sphere represents the nonlocal solution with the meshfree solver and the triangular mesh represents the local solution obtained via finite element approximations.}
    \label{fig:circlecouple}
\end{figure}

\begin{figure}
    \centering
    \begin{minipage}{0.5\textwidth}
    \includegraphics[width = \textwidth]{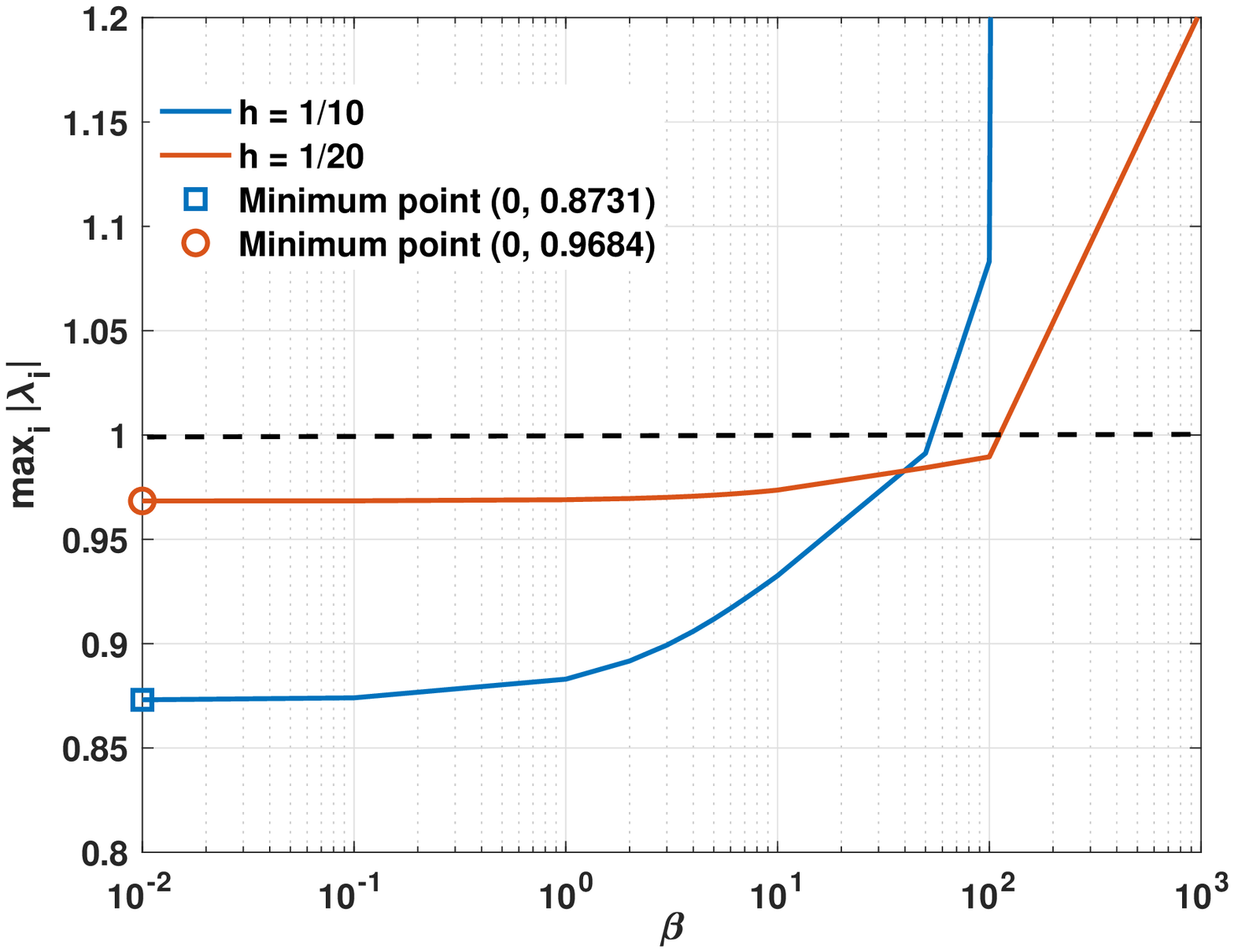}
    \end{minipage}%
    \begin{minipage}{0.5\textwidth}
    \includegraphics[width = \textwidth]{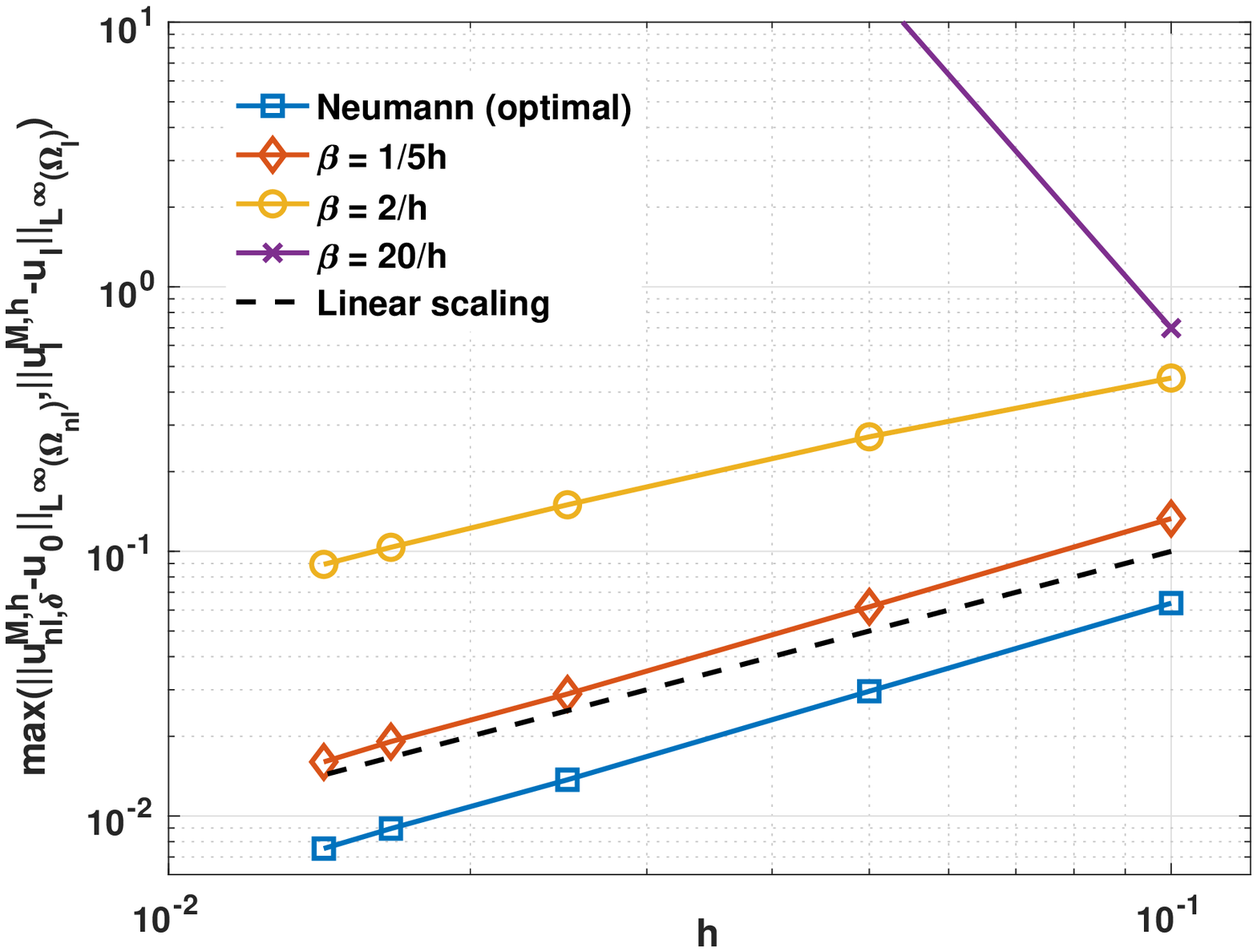}
    \end{minipage}
    \caption{LtN test 2 results for coupling problem with a circular interface when $\alpha_{nl}=\alpha_l=1$. Left: the amplification factor $\underset{i}\max {\lambda_i}$ as a function of Robin coefficient $\beta$ when $h=\{1/10,1/20\}$. Right: convergence of the numertical solution to the local limit with respect to different Robin coefficients, in the $L^{\infty}$ norm.}
    \label{fig:circle_infconv}
\end{figure}

We first study the numerical performance when $\alpha_{nl}=\alpha_l=1$, by employing the same problem setting as in \eqref{eqn:caseAdiri}-\eqref{eqn:caseu0}. Note here since the nonlocal subdomain is fully embedded in the local subdomain, we only need to provide the Dirichlet-type boundary condition on one point for the nonlocal subdomain, to make sure that the nonlocal subproblem is well-defined in the $\beta=0$ case. Specifically, we set $u_{nl,\delta}(x,y,t)=t\sin(x)\cos(y)$ at $(x,y)=(0,-1)$. The simulation results at $T=1$ are plotted in Figure \ref{fig:circlecouple}, where the sphere represents the solution in the nonlocal subdomain and the triangular mesh represents the solution in the local subdomain. To investigate the optimal Robin coefficient, when keeping a fixed ratio $\delta/h = 3.9$ and $\Delta t = 10h^2$ we show the amplification factor $\underset{i}\max |\lambda_i|$ from different Robin coefficients in the left plot of Figure \ref{fig:circle_infconv}. In this case, $\underset{i}\max |\lambda_i|$ achieves the minimum when $\beta=0$, which suggests that the Neumann-type transmission condition is the optimal choice. Moreover, we also notice that comparing with the results in test 1, the curves of $\underset{i}\max |\lambda_i|$ in test 2 show very different trends. In this case, the value of $\underset{i}\max |\lambda_i|$ increases slowly when $\beta\leq\dfrac{1}{h}$, and $\underset{i}\max |\lambda_i|$ becomes larger than $1$ when $\beta\geq \dfrac{5}{h}$. It indicates that the coupling framework performance should not vary much when employing a small $\beta$, and a large $\beta$ is not a preferable choice for this case since the numerical solution may diverge. To numerically verify the amplification factor analysis and to investigate the asymptotic compatibility of the numerical solution, in the right plot of Figure \ref{fig:circle_infconv} we show the convergence of the numerical solution to the analytical local limit in the $L^\infty$ norm at time $T = 1$. Among the 4 different Robin coefficients, the case with Neumann-type transmission condition achieves the optimal $\mcO(h)$ convergence to the local limit, and the case with $\beta=\dfrac{1}{5h}$ also gives a similar convergence rate. When we further increase the Robin coefficient, the convergence rate deteriorates and the coupling framework becomes unstable when $\beta=\dfrac{20}{h}$. These observations are consistent with the amplification factor analysis. The different trends in Figure \ref{fig:line_infconv} and Figure \ref{fig:circle_infconv} also suggest that the optimal Robin coefficient may vary a lot on different domain settings, and therefore a case-by-case analysis of $\beta$ is of critical.

On the coupling problem with circular interface, we now investigate the performance of the non-overlapping coupling framework in handling physical property jumps across the interface. In this test we assume that the two subproblems have dramatically different diffusivities $\alpha_{nl} = 1$ and $\alpha_l = 10$, and consider the following two problem settings:
\begin{itemize}
    \item \textbf{Heterogeneous domain setting A:} 
%\begin{subequations}\label{eqn:cirheteroA}
\begin{align*}
&u^{IC}(x,y)=0, \text{ for } (x,y)\in \Omega_{nl}\cup \Omega_l,\quad u^D_{l}(x,y,t) = t\sin(x)\cos(y), \text{ for } (x,y)\in \Gamma_{d},\\
&f_{nl}(x,y,t) = (1+2t)\sin(x)\cos(y), \text{ for } (x,y) \in \Omega_{nl}, \; f_{l}(x,y,t) = (1+20t)\sin(x)\cos(y), \text{ for } (x,y)\in \Omega_{l},\\
&\lim_{\delta\rightarrow 0} u_{nl,\delta}(x,y,t)=u_{0}(x,y,t)=t\sin(x)\cos(y) \text{ for } (x,y) \in \Omega_{nl}, \quad u_{l}(x,y,t) = t\sin(x)\cos(y), \text{ for } (x,y)\in \Omega_{l}.
\end{align*}
%\end{subequations}
    \item \textbf{Heterogeneous domain setting B:} 
%\begin{subequations}\label{eqn:cirheteroB}
\begin{align*}
&u^{IC}(x,y)=0, \text{ for } (x,y)\in \Omega_{nl}\cup \Omega_l,\quad u^D_{l}(x,y,t) = \dfrac{t((x^2+y^2)^2+1)}{2}, \text{ for } (x,y)\in \Gamma_{d},\\
&f_{nl}(x,y,t) = (x^2+y^2)-4t, \text{ for } (x,y) \in \Omega_{nl}, \; f_{l}(x,y,t) = \dfrac{((x^2+y^2)^2+1)}{2}-80t(x^2+y^2), \text{ for } (x,y)\in \Omega_{l},\\
&\lim_{\delta\rightarrow 0} u_{nl,\delta}(x,y,t)=u_{0}(x,y,t)=t(x^2+y^2) \text{ for } (x,y) \in \Omega_{nl}, \quad u_{l}(x,y,t) = \dfrac{t((x^2+y^2)^2+1)}{2}, \text{ for } (x,y)\in \Omega_{l}.
\end{align*}
%\end{subequations}
\end{itemize}
In both settings we keep a fixed ratio $\delta/h = 3.9$ and take the time step size $\Delta t = 10h^2$. We note that setting A and setting B should have the same optimal Robin coefficient $\beta$. To study this optimal Robin coefficient, in Figure \ref{fig:circle2robin} we plot the amplification factor $\underset{i}\max |\lambda_i|$ as a function of $\beta$ for fixed spatial discretization sizes $h=1/10$ and $h=1/20$, and observe that $\underset{i}\max |\lambda_i|$ achieves its minimum at $\beta=0$, i.e., when the Neumann-type transmission condition is employed. To numerically verify this observation and to study the asymptotic convergence of the analytical solution, in Figure \ref{fig:circle2infconv} the convergence results of numerical solution to the analytical local limit at time $T =1$ are plotted versus decreasing $h$ for both problem setting A (in the left plot) and problem setting B (in the right plot). In the case with problem setting A, the fastest convergence $\mcO(h)=\mcO(\delta)$ is achieved when employing the Neumann-type transmission condition. In the case with problem setting B, the results from $\beta=\dfrac{2}{10h}$ has the smallest difference to the local limit, while tests with $\beta=0$ and $\beta=\dfrac{2}{10h}$ achieve almost the same asymptotic convergence rates to the local limit. In both cases, the numerical solution diverges when employing large $\beta$, which is consistent with the amplification factor analysis. Therefore, in the test with non-zero interface curvature and heterogeneous diffusivities, the amplification factor analysis also provide a good guidance for the optimal Robin coefficient, and the coupling framework employing optimal Robin coefficient is asymptotically compatible.

\begin{figure}[ht!]
    \centering
    \begin{minipage}{0.5\textwidth}
    \includegraphics[width = \textwidth]{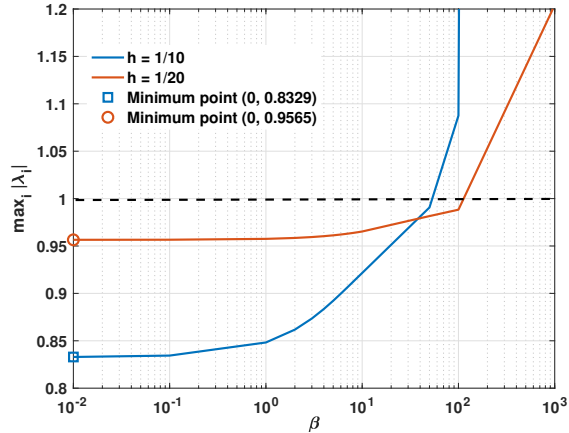}
    \end{minipage}
    \caption{LtN test 2 with the heterogeneous setting ($\alpha_{nl}=1$, $\alpha_l=10$): the amplification factor $\underset{i}\max  |\lambda_i|$ as a function of Robin coefficient $\beta$ when $h = \{1/10, 1/20\}$.}
    \label{fig:circle2robin}
\end{figure}

\begin{figure}[ht!]
    \centering
    \begin{minipage}{0.5\textwidth}
    \includegraphics[width = \textwidth]{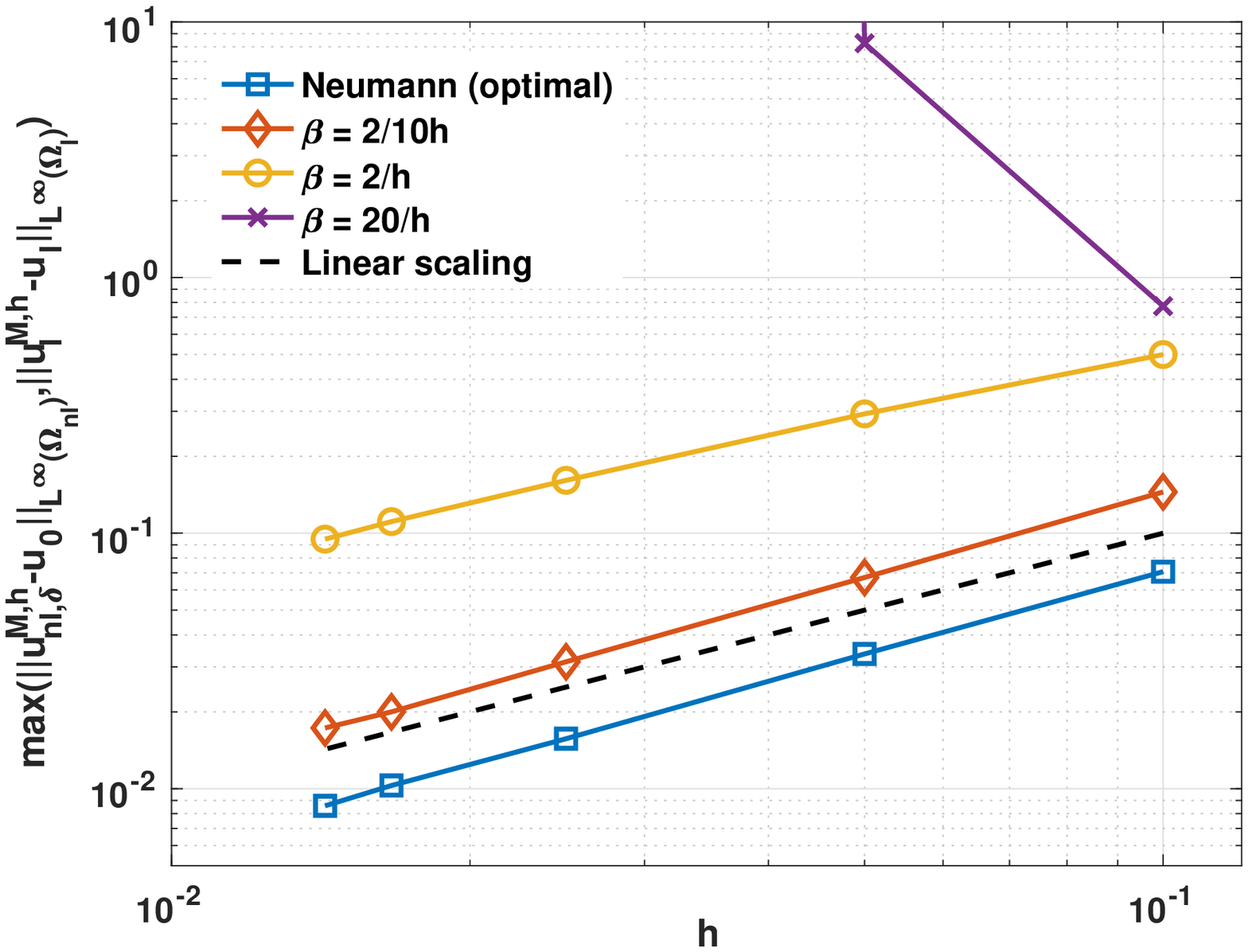}
    \end{minipage}%
    \begin{minipage}{0.5\textwidth}
    \includegraphics[width = \textwidth]{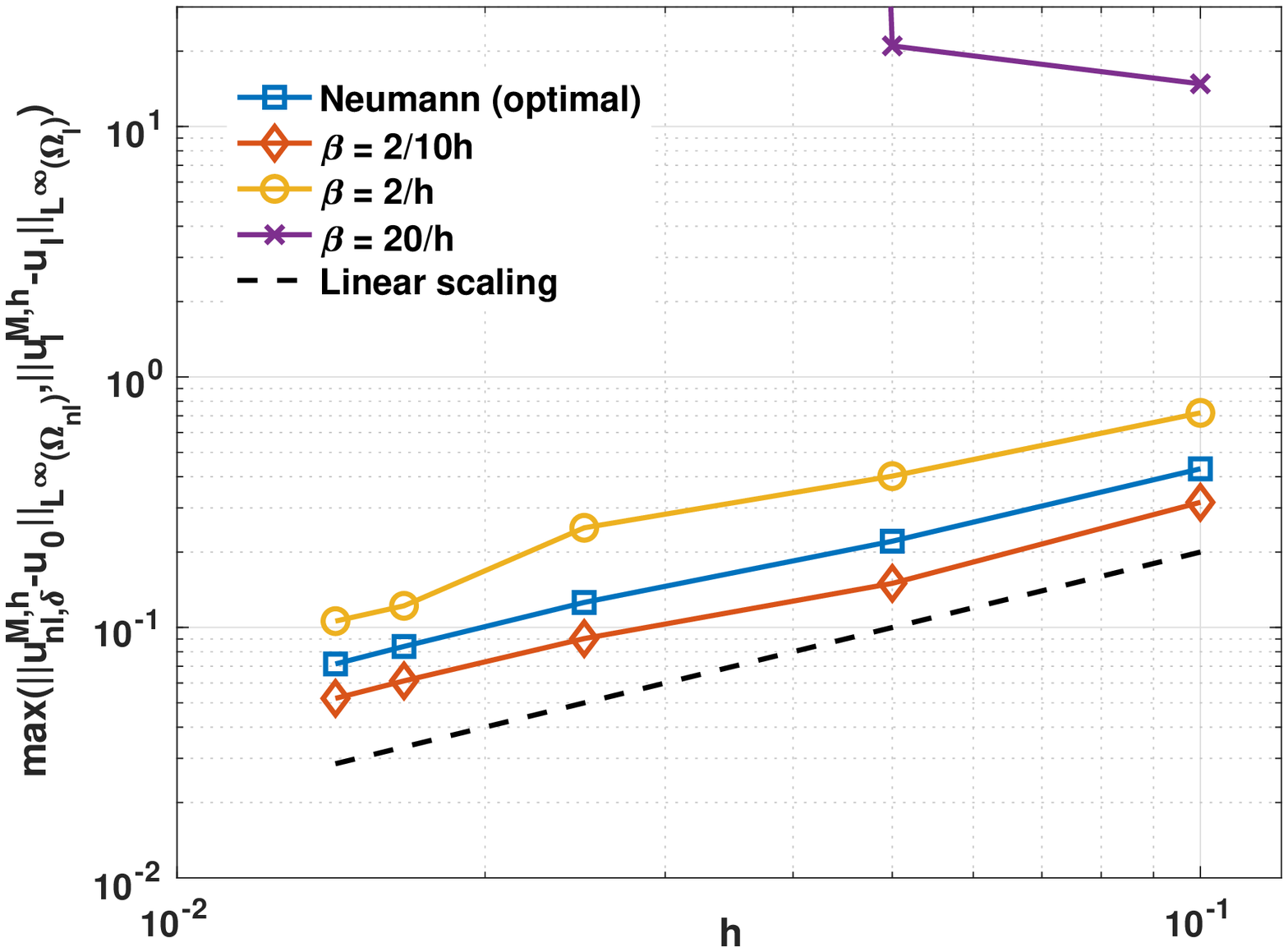}
    \end{minipage}%
    \caption{LtN test 2 results for convergence of the numerical solution to the local limit with different Robin coefficients, on heterogeneous domain setting ($\alpha_{nl}=1$, $\alpha_l=10$). Left: convergence in the $L^\infty$ norm with problem setting A. Right: convergence in the $L^\infty$ norm with problem setting B.}
    \label{fig:circle2infconv}
\end{figure}

\subsubsection{LtN Test 3: coupling problem with a cross-shape interface}\label{sec:test3couple}

\begin{figure}
    \centering
\includegraphics[width =0.7 \textwidth]{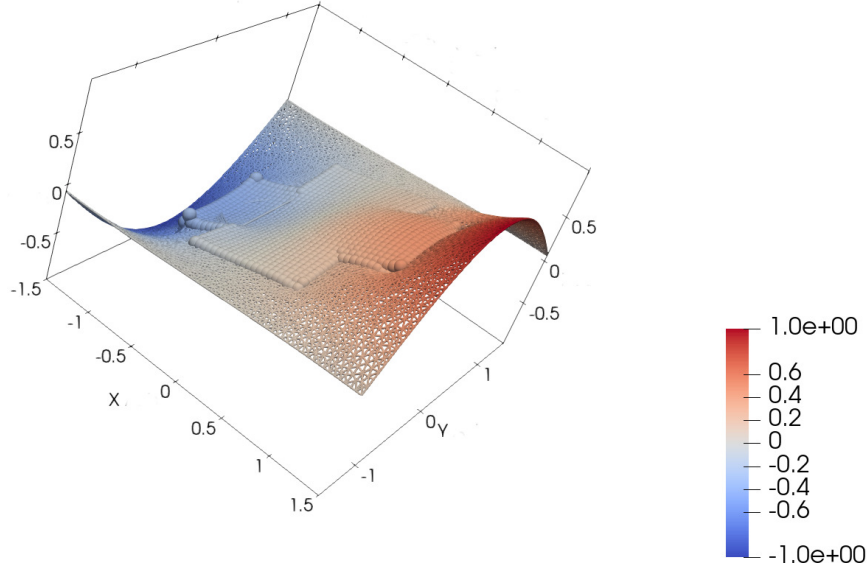}
    \caption{LtN test 3 simulation results for coupling problem with a cross-shape interface when $\alpha_{nl}=\alpha_l=1$. Here the sphere represents the nonlocal solution with the meshfree solver and the triangular mesh represents the local solution obtained via finite element approximations.}
    \label{fig:crosscouple}
\end{figure}

Having demonstrated the asymptotic convergence and the optimal coupling strategy for problems with smooth interfaces, we now apply our approach
to a problem with complicated domain settings, as illustrated in the right plot of Figure \ref{fig:couplingdomain}. The nonlocal subdomain is set as a cross-shape region which is not convex, and the local subdomain is a square region surrounding the nonlocal subdomain. With this test, we aim to investigate the performance of the coupling framework on non-trivial domain settings where the coupling interface is non-smooth and includes corners.

We first study the numerical performance when $\alpha_{nl}=\alpha_l=1$. In this test, the initial temperature $u^{IC}=0$ in the whole domain. In the nonlocal exterior boundary $B\Omega_d$ and the local exterior boundary $\Gamma_d$, we set prescribed Dirichlet boundary conditions as
\begin{equation}\label{eqn:caseAdirit2}
u^{D}_{nl}(x,y,t)=t^2\sin(x)\cos(y),\quad u^{D}_{l}(x,y,t)=t^2\sin(x)\cos(y).
\end{equation}
The external loadings are set as
\begin{equation}\label{eqn:caseAft2}
f_{nl}(x,y,t)=(2t+2t^2)\sin(x)\cos(y), \text{ for } (x,y)\in \Omega_{nl},\quad f_l(x,y,t)=(2t+2t^2)\sin(x)\cos(y), \text{ for }(x,y)\in \Omega_{l}.
\end{equation}
This problem has the following analytical limits:
\begin{equation}\label{eqn:caseult2}
\lim_{\delta\rightarrow 0} u_{nl,\delta}(x,y,t)=u_0(x,y,t)=t\sin(x)\cos(y),\;(x,y)\in \Omega_{nl},\quad u_l(x,y,t) = t^2\sin(x)\cos(y),\; (x,y)\in \Omega_{l}.
\end{equation}
The simulation results at $T=1$ are plotted in Figure \ref{fig:crosscouple}, and the results on amplification factor and convergence to the local limits are demonstrated in Figure \ref{fig:cross_infconv}. In all tests we set $\delta/h=3.5$ and $\Delta t=100h^2$. In the left plot of Figure \ref{fig:cross_infconv} we investigate the optimal Robin coefficient $\beta$ by plotting $\underset{i}\max |\lambda_i|$ as a function of $\beta$ for two different spatial discretization length scales. It can be observed that the minimum value of $\underset{i}\max |\lambda_i|$ occurs at $\beta=\dfrac{1}{5h}$. Moreover, $\underset{i}\max |\lambda_i|>1$ when $\beta=0$ or $\beta>\dfrac{10}{h}$, which suggests the possible deteriorating convergence rate or ever divergence of the numerical solution. In the right plot of Figure \ref{fig:cross_infconv}, we show the $L^{\infty}$ norm of the difference between the numerical results and the analytical local limit at time $T = 1$ for four different values of $\beta$: $0$, $\dfrac{1}{5h}$, $\dfrac{1}{h}$ and $\dfrac{5}{h}$. The numerical results illustrate that when taking $\beta=0$, the coupling framework is unstable. Moreover, when employing the optimal Robin coefficient $\beta=\dfrac{1}{5h}$, the numerical solution has the fastest convergence rate $\mcO(h)=\mcO(\delta)$. Both findings are consistent with the observation from the amplification factor analysis. The above numerical results indicate that the amplification factor analysis helps predicting the optimal Robin coefficient for problems with non-smooth interfaces, and the optimal asymptotic convergence rate $\mcO(h)=\mcO(\delta)$ is achieved.

\begin{figure}
    \centering
    \begin{minipage}{0.5\textwidth}
    \includegraphics[width = \textwidth]{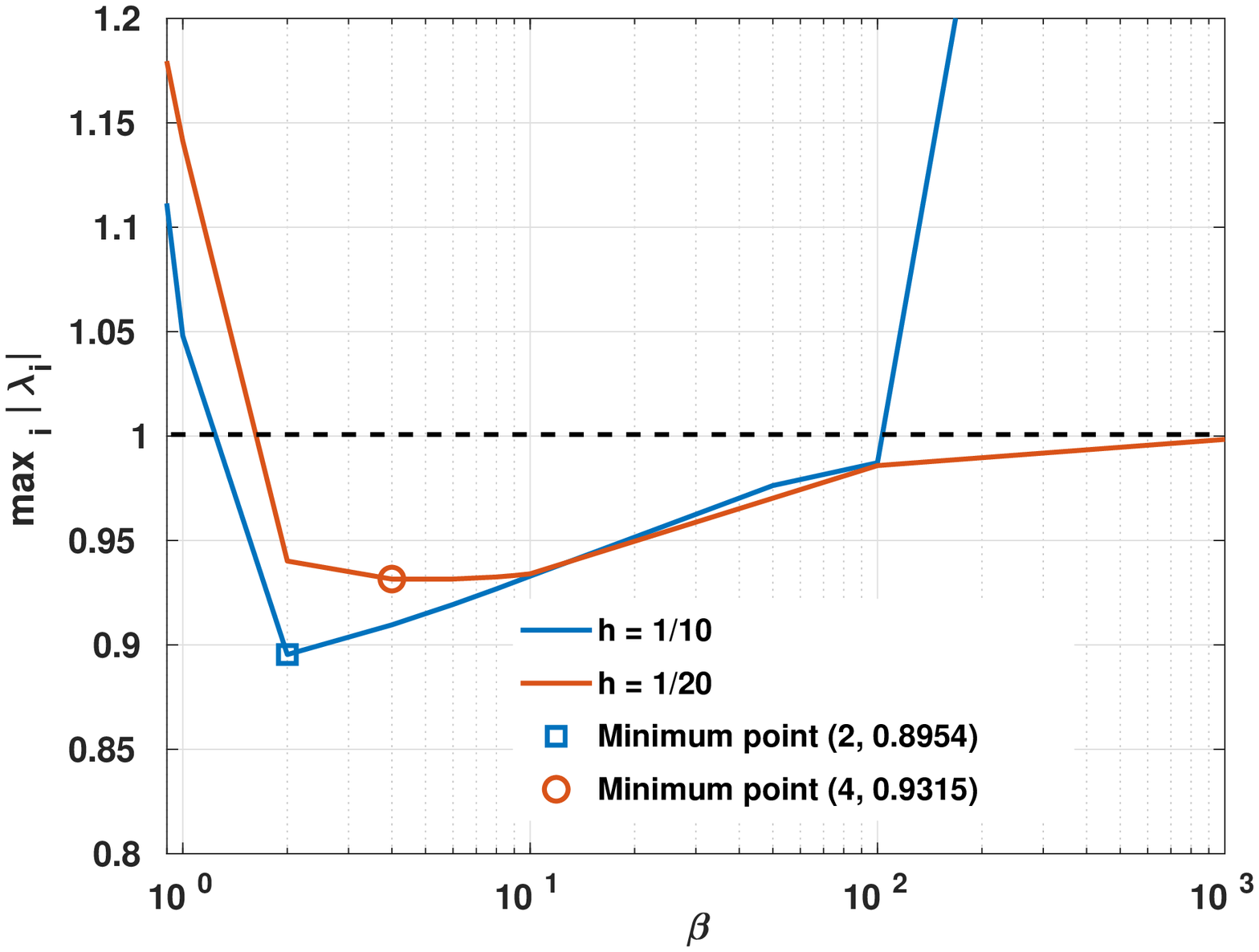}
    \end{minipage}%
    \begin{minipage}{0.5\textwidth}
    \includegraphics[width = \textwidth]{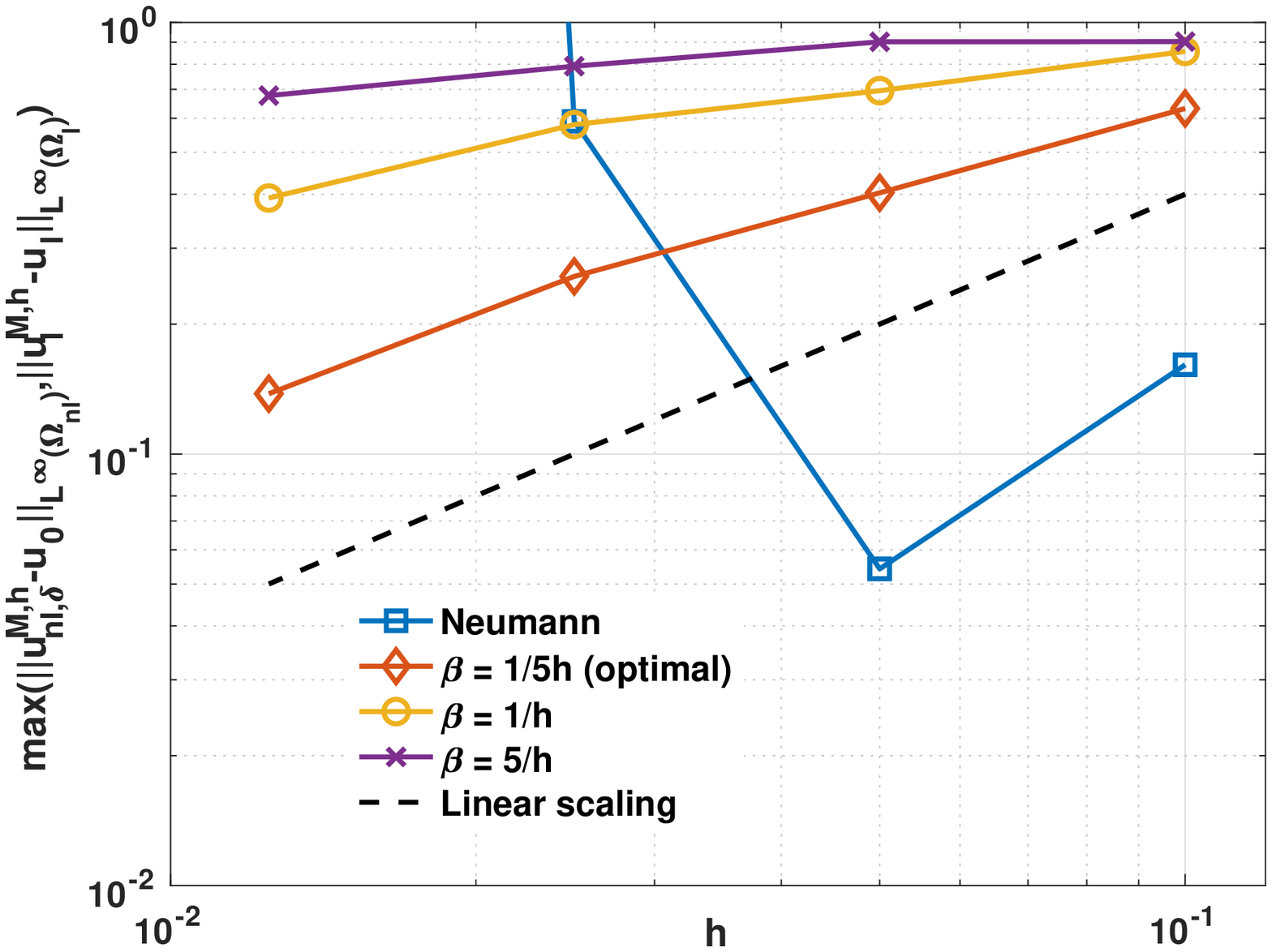}
    \end{minipage}
    \caption{LtN test 3 for the coupling problem with a cross-shape interface when $\alpha_{nl}=\alpha_l=1$. Left: the amplification factor $\underset{i}\max  |\lambda_i|$ as a function of Robin coefficient $\beta$ when $h=\{1/10,1/20\}$. Right: convergence of the numerical solution to the local limit with different Robin coefficients, in the $L^{\infty}$ norm.}
    \label{fig:cross_infconv}
\end{figure}

Lastly, we investigate the performance of the non-overlapping coupling framework in handling heterogeneous domains, by taking $\alpha_{nl}=1$ and $\alpha_l=0.1$ and considering the problem setting as follows: 
%\begin{subequations}\label{eqn:crossheteroA}
\begin{align*}
&u^{IC}(x,y)=0, \text{ for } (x,y)\in \Omega_{nl}\cup \Omega_l,\\
&u^D_{nl,\delta}(x,y,t) = t^2\sin(x)\cos(y), \text{ for } (x,y)\in \Gamma_{d}, \quad u^D_{l}(x,y,t) = t^2\sin(x)\cos(y), \text{ for } (x,y)\in \Gamma_{d},\\
&f_{nl}(x,y,t) = 2(t+t^2)\sin(x)\cos(y), \text{ for } (x,y) \in \Omega_{nl}, \; f_{l}(x,y,t) = 0.2(t+t^2)\sin(x)\cos(y), \text{ for } (x,y)\in \Omega_{l},\\
&\lim_{\delta\rightarrow 0} u_{nl,\delta}(x,y,t)=u_{0}(x,y,t)=t^2\sin(x)\cos(y) \text{ for } (x,y) \in \Omega_{nl}, \quad u_{l}(x,y,t) = t^2\sin(x)\cos(y), \text{ for } (x,y)\in \Omega_{l}.
\end{align*}
%\end{subequations}
We keep a fixed ratio $\delta/h = 3.5$ and take the time step size $\Delta t = 100h^2$, then investigate the optimal Robin coefficient and the aysmptotic convergence performance of the coupling framework. In the left plot of Figure \ref{fig:crosstwoinf}, we plot the amplification factor $\underset{i}\max |\lambda_i|$ as a function of $\beta$ for $h=1/10$ and $h=1/20$. It is observed that the minimum of $\underset{i}\max |\lambda_i|$ occurs at $\beta=\dfrac{1}{5h}$. 
%Moreover, when $h=\dfrac{1}{20}$, the amplification factor $\underset{i}\max |\lambda_i|$ is close to $1$ even when $\beta$ is near the optimal value, which indicates that the convergence rate may be very sensitive to the choice of $\beta$. 
In the right plot of Figure \ref{fig:crosstwoinf}, the convergence of numerical solution to the local limit at $T=1$ are demonstrated for various values of Robin coefficients: $\beta=0$, $\dfrac{1}{5h}$, $\dfrac{1}{h}$ and $\dfrac{10}{h}$. Besides verifying the optimal $\mcO(h)=\mcO(\delta)$ convergence rate when taking the optimal Robin coefficient $\beta=\dfrac{1}{5h}$, the numerical results also demonstrates the importance of picking the optimal $\beta$: when taking other values of $\beta$, much slower numerical convergence or even divergence are observed.

\begin{figure}
    \centering
    \begin{minipage}{0.5\textwidth}
    \includegraphics[width = \textwidth]{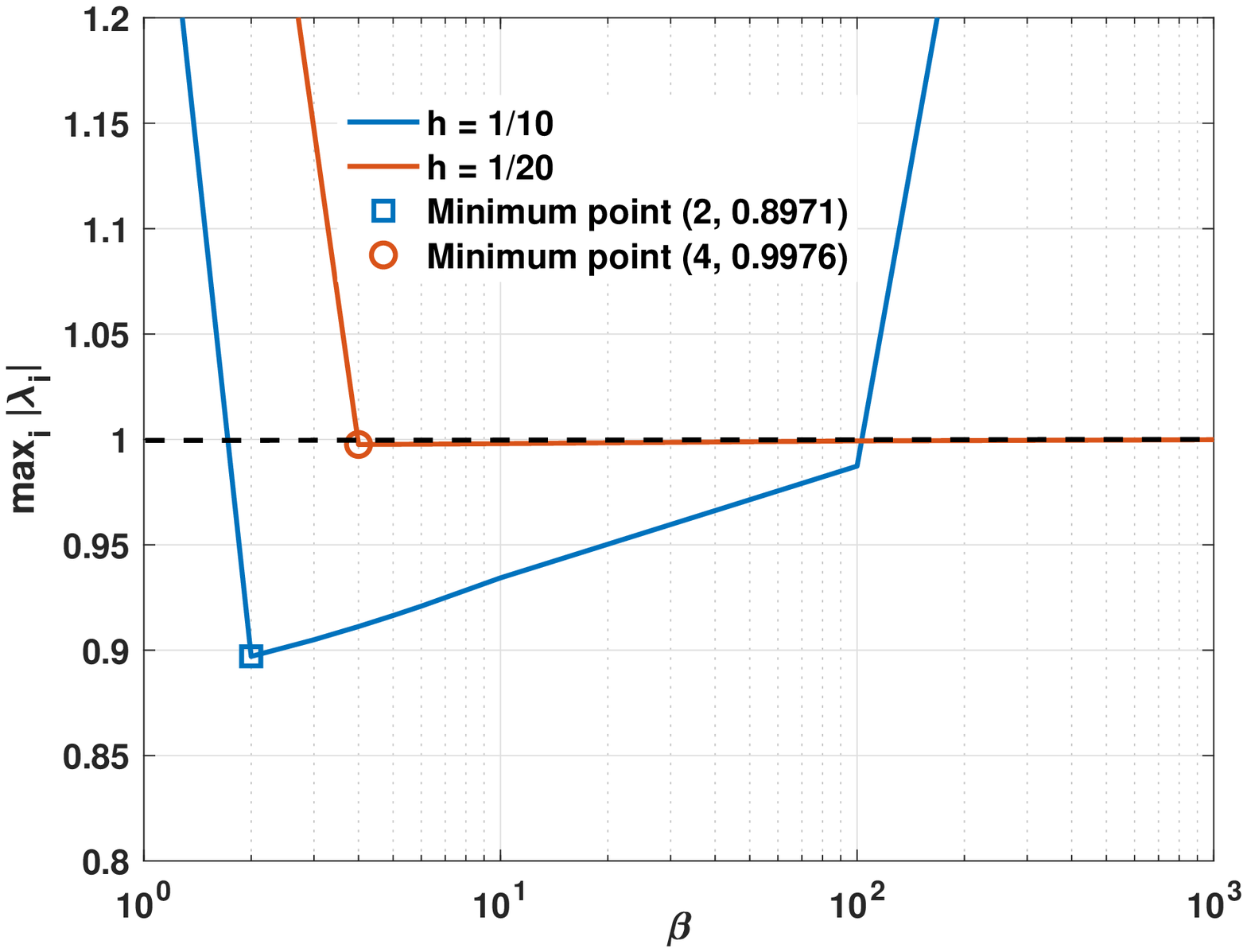}
    \end{minipage}%
    \begin{minipage}{0.5\textwidth}
    \includegraphics[width = \textwidth]{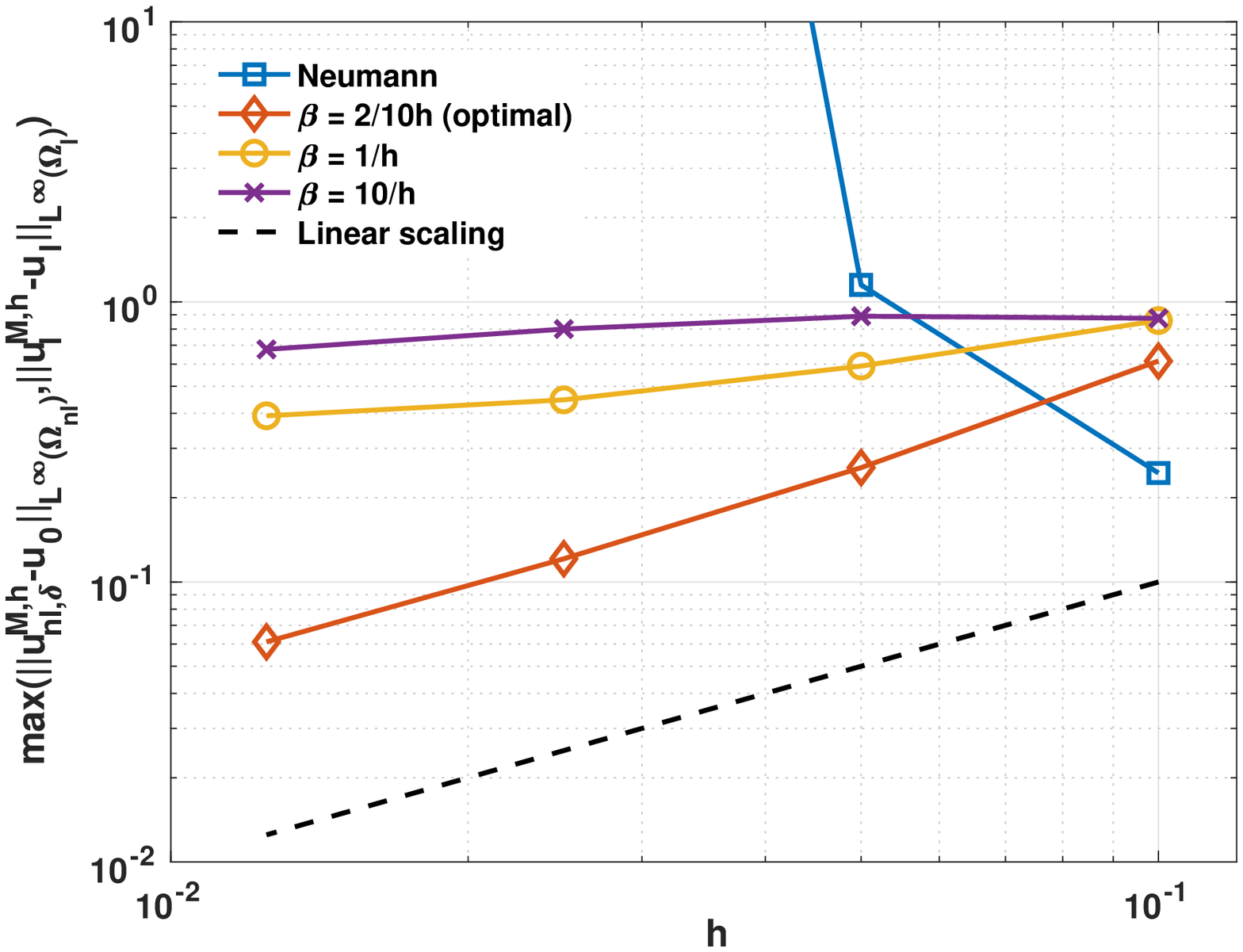}
    \end{minipage}
    \caption{LtN test 3 with the heterogeneous setting ($\alpha_{nl}=1$, $\alpha_l=0.1$). Left: the amplification factor $\underset{i}\max  |\lambda_i|$ as a function of Robin coefficient $\beta$ when $h=\{1/10,1/20\}$. Right: convergence of the numerical solution to the local limit with different Robin coefficients, in the $L^{\infty}$ norm.}
    \label{fig:crosstwoinf}
\end{figure}

\section{Conclusion}\label{sec:conclusion}

Developing a efficient numerical approach for dynamic local-to-nonlocal (LtN) coupling problem with a non-overlapping domain setting is generally challenging due to both modeling and numerical difficulties. From the modeling aspect, since there is no overlapping region between the two subdomains, the prescription of nonlocal transmission conditions, or volume constraints, becomes non-trivial. From the numerical aspect, when employing the partitioned procedure in LtN coupling problems, one not only has to resolve the numerical stability issue as in the classical domain-decomposition problems, but also has to face the challenge of preserving the asymptotic compatibility.

In this work we have developed an explicit coupling strategy to couple the local and nonlocal heat equations without overlapping regions, based on a new nonlocal Robin-type transmission condition. A meshfree discretization method based on the generalized moving least squares (GMLS) approximation is used to solve for the nonlocal heat equation in the nonlocal subdomain, and a first order finite element method is employed for the classical heat equation in the local subdomain. The coupling framework is based on the partitioned procedure such that the local and nonlocal solvers communicate by exchanging interface conditions, which enables a modular software
implementation and the solvers can be treated as black boxes. To resolve the challenge of applying the transmission condition in the nonlocal solver, we have introduced a new nonlocal Neumann-type constraint for the $2D$ nonlocal heat equation which is an analogue to the local flux boundary condition. We have theoretically proved that the proposed nonlocal Neumann-type constraint problem converges with the optimal second-order convergence rate $\mcO(\delta^2)$ to the local limit in the $L^{\infty}(\Omega_{nl})$ norm, and extended this constraint formulation to propose a Robin-type transmission condition. The Neumann and Robin-type formulations are applied on a collar layer inside the domain and therefore require no extrapolation outside the problem domain, which enables the possibility of applying the transmission conditions without overlapping regions. To resolve the numerical challenges in explicit coupling strategy, we provided a numerical approach based on amplification factor analysis to obtain the optimal Robin coefficient. With numerical examples on domains with representative geometries and boundary curvatures, we have verified the robustness and the asymptotic compatibility of both the coupling formulation and the Robin coefficient analysis. Specifically, when employing the optimal Robin coefficient from amplification factor analysis, the optimal $L^\infty$ convergence rate $\mcO(\delta)=\mcO(h)$ to the local limit is observed from the numerical results in all instances. %Finally, we have demonstrated that the regularity assumption may be relaxed in practice and the formulation can be extended to non-convex domain with corners, which greatly improves the applicability of the proposed formulation for more complicated scenarios.

We note that the formulation described in this paper actually provides an approach for applying the Robin-type boundary condition on general compactly supported nonlocal integro-differential equations (IDEs) with radial kernels. Moreover, the coupling framework provides a general coupling strategy for heterogeneous systems, such as multiscale and multiphysics problems. As a natural extension, we are working on the development of local-to-nonlocal coupling framework for mechanical problems with multiphysics, such as the coupling approaches for the incompressible peridynamic model and the surrounding fluid, to study the damage induced by variable amplitude environmental loading.

%Contributions:
%\begin{itemize}
%\item Developed an asymptotically compatible Robin boundary condition for Nonlocal heat equation. The nonlocal solution converges to the local solution with $\mcO(\delta^2)$.
%\item Developed a Robin-Dirichlet coupling strategy for Local-Nonlocal heat equations without overlapping regions.
%\item Developed the numerical approach to obtain optimal Robin coefficient. With the optimal Robin coefficient, a stable explicit coupling strategy is obtained, and it converges to the local limit with $\mcO(\delta)$.
%\end{itemize}

\section*{Acknowledgements}
H. You and Y. Yu were supported by the National Science Foundation under award DMS 1620434. Y. Yu was also partially supported by the Lehigh faculty research grant. D. Kamensky was supported by start-up funds from the University of California, San Diego.

\bibliographystyle{elsarticle-num}
\bibliography{LtNheat}

\end{document}